\documentclass[11pt,letter]{article}

\usepackage[english]{babel}

\usepackage{amssymb}
\usepackage{graphicx}
\usepackage[utf8]{inputenc}
\usepackage{color}
\usepackage{mathrsfs}
\usepackage{amsmath}
\usepackage{amsthm}
\usepackage{bbm}
\usepackage{enumitem}
\usepackage{multirow}
\usepackage{float}
\usepackage{listings}
\numberwithin{figure}{section}

\usepackage[dvipsnames]{xcolor}  

\usepackage{marvosym} 

\usepackage[DIV=13,BCOR=2mm,headinclude=true,footinclude=false]{typearea}
\usepackage{fancyhdr}
    \fancypagestyle{plain}{\fancyhf{}} 

\usepackage{algorithmic}
\usepackage{algorithm}

\makeatletter
\newcommand{\HEADER}[1]{\ALC@it\underline{\textsc{#1}}\begin{ALC@g}}
\newcommand{\ENDHEADER}{\end{ALC@g}}
\makeatother

\usepackage{natbib}
\setcitestyle{authoryear} 

\usepackage{colortbl} 
 \usepackage{color}
\definecolor{Silver}{rgb}{0.752,0.752,0.752}

\usepackage{fourier} 
\usepackage{ upgreek } 

\usepackage{comment} 

\usepackage{titlesec}
\usepackage[colorlinks=true, allcolors=blue]{hyperref}
\titleclass{\subsubsubsection}{straight}[\subsection]

\newcounter{subsubsubsection}[subsubsection]
\renewcommand\thesubsubsubsection{\thesubsubsection.\arabic{subsubsubsection}}

\titleformat{\subsubsubsection}
  {\normalfont\normalsize\bfseries}{\thesubsubsubsection}{1em}{}
\titlespacing*{\subsubsubsection}
{0pt}{3.25ex plus 1ex minus .2ex}{1.5ex plus .2ex}

\makeatletter
\renewcommand\paragraph{\@startsection{paragraph}{5}{\z@}%
  {3.25ex \@plus1ex \@minus.2ex}%
  {-1em}%
  {\normalfont\normalsize\bfseries}}
\renewcommand\subparagraph{\@startsection{subparagraph}{6}{\parindent}%
  {3.25ex \@plus1ex \@minus .2ex}%
  {-1em}%
  {\normalfont\normalsize\bfseries}}
\def\toclevel@subsubsubsection{4}
\def\toclevel@paragraph{5}
\def\toclevel@paragraph{6}
\def\l@subsubsubsection{\@dottedtocline{4}{7em}{4em}}
\def\l@paragraph{\@dottedtocline{5}{10em}{5em}}
\def\l@subparagraph{\@dottedtocline{6}{14em}{6em}}
\makeatother

\usepackage{booktabs}
\usepackage{tabu}
\usepackage[T1]{fontenc}

\usepackage[a4paper,top=2.5cm,bottom=2.5cm,left=3cm,right=3cm,marginparwidth=1.75cm]{geometry}

\usepackage{amsmath}
\usepackage{graphicx}

\usepackage{amssymb}
\usepackage{thmtools}
\usepackage{mathtools}

\newcommand{\comentario}[1]{\textcolor{red}{\textbf{[Comentario: #1]}}}

\newtheorem{definition}{Definition}[section]
\newtheorem{theorem}{Theorem}[section]
\newtheorem{proposition}{Proposition}[section]
\newtheorem{corollary}{Corollary}[section]
\newtheorem{lemma}{Lemma}[section]

\newtheorem{assumption}{Assumption}[section]

\theoremstyle{definition}

\declaretheorem[style=definition,name=Example,qed=$\diamond$,numberwithin=section]{example}

\title{\textbf{ Scenario-based Regularization: A Tractable Framework for Distributionally Robust Stochastic Optimization} }

\author{
\begin{tabular}{cc}
     \textbf{Diego Fonseca} & \textbf{Mauricio Junca}\\
     \texttt{dffonsecav\MVAt eafit.edu.co} & \texttt{mj.junca20\MVAt uniandes.edu.co} \\
     School of Applied Sciences and Engineering & Department of Mathematics\\
     Universidad EAFIT & Universidad de los Andes \\
     Medellín, Colombia & Bogotá, Colombia
\end{tabular}
}

\date{}
\providecommand{\keywords}[1]{\\ \textbf{Kywords: } #1}

\begin{document}
\maketitle

\begin{abstract}
\sloppy We propose a flexible scenario-based regularized Sample Average Approximation (SBR-SAA) framework for stochastic optimization. This work is motivated by challenges in standard Wasserstein Distributionally Robust Optimization (WDRO), where out-of-sample performance, particularly tail risk, is sensitive to the choice of the $p$-norm, and formulations can be computationally intractable. Our method is inspired by the asymptotic expansion of the WDRO objective and introduces a regularizer that penalizes the (sub)gradient norm of the objective at a selected set of scenarios. This framework serves a dual purpose: (i) it provides a computationally tractable alternative to WDRO by using a representative subset of the data, and (ii) it can provide targeted robustness by incorporating user-defined adverse scenarios. We establish the theoretical properties of this framework by proving its equivalence to a decision-dependent WDRO problem, from which we derive finite sample guarantees and asymptotic consistency. We demonstrate the method's efficacy in two applications: (1) a multi-product newsvendor problem, where SBR-SAA serves as a tractable alternative to NP-hard WDRO, and (2) a mean-risk portfolio optimization problem, where it successfully uses historical crisis data to improve out-of-sample performance.
\end{abstract}


\section{Introduction} \label{Sec:Introduccion}

Stochastic optimization is central to decision-making under uncertainty in fields such as finance, energy, engineering design, operations research, and machine learning; see, e.g., \citep{Miller1965,Dentcheva2003,Ziemba2006,ruszczynski2003stochastic,Rigollet2011,LanBook2020}. A standard formulation is
\begin{equation} \label{eqn:SP_intro}
\min_{x\in\mathcal{X}}\mathbb{E}_{\xi\sim\mathbb{P}}[F(x,\xi)],
\end{equation}
where $x$ belongs to a feasible set $\mathcal{X}$, $\xi$ is a random vector with distribution $\mathbb{P}$ supported on $\Xi\subseteq\mathbb{R}^{d}$, and $F$ denotes the (possibly constrained) cost or loss function. In practice, the true distribution $\mathbb{P}$ is unknown, and decisions must be constructed from finite data. Given a sample $\hat{\Xi}_{n}=\{\hat{\xi}_{1},\ldots,\hat{\xi}_{n}\}$ of independent realizations from $\mathbb{P}$, several data-driven paradigms have been proposed. Among the most prominent are Sample Average Approximation (SAA), which replaces the expectation in \eqref{eqn:SP_intro} with the empirical mean, and Distributionally Robust Optimization (DRO), which optimizes with respect to the worst-case distribution in an ambiguity set calibrated from data. In particular, Wasserstein DRO (WDRO), where the ambiguity set is a Wasserstein ball around the empirical measure, has attracted considerable attention due to its theoretical guarantees and empirical performance \citep{Gao2016,MohajerinEsfahani2018,Blanchet2019}.

Optimizing \eqref{eqn:SP_intro} is inherently risk-neutral: the expectation aggregates outcomes linearly, allowing extreme costs or losses to be offset by more favorable realizations. As emphasized in \cite{royset2025risk}, such a criterion can yield decisions with substantial variability or vulnerability to rare but severe events, even if the expected cost is small. In classical portfolio theory, this tension is evident in the Markowitz model, where the introduction of variance as an explicit penalty responds to the inadequacy of optimizing expected return alone \citep{Markowitz1952}. More generally, controlling variability and tail behavior typically requires explicit risk-sensitive components, either through constraints or through composite mean-risk objectives \citep{Fabozzi2010RobustPortfolios,royset2025risk}.

A widely used approach is to consider mean-risk formulations involving the Conditional Value-at-Risk (CVaR). In a generic setting, one may consider
\begin{equation}\label{eqn:MeanRiskEmpirico_intro}
\min_{x\in\mathcal{X}} \left( \mathbb{E}_{\xi\sim\mathbb{P}}[F(x,\xi)]+\varepsilon \mathrm{CVaR}_{\alpha,\xi\sim\mathbb{P}}(F(x,\xi)) \right),
\end{equation}
where $\varepsilon\ge 0$ controls the trade-off between mean performance and tail risk, and $\mathrm{CVaR}_{\alpha}$ focuses on losses beyond a quantile level $\alpha$ \citep{Rockafellar2000}. While such models are conceptually appealing and widely adopted in finance and beyond \citep{Fabozzi2010RobustPortfolios,royset2025risk}, they face important practical difficulties. Reliable estimation of CVaR is demanding: its dependence on extreme quantiles makes it statistically unstable in small or moderate samples, particularly under heavy tails or model misspecification \citep{Hu19122024,royset2025risk}. Embedding CVaR into an optimization problem can be challenging for general (possibly non-smooth or non-convex) functions $F$ \citep{Meng2010}, and in portfolio applications, the resulting estimators can be fragile, amplifying errors in both means and tails \citep{LIM2011163}. Thus, while CVaR-based formulations directly encode tail aversion, their empirical robustness is delicate.

These limitations motivate a more precise view of robustness in data-driven stochastic optimization. In this paper, we emphasize two complementary notions:
\begin{itemize}
    \item \sloppy\emph{Distributional robustness of average performance}: decisions $\hat{x}$ should exhibit stable out-of-sample expected costs $\mathbb{E}_{\xi\sim\mathbb{P}}[F(\hat{x},\xi)]$ across different samples drawn from $\mathbb{P}$.
    \item \sloppy\emph{Robustness against extreme outcomes}: decisions should mitigate the impact of low-probability, high-cost realizations, which can be quantified, for instance, via out-of-sample CVaR or related tail-sensitive measures.
\end{itemize}
A method that only optimizes the empirical mean may fail on both fronts; a method that directly targets tail functionals may suffer from statistical or computational instability. Designing procedures that navigate this trade-off in a principled and tractable manner is a central challenge.

DRO provides a systematic way to address distributional uncertainty by optimizing under the worst-case distribution within an ambiguity set $\mathcal{D}$:
\begin{equation}\label{DROGeneral_intro}
J_{\mathcal{D}}:=\min_{x\in\mathcal{X}}\sup_{\mathbb{Q}\in\mathcal{D}}\mathbb{E}_{\xi\sim\mathbb{Q}}[F(x,\xi)].
\end{equation}
When $\mathcal{D}$ is chosen such that $\mathbb{P}\in\mathcal{D}$ with high confidence, $J_{\mathcal{D}}$ upper bounds the true optimal value $J$ of \eqref{eqn:SP_intro}, and the solutions of \eqref{DROGeneral_intro} enjoy finite-sample performance guarantees. Many constructions for $\mathcal{D}$ have been proposed, including moment-based sets \citep{Scarf1958,Shapiro2002,Popescu2007,Delage2010} and others built from discrete supports or structural information \citep{Lagoa2002,Shapiro2006}. A different line defines $\mathcal{D}$ via statistical divergences or probability metrics around the empirical distribution $\widehat{\mathbb{P}}_{n}$, such as Burg entropy \citep{Wang2015}, Kullback--Leibler divergence \citep{Jiang2015}, Total Variation distance \citep{Sun2015}, or Wasserstein distances \citep{Gao2016,MohajerinEsfahani2018,Blanchet2019}. Wasserstein DRO is particularly attractive because of its geometric structure, its ability to capture distributional shifts in support, and tractable reformulations in many relevant settings; see the survey \citep{KuhnShafieeWiesemannReview2024}.

However, WDRO is not a monolithic object: its behavior depends critically on the order $p$ of the Wasserstein distance, $W_{p}$, the choice of the ground metric, and the ambiguity radius. These modeling choices affect both tractability and the induced out-of-sample risk profile \citep{KuhnShafieeWiesemannReview2024}. While WDRO often improves stability of the expected cost and helps mitigate overfitting, there is no general equivalence between Wasserstein DRO and classical mean-risk formulations of the form \eqref{eqn:MeanRiskEmpirico_intro}. Moreover, for some problem structures, WDRO may effectively reduce to SAA and fail to deliver additional robustness, and for others, tractable reformulations can become computationally demanding or intractable (see Section \ref{Lipproblem}).

To illustrate the sensitivity of WDRO to the metric choice and its implications for tail behavior, we consider the following example.

\begin{example}[Motivating example: the role of the Wasserstein order] \label{MotivationalExample}
We construct a one-dimensional feasible set $\mathcal{X}=[\beta,\beta+1]$ with $\beta>0$ and a loss function $F(x,\xi)$ such that:
\begin{enumerate}
    \item the underlying stochastic program \eqref{eqn:SP_intro} has two optimal solutions located at the endpoints $x=\beta$ and $x=\beta+1$,
    \item for sufficiently large ambiguity radii, the WDRO solutions converge to the optimizer of an associated robust optimization problem, which, by design, lies strictly inside $(\beta,\beta+1)$ \citep[cf.][]{ben2009robust}, and
    \item for intermediate radii, 1-Wasserstein DRO (1-WDRO) and 2-Wasserstein DRO (2-WDRO) tend to select different solutions close to distinct true optima of \eqref{eqn:SP_intro}.
\end{enumerate}
The detailed functional form of $F$, the calibration of parameters, and the full derivations are deferred to Appendix~\ref{Appendix:MotivatingExample}. Here, we highlight the out-of-sample implications that are most relevant for our discussion.

We first compare the out-of-sample distributions of $F(\hat{x}_{1},\xi)$ and $F(\hat{x}_{2},\xi)$, where $\hat{x}_{1}$ and $\hat{x}_{2}$ are decisions obtained from 1-WDRO and 2-WDRO, respectively, using a single training sample and then evaluated on an independent large test sample.

\begin{figure}[t]
\centering
 \begin{tabular}{cc}
  \includegraphics[scale=0.46]{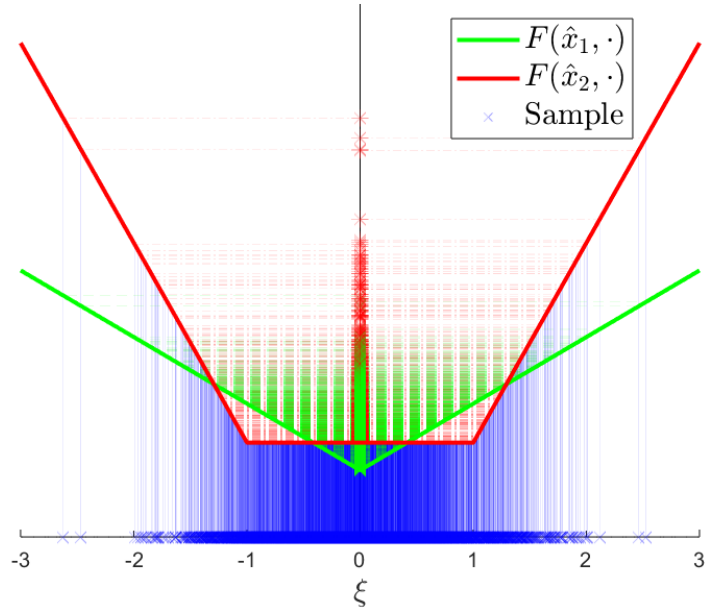} & \includegraphics[scale=0.45]{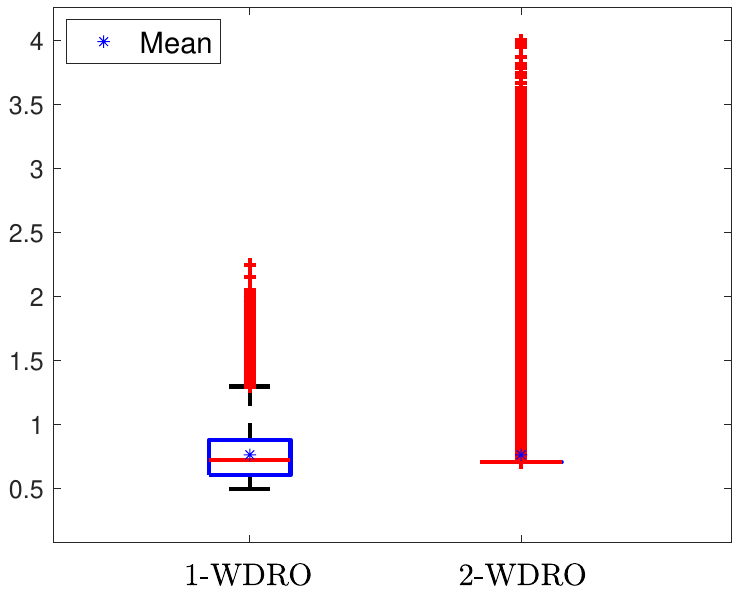} \\
  (a) Scatter plot & (b) Boxplots
 \end{tabular}
\caption{Out-of-sample performance of $F(\hat{x}_{1},\xi)$ (1-WDRO) and $F(\hat{x}_{2},\xi)$ (2-WDRO) for a fixed training sample. Panel (a) shows individual out-of-sample evaluations; panel (b) compares their boxplots. Both methods display similar average performance, but 2-WDRO exhibits noticeably larger extreme values.}\label{fig:OutOfSampleParticularEjeploMotivac_intro}
\end{figure}

Figure~\ref{fig:OutOfSampleParticularEjeploMotivac_intro} shows that, for this configuration, both approaches achieve comparable mean performance, as expected. However, the spread in the upper tail is different: the 2-WDRO solution allows for substantially larger extreme losses than the 1-WDRO solution. Thus, two WDRO formulations with similar expected values can induce markedly different exposure to rare but high-impact outcomes.

To obtain a more systematic comparison, we replicate the experiment across multiple training samples and vary the ambiguity radius $\varepsilon$.

\begin{figure}[t]
\centering
 \begin{tabular}{cc}
  \includegraphics[scale=0.4]{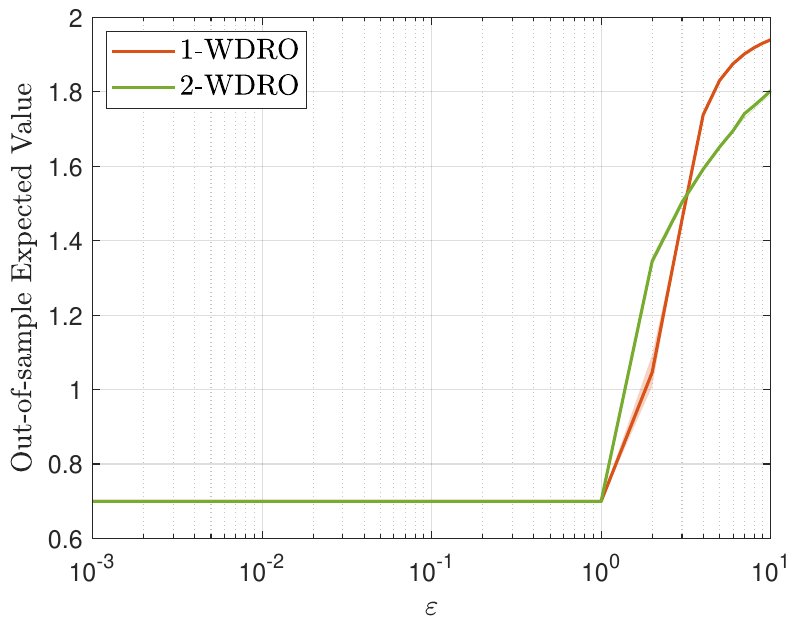} &
  \includegraphics[scale=0.4]{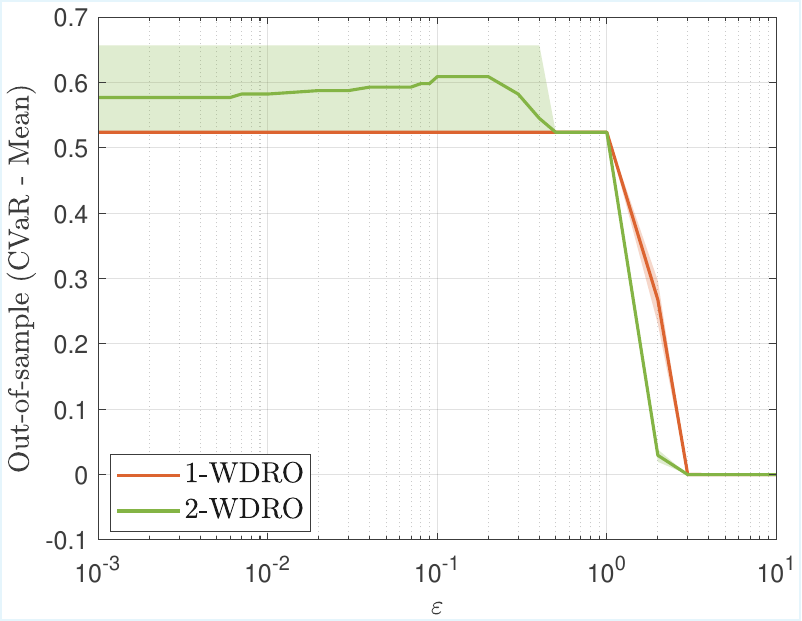} \\
  (a) Out-of-sample expected value & (b) Out-of-sample (CVaR$-\,$Mean)
 \end{tabular}
\caption{Comparison of 1-WDRO and 2-WDRO over multiple training samples and different radii $\varepsilon$. Panel (a) shows tubes between the 10\% and 90\% quantiles and the corresponding means of the out-of-sample expected value. Panel (b) reports the tubes and means for the out-of-sample difference between CVaR$_{\alpha}$ (with $\alpha=5\%$) and the mean. While both methods exhibit similar expected performance, 2-WDRO systematically attains larger CVaR$_\alpha$-to-mean gaps, indicating higher tail risk.}\label{fig:ComparacionEjemploMotivacional_intro}
\end{figure}

Figure~\ref{fig:ComparacionEjemploMotivacional_intro}(a) confirms that 1-WDRO and 2-WDRO deliver comparable out-of-sample expected values over a range of radii. In contrast, Figure~\ref{fig:ComparacionEjemploMotivacional_intro}(b) shows that, for the same configurations, the difference between out-of-sample CVaR$_{\alpha}$ and the mean is consistently larger for 2-WDRO than for 1-WDRO. Hence, in this example, 1-WDRO yields decisions that are less exposed to extreme losses, even though both methods are tuned to similar levels of distributional robustness in expectation.

\end{example}

Example~\ref{MotivationalExample} underscores two points. First, the choice of Wasserstein order is not merely a technical detail; it has direct consequences for the tail-risk behavior of the resulting solutions. Second, it suggests that there may be more transparent ways to encode how the sensitivity of $F(x,\xi)$ to perturbations in $\xi$ should be penalized, rather than relying implicitly on a specific $(p,\mathbf{d})$ configuration.

A key insight in this direction comes from recent asymptotic analyses of Wasserstein DRO \citep{shafieezadeh2019regularization,Blanchet2017,KuhnShafieeWiesemannReview2024,Gao2022VarReg}. Consider the Wasserstein ambiguity set
\[
\mathcal{B}_{\varepsilon}(\widehat{\mathbb{P}}_{n})
:=
\left\{
\mathbb{Q} \;:\; W_p(\mathbb{Q},\widehat{\mathbb{P}}_{n}) \le \varepsilon
\right\},
\]
that is, the ball of radius $\varepsilon$ (in $W_p$) centered at the empirical distribution $\widehat{\mathbb{P}}_{n}$. Under suitable smoothness, growth, and integrability conditions on $F$ (see Lemma~\ref{Lemma:taylorExpansion_intro_revised} in Section~\ref{subsec:wdro_background}), one can show that for each fixed decision $x$ and small $\varepsilon>0$,
\begin{equation}\label{eqn:IntroExpansionInformal}
\sup_{\mathbb{Q} \in \mathcal{B}_{\varepsilon}(\widehat{\mathbb{P}}_{n})} \mathbb{E}_{\xi\sim\mathbb{Q}}[F(x,\xi)]
=
\mathbb{E}_{\xi\sim\widehat{\mathbb{P}}_{n}}[F(x,\xi)]
+
\varepsilon \,\Phi_p\Big(\{\|\nabla_{\xi}F(x,\hat{\xi}_i)\|_*\}_{i=1}^n\Big)
+ o(\varepsilon),
\end{equation}
where $\|\cdot\|_*$ is the dual norm of the ground norm defining $W_p$, and $\Phi_p$ is a norm-like aggregation of the scenario-wise gradient norms. Informally:
\begin{itemize}
    \item for $p=1$, $\Phi_p$ behaves like the maximum (essential supremum) of $\|\nabla_{\xi}F(x,\hat{\xi}_i)\|_*$,
    \item for $p=\infty$, $\Phi_p$ behaves like the empirical mean of these norms,
    \item for $1<p<\infty$, $\Phi_p$ corresponds to an $\ell_q$-type combination of the conjugated norms $\|\nabla_{\xi}F(x,\hat{\xi}_i)\|_*$, where $q=p/(p-1)$ is the H\"older conjugate of $p$.
\end{itemize}
Thus, the leading correction term in \eqref{eqn:IntroExpansionInformal} explicitly depends on how sensitive $F(x,\xi)$ is to perturbations of $\xi$ at the observed scenarios. In other words, for small radii, Wasserstein DRO behaves like an SAA problem regularized by a functional of the gradient norms, whose precise structure is dictated by the choice of $p$, and penalizes the sensitivity of $F(x,\xi)$ with respect to perturbations in $\xi$. This perspective offers a unifying explanation for phenomena such as those observed in Example~\ref{MotivationalExample}: different choices of $p$ correspond to different ways of aggregating gradient norms across scenarios, thereby inducing different robustness profiles with respect to extreme outcomes.

At the same time, this implicit regularization viewpoint also exposes important limitations of WDRO as a practical tool. Implementing WDRO requires solving, or approximating, potentially complex min-max formulations whose tractability depends delicately on $F$, the chosen metric, and the dimension of $\xi$ \citep{KuhnShafieeWiesemannReview2024}. The appropriate Wasserstein order, ground metric, and ambiguity radius are problem-dependent and rarely obvious \emph{a priori}. In some models, WDRO may degenerate to SAA and provide little additional robustness; in others, the computational burden of exact reformulations (or their relaxations) may be prohibitive. These considerations suggest seeking alternatives  that attempt to:
\begin{itemize}
    \item retain the desirable robustness interpretations inspired by WDRO,
    \item make the penalized sensitivity aspects explicit and controllable, and
    \item remain computationally tractable for complex models.
\end{itemize}

We propose a flexible scenario-based regularized SAA framework. Our approach is directly inspired by the gradient-norm corrections appearing in the asymptotic expansion of WDRO, but it decouples the regularization design from the specific choice of Wasserstein order and metric. Concretely, we consider a Regularized SAA, which we term Scenario-Based Regularized SAA (SBR-SAA), where the regularizer
\[
\widehat{R}_m(x)
=
\left(\sum_{j=1}^m r_j \left\|\nabla_\xi F(x, \zeta_j)\right\|_*^{2}\right)^{\frac{1}{2}}
\]
is constructed from a selected set of scenarios $\{\zeta_j\}_{j=1}^m$ and nonnegative weights $\{r_j\}_{j=1}^m$. The key modeling flexibility is that these scenarios can be:
\begin{enumerate}
    \item a representative subset of the observed data, yielding a computationally efficient surrogate for WDRO based on a reduced scenario set, or
    \item user-specified \emph{adverse scenarios}, enabling targeted robustness against specific extreme events (e.g., historical crises or stress situations) without directly embedding CVaR or solving a full min--max DRO problem.
\end{enumerate}
This design makes explicit how local sensitivity of $F(x,\xi)$ is penalized and allows practitioners to encode structural and domain knowledge directly at the level of regularization. Evidently, we could choose another $\ell_q$-type formulation for $\widehat{R}_m(x)$ (see Section \ref{subsec:Portfolio}), but we use the above to emphasize the computational aspect of our approach. 

A central theoretical contribution of this paper is to show that the proposed SBR-SAA framework admits an exact interpretation as a WDRO problem with a modified objective function and a decision-dependent Wasserstein ball centered at a suitably perturbed empirical distribution. In particular, we prove an equivalence between SBR-SAA and a decision-dependent WDRO formulation, which leads to finite-sample performance guarantees and asymptotic consistency analogous to those established for classical WDRO \citep{MohajerinEsfahani2018}. This connection rigorously justifies SBR-SAA as a principled distributionally robust procedure, while highlighting its additional modeling flexibility.

We also demonstrate, through numerical experiments, that SBR-SAA can achieve a favorable balance between stability of the out-of-sample expected cost and protection against extreme events. In a multi-product newsvendor problem, SBR-SAA provides a tractable alternative to NP-hard WDRO formulations while preserving robustness guarantees. In a mean-risk portfolio optimization setting, it leverages historically adverse periods as adverse scenarios, improving out-of-sample tail performance without relying solely on direct CVaR penalization.

\subsection*{Contributions}

In summary, this work makes the following contributions:
\begin{enumerate}
    \item \textbf{Scenario-based regularization framework.} We introduce SBR-SAA, a gradient-norm regularized SAA framework that augments the empirical objective with a scenario-based sensitivity penalty. The regularizer is designed to be flexible, allowing the use of representative or adverse scenarios to shape robustness properties explicitly.
    \item \textbf{Connection to Wasserstein DRO.} We show that SBR-SAA is equivalent to a suitably constructed WDRO problem with a modified objective, radius, and (potentially decision-dependent) Wasserstein ball. This result formalizes the intuition suggested by asymptotic expansions of WDRO \citep{Blanchet2017,shafieezadeh2019regularization,KuhnShafieeWiesemannReview2024,Gao2022VarReg}, and clarifies how gradient-based regularization emerges from distributional robustness considerations.
    \item \textbf{Theoretical guarantees.} Building on the WDRO equivalence, we establish finite-sample guarantees and asymptotic consistency for SBR-SAA, in the spirit of the performance bounds and convergence results known for Wasserstein DRO \citep{MohajerinEsfahani2018}. These results ensure that, as the data size grows, SBR-SAA recovers optimal solutions of the original problem \eqref{eqn:SP_intro}.
    \item \textbf{Practical performance in applications.} We illustrate the effectiveness and interpretability of SBR-SAA in two applications: (i) a multi-product newsvendor problem, where SBR-SAA acts as a tractable surrogate for a challenging WDRO formulation, and (ii) a mean-risk portfolio optimization problem, where incorporating crisis scenarios into the regularizer leads to improved out-of-sample tail performance.
\end{enumerate}

The remainder of the paper is organized as follows. Section~\ref{subsec:wdro_background} provides definitions and reviews relevant results on Wasserstein DRO, including the detailed asymptotic expansion lemma underlying our regularization perspective. Section~\ref{sec:GradientRegularization} develops the SBR-SAA framework, establishes its equivalence to a decision-dependent WDRO, and derives its theoretical guarantees. Section~\ref{sec:NumericalExperiments} presents numerical experiments for the multi-product newsvendor and portfolio problems. Additional technical details, including the full construction and analysis of the motivating example, are provided in the Appendix.

\vspace{0.3cm}

\noindent\textbf{Notation.}
Throughout this paper, we denote by $\xi$ a random vector defined on a probability space $(\Omega,\mathcal{F},\mathbb{P})$ taking values in a support set $\Xi\subseteq\mathbb{R}^{d}$.
The decision variable is denoted by $x\in\mathcal{X}$, where $\mathcal{X}\subseteq\mathbb{R}^k$ is non-empty.
We use $\|\cdot\|$ for an arbitrary norm in a finite-dimensional Euclidean space, and $\|\cdot\|_{q}$, $q\in[1,\infty)$, for the usual $\ell_q$-norm.
For any positive integer $N$, we write $[N]:=\{1,2,\ldots,N\}$.
We use the standard Bachmann--Landau notation $O(\cdot)$ and $o(\cdot)$ for asymptotic bounds.

Given a measurable mapping $f:\Xi\to\mathbb{R}^{k}$ and a probability measure $\mathbb{Q}$ on $(\Xi,\mathcal{B}(\Xi))$, we denote by $f_{\#}\mathbb{Q}$ the \emph{push-forward} of $\mathbb{Q}$ through $f$, defined by
$f_{\#}\mathbb{Q}(A):=\mathbb{Q}(f^{-1}(A))$ for all Borel sets $A\subseteq\mathbb{R}^{k}$;
equivalently, $f_{\#}\mathbb{Q}$ is the distribution of $f(\xi)$ when $\xi\sim\mathbb{Q}$.
For a proper function $g:\mathbb{R}^n\to\mathbb{R}\cup\{+\infty\}$, we denote its convex conjugate by
$g^{*}(y):=\sup_{z\in\mathbb{R}^n}\{\langle y,z\rangle-g(z)\}$ when used.
Whenever we refer to the $q$-Lipschitz modulus of a function $h:\mathbb{R}^n\to\mathbb{R}$, we write
$\|h\|_{\mathrm{Lip},q}:=\sup_{x\neq y}|h(x)-h(y)|/\|x-y\|_{q}$, provided this quantity is finite.


\section{Background on Wasserstein Distributionally Robust Optimization}
\label{subsec:wdro_background}

This section reviews the foundational concepts and background results that underpin our work. Our proposed method draws inspiration from and builds upon the framework of Wasserstein Distributionally Robust Optimization. As formulated in \eqref{DROGeneral_intro}, WDRO considers an ambiguity set $\mathcal{D}$ constructed as a ball of probability measures around a reference distribution (typically the empirical distribution $\widehat{\mathbb{P}}_n$), where this ball is defined in terms of the Wasserstein distance.

We briefly recall the definition of the Wasserstein distance; see \cite{Vasershtein1969,Villani2008} for comprehensive treatments.

\begin{definition}[Wasserstein distance]\label{Def:MetricaWasserstein_intro}
\sloppy The $p$-\textit{Wasserstein distance} $W_{p}(\mathbb{P},\mathbb{Q})$ between $\mathbb{P},\mathbb{Q}\in\mathcal{P}_{p}(\mathcal{X})$ is defined by
\begin{equation*}
 W_{p}(\mathbb{P},\mathbb{Q}):=\begin{cases}\left(\inf\limits_{\Pi\in\mathcal{P}(\mathcal{X}\times\mathcal{X})}\left\{{\displaystyle\int_{\mathcal{X}\times\mathcal{X}}}\mathbf{d}^{p}(\xi,\zeta)\Pi(d\xi,d\zeta)\: :\: \Pi \mbox{ has marginals }\mathbb{P},\mathbb{Q} \right\}\right)^{1/p},  &  p\in[1,\infty), \\
\inf\limits_{\Pi\in\mathcal{P}(\mathcal{X}\times\mathcal{X})}\left\{\Pi-\underset{(\xi,\zeta)\in\mathcal{X}\times\mathcal{X}}{\mathrm{ess\:sup}}\mathbf{d}(\xi,\zeta)\: :\: \Pi \mbox{ has marginals }\mathbb{P},\mathbb{Q} \right\}, &  p=\infty,
 \end{cases}
\end{equation*}
where $\mathcal{P}_{p}(\mathcal{X}):=\left\{\mathbb{Q}\in\mathcal{P}(\mathcal{X})\: :\: \int_{\mathcal{X}}\mathbf{d}^{p}(\xi,\zeta_{0})\mathbb{Q}(d\xi) < \infty\ \mbox{for some }\zeta_{0}\in\mathcal{X}\right\}$ for each $p\in[1,\infty)$, $\mathcal{P}_{\infty}(\mathcal{X})=\mathcal{P}(\mathcal{X})$, and $\mathbf{d}$ is a metric on $\mathcal{X}$.
\end{definition}

Given a nominal distribution $\widehat{\mathbb{P}}_{n}$ and parameters $p\in[1,\infty]$ and $\varepsilon>0$, the associated Wasserstein ambiguity set is the ball
\[
\mathcal{B}_{\varepsilon}(\widehat{\mathbb{P}}_{n})
:=
\big\{
\mathbb{Q}\in\mathcal{P}(\Xi)
:
W_p(\mathbb{Q},\widehat{\mathbb{P}}_{n}) \le \varepsilon
\big\}.
\]
Within this framework, the WDRO counterpart of \eqref{eqn:SP_intro} is
\begin{equation}\label{eqn:p-WDRO}
\min_{x\in\mathcal{X}}
\sup_{\mathbb{Q}\in\mathcal{B}_{\varepsilon}(\widehat{\mathbb{P}}_{n})}
\mathbb{E}_{\xi\sim\mathbb{Q}}[F(x,\xi)],
\end{equation}
where the inner supremum accounts for worst-case perturbations of the empirical distribution within the Wasserstein ball. We call (\ref{eqn:p-WDRO}) the $p$-WDRO problem.

A key reason for the popularity of Wasserstein-based ambiguity sets is that, in many settings, the resulting $p$-WDRO problem admits a reformulation via strong duality. The following theorem, adapted from \cite{Gao2016,Blanchet2019} (see also \cite{MohajerinEsfahani2018,Mehrotra2017}), is a cornerstone result.

\begin{theorem}[\cite{Gao2016,Blanchet2019}]\label{Thm:ReformulacionDROWInterno}
Assume that $F(x,\cdot)$ is upper semicontinuous with respect to $\xi$ for every $x\in\mathcal{X}$. Then the $p$-WDRO problem (\ref{eqn:p-WDRO})
is equivalent to
\begin{equation}\label{Eqn:ReformulacionDROW}
\left\{
\begin{array}{lll}
\displaystyle \inf_{x\in\mathcal{X},\,\lambda\geq 0,\,s\in\mathbb{R}^n}
&
\displaystyle
\lambda \varepsilon^{p}
+\frac{1}{n}\sum_{i=1}^{n}s_{i}
&
\\[1ex]
\text{subject to}
&
\displaystyle
\sup_{\xi\in\Xi}\Big(F(x,\xi)-\lambda \mathbf{d}^{p}(\xi,\widehat{\xi}_{i}) \Big)
\leq s_{i},
&
\forall i \in [n],
\end{array}
\right.
\end{equation}
where $\mathbf{d}$ is the ground metric used to define $W_p$.
\end{theorem}

Theorem~\ref{Thm:ReformulacionDROWInterno} transforms the challenging min-max problem over a space of probability measures into an optimization problem over $(x,\lambda,s)$, but the resulting formulation can still be demanding. The constraints involve, for each $i\in[n]$, a supremum over $\xi\in\Xi$ that may itself be a nontrivial optimization problem, possibly leading to semi-infinite programs or requiring problem-specific analysis. 

Beyond tractability, WDRO admits an interpretation as an implicitly regularized version of SAA. Recent results show that, under suitable regularity conditions, the worst-case expectation over a small Wasserstein ball admits a first-order expansion in the radius. This expansion, which we use as a key analytical tool, precisely outlines how Wasserstein DRO penalizes the sensitivity of $F(x,\xi)$ with respect to perturbations in $\xi$. The following lemma, based on Theorem~8.7 in \cite{KuhnShafieeWiesemannReview2024} for $p\in[1,\infty)$ and Lemma~1 in \cite{Gao2022VarReg} for $p=\infty$, formalizes this connection.

\begin{lemma}[Taylor expansion of worst-case expectation]\label{Lemma:taylorExpansion_intro_revised}
In \eqref{DROGeneral_intro}, suppose that $p \in [1, \infty]$, where $W_p$ is induced by a norm $\|\cdot\|$ on $\mathbb{R}^d$, and that $\Xi$ is convex. Assume also that:
\begin{enumerate}
    \item[(i)] \textbf{Growth condition (for $p \in [1,\infty)$).} There exist $G, \delta_0 > 0$ such that
    \[
    F(x,\xi) - F(x,\xi') \le G\|\xi - \xi'\|^p
    \]
    for all $\xi, \xi' \in \Xi$ with $\|\xi - \xi'\| > \delta_0$ and all $x\in\mathcal{X}$.
    \item[(ii)] \textbf{Smoothness condition.} There exists $L > 0$ such that
    \[
    \|\nabla_{\xi} F(x,\xi) - \nabla_{\xi} F(x,\xi')\|_* \le L\|\xi - \xi'\|
    \]
    for all $\xi, \xi' \in \Xi$ and $x\in\mathcal{X}$, where $\|\cdot\|_*$ is the norm dual to $\|\cdot\|$. (For $p=\infty$, this is the main requirement; see \cite{Gao2022VarReg}.)
    \item[(iii)] \textbf{Integrability condition (for $p \in [1,\infty)$).} Let $q = p/(p - 1)$ be the H\"older conjugate of $p$. Then
    \[
    \mathbb{E}_{\xi\sim\widehat{\mathbb{P}}_n}\big[\|\nabla_{\xi} F(x,\xi)\|_*^q\big]
    \quad\text{and}\quad
    \mathbb{E}_{\xi\sim\widehat{\mathbb{P}}_n}\big[\|\nabla_{\xi} F(x,\xi)\|_*^{2q-2}\big]
    \]
    are finite, and the ratio
    \[
    \frac{
    \mathbb{E}_{\xi\sim\widehat{\mathbb{P}}_n}[\|\nabla_{\xi} F(x,\xi)\|_*^{2q-2}]
    }{
    \big(\mathbb{E}_{\xi\sim\widehat{\mathbb{P}}_n}[\|\nabla_{\xi} F(x,\xi)\|_*^q]\big)^{2/p}
    }
    \]
    is bounded.
\end{enumerate}
Then, for each fixed $x\in\mathcal{X}$ and small $\varepsilon>0$,
\begin{equation} \label{eqn:RegBound_intro_revised}
\sup_{\mathbb{P} \in  \mathcal{B}_{\varepsilon}(\widehat{\mathbb{P}}_n)} \mathbb{E}_{\xi\sim\mathbb{P}}[F(x,\xi)]
=
\mathbb{E}_{\xi\sim\widehat{\mathbb{P}}_n}[F(x,\xi)]
+
\varepsilon \left(\mathbb{E}_{\xi\sim\widehat{\mathbb{P}}_n}[\|\nabla_{\xi} F(x,\xi)\|_*^q]\right)^{\!1/q}
+ o(\varepsilon).
\end{equation}
For $p=\infty$, the term $\left(\mathbb{E}_{\xi\sim\widehat{\mathbb{P}}_n}[\|\nabla_{\xi} F(x,\xi)\|_*^q]\right)^{1/q}$ in \eqref{eqn:RegBound_intro_revised} is replaced by $\mathbb{E}_{\xi\sim\widehat{\mathbb{P}}_n}[\|\nabla_{\xi} F(x,\xi)\|_*]$, while for $p=1$ it becomes $\mathrm{ess\,sup}_{\xi \sim \widehat{\mathbb{P}}_n} \|\nabla_{\xi} F(x,\xi)\|_*$ under the appropriate interpretation for $q=\infty$.
\end{lemma}

Lemma~\ref{Lemma:taylorExpansion_intro_revised} shows that, to first order in $\varepsilon$, the WDRO objective over a Wasserstein ball around $\widehat{\mathbb{P}}_n$ coincides with the empirical expectation plus a regularization term determined by the distribution of gradient norms $\|\nabla_{\xi}F(x,\widehat{\xi}_i)\|_*$. Different choices of $p$ correspond to different ways of aggregating these norms (via $q$), which in turn yield different robustness profiles. This perspective underlies our scenario-based regularization framework: instead of relying solely on the implicit regularization induced by a particular Wasserstein metric, we design an explicit, scenario-driven gradient-norm penalty that inherits the distributional robustness interpretation of WDRO while offering additional modeling flexibility.

\section{A Flexible Scenario-based Regularization Framework}
\label{sec:GradientRegularization}

Building upon the motivation to achieve robust solutions against both distributional shifts and extreme outcomes, this section details our proposed approach. As foreshadowed in Section \ref{Sec:Introduccion}, we address the stochastic optimization problem \eqref{eqn:SP_intro} by introducing a regularized Sample Average Approximation (SAA) scheme. This scheme is inspired by the asymptotic structure of WDRO but offers greater flexibility and computational control.

\subsection{Formulation of the Scenario-based Regularized SAA}
\label{subsec:FormulationGNR}

Our core proposal, which we call Scenario-based Regularized SAA (SBR-SAA), is to solve the following regularized optimization problem:
\begin{equation}\label{eq:SBR-SAA_sec}
\min_{x\in\mathcal{X}} \; \frac{1}{n}\sum_{i=1}^{n}F(x,\widehat{\xi}_i)+\varepsilon\, \widehat{R}_{m}(x),
\end{equation}
where the first term is the standard SAA objective based on the empirical data $\hat{\Xi}_{n}= \{\widehat{\xi}_1, \ldots, \widehat{\xi}_n\}$, $\varepsilon \ge 0$ is a regularization parameter that controls the trade-off with the penalty term, and $\widehat{R}_{m}(x)$ is the Scenario-based Regularization term defined as:
\begin{equation}\label{eqn:RegTerm_sec}
\widehat{R}_{m}(x)=\left(\sum_{j=1}^{m}r_j\,\Big\|\nabla_\xi F(x,\zeta_j)\Big\|_*^{2}\right)^{\frac{1}{2}}.
\end{equation}
Here, $\Xi_{\mathrm{reg}}=\{\zeta_j\}_{j=1}^m$ represents a finite set of $m$ pre-selected \textit{scenarios}, and $r_j \ge 0$ are the corresponding weights such that $\sum_{j=1}^{m}r_j=1$. The term $\nabla_\xi F(x,\zeta_j)$ denotes the gradient of the objective function $F(x,\xi)$ with respect to the random variable $\xi$, evaluated at the decision $x$ and scenario $\zeta_j$. We use the dual norm $\|\cdot\|_*$ to maintain consistency with Lemma \ref{Lemma:taylorExpansion_intro_revised}. In cases where $F(x,\cdot)$ is not differentiable with respect to $\xi$, the gradient can be replaced by a subgradient (i.e., an element of the subdifferential) of $F(x,\cdot)$ at $\zeta_j$.

\sloppy The regularization term $\widehat{R}_{m}(x)$ in \eqref{eqn:RegTerm_sec} penalizes the weighted average magnitude of the (sub)gradients of $F(x,\cdot)$ at the scenarios in $\Xi_{\mathrm{reg}}$. The flexibility of this framework stems from the choice of $\Xi_{\mathrm{reg}}$ and the weights $\{r_j\}$, which enables the two distinct purposes identified in the introduction.

\begin{enumerate}
    \item \textbf{WDRO surrogate:} The goal is to create a computationally tractable alternative to standard WDRO that exhibits good out-of-sample performance (in mean and tail). Here, $\Xi_{\mathrm{reg}}$ is chosen as a representative subset of the full empirical sample $\hat{\Xi}_n$, i.e., $\Xi_{\mathrm{reg}} \subset \hat{\Xi}_n$ with $m < n$. This is motivated by the fact that using $m=n$ to approximate the regularizer term $\mathbb{E}_{\xi\sim\hat{\mathbb{P}}_n}[\|\nabla_{\xi} F(x,\xi)\|_*]$ can be computationally prohibitive if evaluating the gradient norm is expensive. By using a smaller and properly selected subset, we aim to retain the robustness properties of WDRO at a fraction of the computational cost. The selection of $\Xi_{\mathrm{reg}}$ and the weights $\{r_j\}$ becomes a key modeling decision, for which we will later propose a principled method based on k-medoids and Wasserstein distance minimization to ensure the $m$-point measure well-approximates $\hat{\mathbb{P}}_n$.
    
    \item \textbf{Targeted Robustness:} The goal is to protect the solution against specific, high-impact events explicitly. Here, $\Xi_{\mathrm{reg}}$ is defined as a set of \textbf{adverse scenarios}. An adverse scenario $\zeta_j$ is a specific realization of $\xi$ that signifies extreme, unfavorable, or highly sensitive conditions for the system. These scenarios might be identified from historical data (e.g., the 2008 financial crisis, the 2020 COVID-19 pandemic), derived from expert judgment, or generated via stress testing. In this context, $\widehat{R}_m(x)$ acts as a targeted penalty, pushing the optimizer towards solutions $x$ that are less sensitive in these critical regions. This provides a mechanism for incorporating domain-specific knowledge about tail risks directly into the optimization, a feature not available in standard data-driven WDRO.
\end{enumerate}

By incorporating this regularizer, the SBR-SAA framework aims to find a decision $x$ that not only performs well on average (due to the SAA term) but is also robust. Depending on the choice of $\Xi_{\mathrm{reg}}$, this robustness manifests either as stable out-of-sample performance (emulating WDRO) or as explicit insensitivity to known critical conditions. The weights $r_j$ provide additional flexibility in assigning relative importance to each scenario in $\Xi_{\mathrm{reg}}$.


\subsection{Equivalence to a Decision-Dependent Wasserstein DRO}
\label{subsec:EquivalenceWDRO}

To analyze the theoretical properties of our proposed approach \eqref{eq:SBR-SAA_sec}, particularly its robustness characteristics, we first establish its connection to a form of Distributionally Robust Optimization (DRO). We consider a more general regularized SAA problem:
\begin{equation}\label{eqn:SAARegGeneral_sec}
\min_{x\in \mathcal{X}} \frac{1}{n}\sum_{i=1}^{n}F\bigl(x,\widehat{\xi}_{i}\bigr) + \varepsilon R(x),
\end{equation}
where $R(x)$ is a non-negative function, $R:\mathcal{X}\to \mathbb{R}_{+}$, which may or may not depend on the sample $\hat{\Xi}_n$. We adopt this general perspective first to demonstrate the broad applicability of the theoretical framework. Our proposed problem \eqref{eq:SBR-SAA_sec} is a special case of \eqref{eqn:SAARegGeneral_sec} with $R(x)$ defined as in \eqref{eqn:RegTerm_sec}. We will show that problems of this form are equivalent to a specific type of WDRO where the ambiguity set itself depends on the decision variable $x$.

\begin{theorem}\label{Thm:EquivalenceRegSAAvsWDRO}
Let $\varepsilon>0$. If $p \geq 1$, $R(x)\geq 0$ (which may depend on the sample), and for all $x\in\mathcal{X}$, $F(x,\Xi)$ is an interval and $\sup_{\xi\in\Xi}F(x,\xi)=\infty$, then the problem \eqref{eqn:SAARegGeneral_sec} is equivalent, almost surely, to
\begin{equation} \label{eqn:WDROFormSAARegGen_sec}
\widehat{J}^{\mathrm{DD}}_{R,n}(\varepsilon):={\displaystyle \min_{x\in\mathcal{X}} }  {\displaystyle\sup_{\mathbb{Q}\in\mathcal{B}_{\varepsilon R(x)}\left(F(x,\cdot)_{\#}\widehat{\mathbb{P}}_{n}\right) } \mathbb{E}_{\varsigma\sim\mathbb{Q}}[\varsigma] }
\end{equation}
where $\mathcal{B}_{\varepsilon R(x)}\left(F(x,\cdot)_{\#}\widehat{\mathbb{P}}_{n}\right)$ is a ball centered at the propagated empirical distribution $F(x,\cdot)_{\#}\widehat{\mathbb{P}}_{n}$ with radius $\varepsilon R(x)$. This ball is defined with respect to the $p$-Wasserstein distance in $\mathbb{R}$, where the ground cost function used is $\mathbf{d}(\varsigma_1,\varsigma_2)=\left|\varsigma_1-\varsigma_2\right|$.
Specifically, for any given $x \in \mathcal{X}$ and sample realization, the following identity holds almost surely:
\begin{equation} \label{eqn:WDROFormSAARegGen_presec}
\frac{1}{n}\sum_{i=1}^{n}F\bigl(x,\widehat{\xi}_{i}\bigr) + \varepsilon R(x) =   {\displaystyle\sup_{\mathbb{Q}\in\mathcal{B}_{\varepsilon R(x)}\left(F(x,\cdot)_{\#}\widehat{\mathbb{P}}_{n}\right) } \mathbb{E}_{\varsigma\sim\mathbb{Q}}[\varsigma] }
\end{equation}
\end{theorem}
\begin{proof}
We first consider the case where $R(x) = C$ is a deterministic constant $C > 0$. Under the given assumptions, the identity 
\begin{equation}\label{eqn:RegToDRODDFixed}
\frac{1}{n}\sum_{i=1}^{n}F\bigl(x,\widehat{\xi}_{i}\bigr)\;+\;\varepsilon\,C
\;=\;
\sup_{\mathbb{Q}\in\mathcal{B}_{\varepsilon C}\!\left(F(x,\cdot)_{\#}\widehat{\mathbb{P}}_{n}\right)} \ \mathbb{E}_{\varsigma\sim\mathbb{Q}}[\varsigma]
\end{equation}
holds for any $x\in\mathcal{X}$. The proof of this deterministic result is a known property related to the dual representation of 1-Wasserstein distance on $\mathbb{R}$ and is deferred to the Appendix \ref{Sec:Appendix:ProofThmEquivalenceRegSAAvsWDRO}.

Now, we extend this to the case where $R(x)$ depends on the sample $\hat{\Xi}_n$. We fix a sample realization $\omega$ from the underlying probability space $\Omega^n$. For this fixed $\omega$, both the empirical measure $\widehat{\mathbb{P}}_{n}(\omega)$ and the regularizer $R(x)(\omega)$ are fixed entities. Let $\mu_\omega := F(x,\cdot)_{\#}\widehat{\mathbb{P}}_{n}(\omega)$ and $R_\omega := R(x)(\omega) > 0$. We can now apply the deterministic result \eqref{eqn:RegToDRODDFixed} using the fixed center $\mu_\omega$ and the fixed constant radius $C = R_\omega$. This yields:
\[
\frac{1}{n}\sum_{i=1}^{n}F\bigl(x,\widehat{\xi}_{i}(\omega)\bigr)\;+\;\varepsilon\,R_\omega
\;=\;
\sup_{\mathbb{Q}\in\mathcal{B}_{\varepsilon R_\omega}\!\left(\mu_\omega\right)} \ \mathbb{E}_{\varsigma\sim\mathbb{Q}}[\varsigma].
\]
Since this equality holds pointwise for almost every realization $\omega$, the identity \eqref{eqn:WDROFormSAARegGen_presec} holds almost surely. The equivalence of the minimization problems in \eqref{eqn:WDROFormSAARegGen_sec} follows by taking the minimum with respect to $x\in\mathcal{X}$ on both sides.
\end{proof}

We use the superscript "DD" in $\widehat{J}^{\mathrm{DD}}_{R,n}(\varepsilon)$ to emphasize that the ambiguity set in \eqref{eqn:WDROFormSAARegGen_sec} is \textit{decision-dependent}. This formulation contrasts with traditional WDRO models (such as those based on \eqref{DROGeneral_intro} with a fixed Wasserstein ball around $\widehat{\mathbb{P}}_n$). In standard WDRO, the ambiguity set typically comprises probability distributions on $\Xi$. In contrast, problem \eqref{eqn:WDROFormSAARegGen_sec} considers an ambiguity set of distributions on $\mathbb{R}$ (the space of outcomes of $F(x,\xi)$). Crucially, for each decision $x$, a distinct ambiguity set is defined, with both its center $F(x,\cdot)_{\#}\widehat{\mathbb{P}}_{n}$ and its radius $\varepsilon R(x)$ being functions of $x$.


\subsection{Theoretical Guarantees}
\label{subsec:TheoreticalGuarantees}

The equivalence established in Theorem \ref{Thm:EquivalenceRegSAAvsWDRO} allows us to leverage tools from DRO theory to analyze the properties of the general regularized SAA approach \eqref{eqn:SAARegGeneral_sec}. We now focus on establishing finite sample guarantees and asymptotic consistency for these methods.

\subsubsection{Preliminary Results}

We begin by noting a simple but useful observation that will streamline the analysis. Under the conditions of Theorem~\ref{Thm:EquivalenceRegSAAvsWDRO}, replacing the Wasserstein radius in \eqref{eqn:WDROFormSAARegGen_sec} by a state–dependent scaling $R(x)+\alpha$ (with any fixed $\alpha>0$) preserves the set of \emph{minimizers in $x$}. Thus, while the optimal objective values may differ in general, the argmin sets of the two programs coincide. For convenience, we therefore work with the argmin–equivalent formulation
\begin{equation}\label{eqn:WDROFormSAARegGenTranslac_sec}
\widehat{J}^{\mathrm{DD}}_{R+\alpha,n}(\varepsilon)
:=\min_{x\in\mathcal{X}} \ \sup_{\mathbb{Q}\in\mathcal{B}_{\varepsilon (R(x)+\alpha)}\!\left(F(x,\cdot)_{\#}\widehat{\mathbb{P}}_{n}\right)} 
\ \mathbb{E}_{\varsigma\sim\mathbb{Q}}[\varsigma],
\end{equation}
with the understanding that it yields the same optimal solutions $x^\star$ as \eqref{eqn:WDROFormSAARegGen_sec} even though the corresponding optimal values need not coincide.

Our objective is to analyze the behavior of \eqref{eqn:WDROFormSAARegGenTranslac_sec} as $n \to \infty$ and to determine whether an appropriate choice of $\varepsilon$ ensures that its optimal value and solutions converge to those of the true problem \eqref{eqn:SP_intro}. The following assumption regarding the Lipschitz continuity of $F$ with respect to $\xi$ is key.

\begin{assumption}[Lipschitz Continuity of $F$]\label{AssumptionPrincipal_sec}
We assume that $F(x,\cdot)$ is a Lipschitz function with respect to $\xi$ for each $x \in \mathcal{X}$. That is, there exists $\gamma_{x,F} < \infty$ such that for all $\xi_1,\xi_2\in\Xi$, we have $|F(x,\xi_1)-F(x,\xi_2)|\leq \gamma_{x,F}\left\|\xi_1-\xi_2\right\|$, where $\|\cdot\|$ is a norm on $\mathbb{R}^d$. 
\end{assumption}

The next lemma relates the Wasserstein distance between propagated distributions to the Wasserstein distance between the original distributions.

\begin{lemma}\label{lem:WassersteinPropagationBound_user}
Let $\alpha\geq 0$ such that $\gamma_{x,F} \leq R(x)+\alpha$ for all $x\in\mathcal{X}$ almost surely (assuming $R(x)$ could depend on the sample). If $\sup_{x\in \mathcal{X}} \gamma_{x,F}<\infty$, then for $p \geq 1$ it follows that
\[
W_{p}\bigl(F(x,\cdot)_{\#}\widehat{\mathbb{P}}_{n},F(x,\cdot)_{\#}\mathbb{P}\bigr) \leq (R(x)+\alpha) \, W_{p}(\widehat{\mathbb{P}}_{n},\mathbb{P}),
\]
where $W_{p}(\widehat{\mathbb{P}}_{n},\mathbb{P})$ uses ground cost $\mathbf{d}(\xi_1,\xi_2)=\|\xi_1-\xi_2\|$ and $W_{p}\bigl(F(x,\cdot)_{\#}\widehat{\mathbb{P}}_{n},F(x,\cdot)_{\#}\mathbb{P}\bigr)$ uses ground cost $\mathbf{d}(\varsigma_1,\varsigma_2)=|\varsigma_1-\varsigma_2|$.
\end{lemma}
\noindent The proof of this lemma is deferred to Appendix \ref{Apendice:Pruebalem:WassersteinPropagationBound_user}.

\noindent
Note that if $\gamma_{x,F} \leq R(x)$ for all $x\in\mathcal{X}$, then one may simply take $\alpha=0$. As a direct consequence of Lemma~\ref{lem:WassersteinPropagationBound_user} together with the argmin–equivalence remark above, we obtain the following probabilistic containment.

\begin{corollary}\label{cor:ProbRelationContainment}
Under the same conditions as Lemma \ref{lem:WassersteinPropagationBound_user} and Theorem \ref{Thm:EquivalenceRegSAAvsWDRO},
\[
\mathbb{P}^{n}\!\left(\forall x\in \mathcal{X},\:\: F(x,\cdot)_{\#}\mathbb{P}\in \mathcal{B}_{\varepsilon\cdot(R(x)+\alpha)}\!\left(F(x,\cdot)_{\#}\widehat{\mathbb{P}}_{n}\right)\right) 
\ \geq\
\mathbb{P}^{n}\!\left(\mathbb{P}\in \mathcal{B}_{\varepsilon}(\widehat{\mathbb{P}}_{n})\right).
\]
Here, the ball $\mathcal{B}_{\varepsilon\cdot(R(x)+\alpha)}\left(F(x,\cdot)_{\#}\widehat{\mathbb{P}}_{n}\right)$ uses the $p$-Wasserstein distance with ground cost $|\varsigma_1-\varsigma_2|$, and $\mathcal{B}_{\varepsilon}(\widehat{\mathbb{P}}_{n})$ uses the $p$-Wasserstein distance with ground cost $\|\xi_1-\xi_2\|$.
\end{corollary}

The Kantorovich–Rubinstein duality theorem provides a fundamental connection for the 1-Wasserstein distance.

\begin{theorem}[Kantorovich–Rubinstein Duality]\label{Thm:KantorovichRubinsteinDD_sec}
For any $x\in\mathcal{X}$ and any probability measure $\mathbb{Q}$ on $F(x,\Xi)$, the 1-Wasserstein distance is given by
\[
W_{1}\left(\mathbb{Q},F(x,\cdot)_{\#}\mathbb{P}\right) = \sup_{f: F(x,\Xi) \to \mathbb{R}, \|f\|_{\mathrm{Lip}} \le 1} \left\{ \int_{F(x,\Xi)} f(\varsigma)\,\mathbb{Q}(d\varsigma) - \int_{F(x,\Xi)} f(\varsigma)\,(F(x,\cdot)_{\#}\mathbb{P})(d\varsigma) \right\},
\]
where $\|f\|_{\mathrm{Lip}} \le 1$ means $f$ is 1-Lipschitz, i.e., $|f(\varsigma_1)-f(\varsigma_2)| \le |\varsigma_1-\varsigma_2|$ for all $\varsigma_1,\varsigma_2 \in F(x,\Xi)$.
\end{theorem}


\subsubsection{Finite Sample Guarantees}

In this section, let us recall that originally in \eqref{eq:SBR-SAA_sec} we noted the regularizer as $\widehat{R}_{m}$, where by definition $\widehat{R}_{m}$ depends on $m$ scenarios. These scenarios may or may not depend on the sample, depending on our purposes. For this reason, the previous results (Theorem \ref{Thm:EquivalenceRegSAAvsWDRO} -- Corollary \ref{cor:ProbRelationContainment}) focus on a general regularizer $R$ that may or may not depend on the sample. This general treatment was deliberate, as it covers the forms that $\widehat{R}_{m}$ can take. We now proceed by specializing this general framework to our specific regularizer $\widehat{R}_{m}(x)$ from \eqref{eqn:RegTerm_sec}.

To establish finite sample guarantees, we introduce further assumptions, inspired by \cite{MohajerinEsfahani2018}.

\begin{assumption}\label{Assumption1FiniteGarant_sec}
There exists $a>1$ such that
\[
A := \mathbb{E}_{\xi\sim\mathbb{P}}[\exp(\|\xi\|^{a})] = \int_\Xi \exp(\|\xi\|^{a}) \mathbb{P}(d\xi) < \infty.
\]
\end{assumption}

\begin{assumption}\label{Assumption2FiniteGarant_sec}
The function $F$ satisfies $\sup_{x\in \mathcal{X}} \gamma_{x,F}< \infty$. For all $x\in\mathcal{X}$, $F(x,\Xi)$ is an interval with $\sup_{\xi\in\Xi}F(x,\xi)=\infty$. Moreover, there exists $\alpha \ge 0$ such that $\gamma_{x,F} \leq \widehat{R}_{m}(x)+\alpha$ for all $x\in\mathcal{X}$ almost surely.
\end{assumption}

Under these assumptions, we can characterize the radius $\varepsilon_n$ such that the true propagated distribution $F(x,\cdot)_{\#}\mathbb{P}$ lies within the decision-dependent Wasserstein ball around $F(x,\cdot)_{\#}\widehat{\mathbb{P}}_{n}$ with high probability.

\begin{lemma}\label{lemma:EpsilonCharacterizaion_sec}
If Assumptions \ref{Assumption1FiniteGarant_sec} and \ref{Assumption2FiniteGarant_sec} hold, and if we choose $\beta\in(0,1)$, then there exists positive constants $c_{1}$ and $c_{2}$ that only depend on $a$, $A$ and $d$ such that if we set
\[
\varepsilon_n(\beta) := 
\begin{cases}
\left(\frac{\log(c_1 \beta^{-1})}{c_2 n}\right)^{1/\max\{d,2\}} & \text{if } n \ge \frac{\log(c_1 \beta^{-1})}{c_2}, \\[10pt]
\left(\frac{\log(c_1 \beta^{-1})}{c_2 n}\right)^{1/a} & \text{if } n < \frac{\log(c_1 \beta^{-1})}{c_2},
\end{cases}
\]
then
\[
\mathbb{P}^{n}\left(\forall x\in \mathcal{X},\:\: F(x,\cdot)_{\#}\mathbb{P}\in \mathcal{B}_{\varepsilon_n(\beta)\cdot(\widehat{R}_{m}(x)+\alpha)}\left(F(x,\cdot)_{\#}\widehat{\mathbb{P}}_{n}\right) \right) > 1-\beta.
\]
\end{lemma}
\begin{proof}
This lemma is a direct consequence of Corollary \ref{cor:ProbRelationContainment} (with $R=\widehat{R}_{m}$), which relates the containment of the propagated measures to the containment of the original measures, and Theorem 3.4 in \citep{MohajerinEsfahani2018}, which provides the convergence rate $\varepsilon_n(\beta)$ for $W_p(\widehat{\mathbb{P}}_n, \mathbb{P})$.
\end{proof}

An immediate consequence is:

\begin{corollary}\label{corolary:ContenencyBall_sec}
If Assumptions \ref{Assumption1FiniteGarant_sec} and \ref{Assumption2FiniteGarant_sec} hold, and if $\beta\in(0,1)$ and $\varepsilon_n(\beta)$ is chosen as in Lemma \ref{lemma:EpsilonCharacterizaion_sec}, then for any optimal solution $\hat{x}_{m,n}$ of \eqref{eqn:SAARegGeneral_sec} (with $R=\widehat{R}_m$ and $\varepsilon=\varepsilon_n(\beta)$), we have 
\[
\mathbb{P}^{n}\left(F(\hat{x}_{m,n},\cdot)_{\#}\mathbb{P}\in \mathcal{B}_{\varepsilon_n(\beta)\cdot(\widehat{R}_{m}(\hat{x}_{m,n})+\alpha)}\left(F(\hat{x}_{m,n},\cdot)_{\#}\widehat{\mathbb{P}}_{n}\right) \right) > 1-\beta.
\]
\end{corollary}
\begin{proof}
The event in the probability statement is a subset of the event $\{\forall x\in \mathcal{X},\:\: F(x,\cdot)_{\#}\mathbb{P}\in \mathcal{B}_{\varepsilon_n(\beta)\cdot(\widehat{R}_{m}(x)+\alpha)}\left(F(x,\cdot)_{\#}\widehat{\mathbb{P}}_{n}\right) \}$. The result follows directly from Lemma \ref{lemma:EpsilonCharacterizaion_sec}.
\end{proof}

We now establish the finite sample guarantee, which provides a probabilistic upper bound on the out-of-sample performance of the solution obtained from our regularized problem.

\begin{theorem} [Finite sample guarantees]\label{Thm:IneuqlityProbaOptimalValueAdv}
Assume that Assumptions \ref{Assumption1FiniteGarant_sec} and \ref{Assumption2FiniteGarant_sec} hold. Consider $\beta\in(0,1)$ and let $\varepsilon_n(\beta)$ be as in Lemma \ref{lemma:EpsilonCharacterizaion_sec}. If $\widehat{J}_{m,n}^{\mathrm{DD}}$ and $\hat{x}_{m,n}^{\mathrm{DD}}$ denote the optimal value and an optimal solution to \eqref{eqn:WDROFormSAARegGen_sec} with $R=\widehat{R}_{m}$ and $\varepsilon= \varepsilon_n(\beta)$ (thus, by Theorem \ref{Thm:EquivalenceRegSAAvsWDRO}, $\hat{x}_{m,n}^{\mathrm{DD}}$ is also an optimal solution of \eqref{eqn:SAARegGeneral_sec}), then
\[
\mathbb{P}^{n}\left( \mathbb{E}_{\xi\sim\mathbb{P}}\left[F\left(\hat{x}_{m,n}^{\mathrm{DD}},\xi\right)\right] \leq \widehat{J}_{m,n}^{\mathrm{DD}} + \varepsilon_n(\beta)\alpha \right) > 1-\beta.
\]
\end{theorem}

\begin{proof}
By definition, $\mathbb{E}_{\xi\sim\mathbb{P}}\left[F\left(\hat{x}_{m,n}^{\mathrm{DD}},\xi\right)\right]=\mathbb{E}_{\varsigma\sim F\left(\hat{x}_{m,n}^{\mathrm{DD}},\cdot\right)_{\#}\mathbb{P}}[\varsigma]$. Let $\mathcal{E}$ be the event from Corollary \ref{corolary:ContenencyBall_sec}, which occurs with probability $> 1-\beta$. On this event, the true propagated measure $F(\hat{x}_{m,n}^{\mathrm{DD}},\cdot)_{\#}\mathbb{P}$ is contained in the ambiguity set $\mathcal{B}_{\varepsilon_n(\beta)\cdot(\widehat{R}_{m}(\hat{x}_{m,n}^{\mathrm{DD}})+\alpha)}\left(F(\hat{x}_{m,n}^{\mathrm{DD}},\cdot)_{\#}\widehat{\mathbb{P}}_{n}\right)$.
Therefore, on this event $\mathcal{E}$, we have:
\begin{align*}
\mathbb{E}_{\xi\sim\mathbb{P}}\left[F\left(\hat{x}_{m,n}^{\mathrm{DD}},\xi\right)\right] &= \mathbb{E}_{\varsigma\sim F\left(\hat{x}_{m,n}^{\mathrm{DD}},\cdot\right)_{\#}\mathbb{P}}[\varsigma] \\
&\leq {\displaystyle\sup_{\mathbb{Q}\in\mathcal{B}_{\varepsilon_n(\beta) \left(\widehat{R}_{m}\left(\hat{x}_{m,n}^{\mathrm{DD}}\right)+\alpha\right)}\left(F(\hat{x}_{m,n}^{\mathrm{DD}},\cdot)_{\#}\widehat{\mathbb{P}}_{n}\right) } \mathbb{E}_{\varsigma\sim\mathbb{Q}}[\varsigma] } \\
&= \widehat{J}_{m,n}^{\mathrm{DD}}+\varepsilon_{n}(\beta)\alpha.
\end{align*}
The last equality follows from Theorem \ref{Thm:EquivalenceRegSAAvsWDRO}. Since this inequality holds on the event $\mathcal{E}$, the theorem is proven.
\end{proof}


\subsubsection{Asymptotic Consistency}

We now turn to the asymptotic behavior of the proposed method as the sample size $n \to \infty$.
\begin{lemma}[Convergence of Distributions]\label{lem:ConvergenceOfdistributions_sec}
If Assumptions \ref{Assumption1FiniteGarant_sec} and \ref{Assumption2FiniteGarant_sec} hold, and $\widehat{R}_{m}(x)<\infty$ for all $x\in\mathcal{X}$, and if $\beta_n \in (0,1)$, $n \in \mathbb{N}$, satisfies $\sum_{n=1}^\infty \beta_n < \infty$ and $\lim_{n \to \infty}\varepsilon_n(\beta_n)=0$ (where $\varepsilon_n(\beta_n)$ is as in Lemma \ref{lemma:EpsilonCharacterizaion_sec}), then for each $x\in\mathcal{X}$, any sequence $\hat{\mathbb{Q}}_n^{x} \in \mathcal{B}_{\varepsilon_n(\beta_n)\cdot(\widehat{R}_{m}(x)+\alpha)}\left(F(x,\cdot)_{\#}\widehat{\mathbb{P}}_{n}\right)$ (which may depend on the training data) converges under the 1-Wasserstein metric (and thus weakly) to $F(x,\cdot)_{\#}\mathbb{P}$ almost surely with respect to $\mathbb{P}^\infty$:
\[
\mathbb{P}^\infty\left(\forall x\in\mathcal{X},\:\lim_{n \to \infty} W_{1}\bigl(F(x,\cdot)_{\#}\mathbb{P},\hat{\mathbb{Q}}_n^{x}\bigr)=0\right)=1.
\]
\end{lemma}
\begin{proof}
Since $\hat{\mathbb{Q}}_n^{x} \in \mathcal{B}_{\varepsilon_n(\beta_n)\cdot(\widehat{R}_{m}(x)+\alpha)}\left(F(x,\cdot)_{\#}\widehat{\mathbb{P}}_{n}\right)$, we have $W_{1}(F(x,\cdot)_{\#}\widehat{\mathbb{P}}_{n},\hat{\mathbb{Q}}_n^{x}) \le \varepsilon_n(\beta_n)(\widehat{R}_{m}(x)+\alpha)$.
By the triangle inequality,
\begin{align*}
W_{1}(F(x,\cdot)_{\#}\mathbb{P},\hat{\mathbb{Q}}_n^{x}) &\leq W_{1}(F(x,\cdot)_{\#}\mathbb{P},F(x,\cdot)_{\#}\widehat{\mathbb{P}}_{n}) + W_{1}(F(x,\cdot)_{\#}\widehat{\mathbb{P}}_{n},\hat{\mathbb{Q}}_n^{x}) \\
&\leq  W_{1}(F(x,\cdot)_{\#}\mathbb{P},F(x,\cdot)_{\#}\widehat{\mathbb{P}}_{n}) + \varepsilon_n(\beta_n)(\widehat{R}_{m}(x)+\alpha).
\end{align*}
From Lemma \ref{lemma:EpsilonCharacterizaion_sec} (specifically, the condition $\mathbb{P}^{n}\left(\forall x\in\mathcal{X},\:\: W_1(F(x,\cdot)_{\#}\widehat{\mathbb{P}}_n, F(x,\cdot)_{\#}\mathbb{P}) \leq \varepsilon_n(\beta_n)(\widehat{R}_{m}(x)+\alpha) \right) > 1-\beta_n$ which follows from it), we have that with probability at least $1-\beta_n$,
\[
\forall x\in\mathcal{X},\:\: W_{1}(F(x,\cdot)_{\#}\mathbb{P},F(x,\cdot)_{\#}\widehat{\mathbb{P}}_{n}) \leq \varepsilon_n(\beta_n)(\widehat{R}_{m}(x)+\alpha).
\]
Thus, with probability at least $1-\beta_n$,
\[
\mathbb{P}^n\left(\forall x\in\mathcal{X}, W_{1}(F(x,\cdot)_{\#}\mathbb{P},\hat{\mathbb{Q}}_n^{x}) \leq 2\varepsilon_n(\beta_n)(\widehat{R}_{m}(x)+\alpha)\right)\geq 1-\beta_n.
\]
Since $\sum_{n=1}^\infty \beta_n < \infty$, the Borel--Cantelli Lemma implies that with $\mathbb{P}^\infty$-probability 1, $W_{1}(F(x,\cdot)_{\#}\mathbb{P},\hat{\mathbb{Q}}_n^{x}) \leq 2\varepsilon_n(\beta_n)(\widehat{R}_{m}(x)+\alpha)$ for all sufficiently large $n$, for every $x\in\mathcal{X}$. As $\lim_{n \to \infty}\varepsilon_n(\beta_n)=0$ and $\widehat{R}_{m}(x)+\alpha$ is bounded, it follows that $W_{1}(F(x,\cdot)_{\#}\mathbb{P},\hat{\mathbb{Q}}_n^{x})\to 0$ almost surely.
\end{proof}

We are now prepared to state the main result on asymptotic consistency.

\begin{theorem}[Asymptotic Consistency]\label{thm:AsymptoticConsistency_sec}
Suppose Assumptions \ref{Assumption1FiniteGarant_sec} and \ref{Assumption2FiniteGarant_sec} hold (for $p=1$). Let $\beta_n \in (0,1)$, $n \in \mathbb{N}$, satisfy $\sum_{n=1}^\infty \beta_n < \infty$ and $\lim_{n\to\infty}\varepsilon_n(\beta_n)=0$, where $\varepsilon_n(\beta_n)$ is as in Lemma \ref{lemma:EpsilonCharacterizaion_sec}. Let $\widehat{J}_{m,n}^{\mathrm{DD}}$ and $\hat{x}_{m,n}^{\mathrm{DD}}$ denote the optimal value and an optimizer of \eqref{eqn:WDROFormSAARegGen_sec} with $R=\widehat{R}_{m}$ and  $\varepsilon = \varepsilon_n(\beta_n)$, for each $n \in \mathbb{N}$. Thus, $\hat{x}_{m,n}^{\mathrm{DD}}$ is also an optimal solution of \eqref{eqn:SAARegGeneral_sec} with $\varepsilon = \varepsilon_n(\beta_n)$.

\begin{enumerate}
\item[(i)] If $F(x,\xi)$ is upper semicontinuous in $x$ for each $\xi \in \Xi$, then $\lim_{n \to \infty} \widehat{J}_{m,n}^{\mathrm{DD}} = J^*$ almost surely, where $J^*$ is the optimal value of the true problem \eqref{eqn:SP_intro}.

\item[(ii)] If the conditions of (i) hold, $\mathcal{X}$ is closed, $F(x,\xi)$ is lower semicontinuous in $x$ for each $\xi \in \Xi$, and there exists $L \ge 0$ such that $|F(x,\xi)| \le L(1+\|\xi\|)$ for all $x \in \mathcal{X}$ and $\xi \in \Xi$, then any accumulation point of the sequence $\{\hat{x}_{m,n}^{\mathrm{DD}}\}_{n \in \mathbb{N}}$ is almost surely an optimal solution of \eqref{eqn:SP_intro}.
\end{enumerate}
\end{theorem}

\noindent The proof of this theorem is deferred to Appendix \ref{sec:ProofAsymptoticConsistency_sec}.

This Asymptotic Consistency Theorem is of paramount importance. Part (i) assures us that, as the amount of available data ($n$) increases, the optimal value obtained from our regularized SAA approach (which is equivalent to the decision-dependent WDRO) converges to the true optimal value of the underlying stochastic problem \eqref{eqn:SP_intro}. This means our method does not systematically overestimate or underestimate the true achievable performance in the long run. Part (ii) provides an even stronger guarantee: the actual solutions (decision vectors) generated by our method converge to the true optimal decisions under an appropriate choice of $\varepsilon$. This ensures that the method is not only finding the right performance level but is also identifying the correct strategies or configurations as data accumulates. Together, these results provide strong theoretical validation for the proposed approach, indicating its reliability and convergence to optimality in data-rich environments.

%

\section{Numerical Experiments}
\label{sec:NumericalExperiments}

In this section, we validate our proposed Scenario-based Regularized SAA framework through two distinct numerical experiments. The first experiment, on a multi-product newsvendor problem, uses simulated data to demonstrate the method's utility as a computationally tractable subrogate to WDRO for achieving robust out-of-sample performance. The second experiment, on a mean-risk portfolio optimization problem, uses real-world financial data to showcase the method's ability to incorporate expert-defined adverse scenarios, as discussed in our motivation.

\subsection{Multi-Product Newsvendor Problem}
\label{subsec:Newsvendor}

This first experiment addresses the classic multi-product newsvendor problem. We first formulate the problem, analyze its properties (subgradient and Lipschitz modulus), and then detail the intractability or inadequacy of standard WDRO approaches. This analysis provides the critical motivation for applying our SBR-SAA method, for which we derive an exact MISOCP reformulation. Finally, we present the experimental design and numerical results.

\subsubsection{Problem Formulation}

We consider a single-period, multi-product Newsvendor problem with $d$ items. The decision vector is $x\in\mathbb{R}^d$ (order quantities), and the random demand is $\xi\in\Xi:=\{\xi\in\mathbb{R}^d:\ \xi_i\ge 0,\ i=1,\dots,d\}$. The feasible set is
\[
\mathcal{X}:=\{x\in\mathbb{R}^d:\ 0\le x_i\le a_i,\ i=1,\dots,d\},\qquad a\in\mathbb{R}_+^d,
\]
and we define the component-wise minimum operator $\min\{x,\xi\}:=(\min\{x_i,\xi_i\})_{i=1}^d$.

\paragraph{Cost Function.}
Given $x$ and a realization $\xi$, the cost (negative net revenue) is
\begin{equation}\label{eq:F-original-News}
F(x,\xi) \;=\; c^\top x - v^\top \min\{x,\xi\} - g^\top (x-\min\{x,\xi\}) + b^\top (\xi-\min\{x,\xi\}),
\end{equation}
where $c$ is the unit procurement cost, $v$ the selling price, $g$ the salvage value, and $b$ the shortage penalty. We adopt the standard economic assumptions (component-wise):
\begin{equation}\label{eq:assumptions-News}
v>g,\qquad b\ge 0,\qquad b\ge v-g.
\end{equation}
The first assumption implies selling is more profitable than salvaging. The second is natural for penalties. The third, $b \ge v-g$, requires that failing to meet demand costs at least the lost opportunity $v-g$. These hypotheses are common and reasonable in applications involving opportunity costs and loss of goodwill.

From \eqref{eq:F-original-News}, we obtain the equivalent compact form
\begin{equation}\label{eq:F-compacta-News}
F(x,\xi) \;=\; (c-g)^\top x \;+\; b^\top \xi \;+\; (g-v-b)^\top \min\{x,\xi\}.
\end{equation}
Note that under \eqref{eq:assumptions-News}, the vector $g-v-b$ is strictly negative. The stochastic optimization problem is thus:
\begin{equation} \label{eqn:SPNewsvendorMultiproduct}
\min_{x\in\mathcal{X}}\mathbb{E}_{\xi\sim\mathbb{P}}[F(x,\xi)].
\end{equation}

\subsubsection{Problem Properties: Subgradient and Lipschitz Modulus}

The coordinate-wise structure of \eqref{eq:F-compacta-News} allows for a simple characterization of the subdifferential with respect to $\xi$.

\begin{proposition}[Subdifferential in $\xi$]\label{prop:subdiff-News}
For a fixed $x\in\mathcal{X}$, the subdifferential $\partial_{\xi}F(x,\xi)$ is the Cartesian product of the intervals $\partial_{\xi_i}F(x,\xi)$ for $i=1,\dots,d$, where
\[
\partial_{\xi_i}F(x,\xi)=
\begin{cases}
\{\,g_i-v_i\,\}, & \xi_i<x_i,\\[2pt]
[g_i-v_i,\,b_i], & \xi_i=x_i,\\[2pt]
\{\,b_i\,\}, & \xi_i>x_i.
\end{cases}
\]
\end{proposition}
\begin{proof}
The function $F(x,\xi)$ is separable with respect to $\xi$: $F(x,\xi)=\mathrm{const}(x)+\sum_{i=1}^d \phi_i(\xi_i)$, where
\[
\phi_i(\xi_i)=b_i\,\xi_i+(g_i-v_i-b_i)\min\{x_i,\xi_i\}.
\]
If $\xi_i<x_i$, $\min\{x_i,\xi_i\}=\xi_i$, so $\phi_i(\xi_i)=(b_i+g_i-v_i-b_i)\xi_i=(g_i-v_i)\,\xi_i$, and the derivative is $g_i-v_i$.
If $\xi_i>x_i$, $\min\{x_i,\xi_i\}=x_i$, so $\phi_i(\xi_i)=b_i\,\xi_i+(g_i-v_i-b_i)x_i$, and the derivative is $b_i$.
At the kink $\xi_i=x_i$, $\phi_i$ is a concave function (as $g_i-v_i-b_i < 0$ by assumption \eqref{eq:assumptions-News}). The subdifferential is the interval defined by the right and left derivatives, $[\partial^+\phi_i, \partial^-\phi_i] = [b_i, g_i-v_i]$. Given $b \ge v-g > 0$, we have $b_i \ge v_i-g_i > 0$, so $b_i > 0$ and $g_i-v_i < 0$, which implies $g_i-v_i < b_i$. Thus, the subdifferential is $[g_i-v_i, b_i]$.
\end{proof}

\begin{corollary}[Lipschitz Modulus in $\xi$]\label{cor:lipschitz-News}
Under \eqref{eq:assumptions-News}, the Lipschitz modulus of $F(x,\cdot)$ with respect to the $\ell_2$-norm satisfies
\[
\gamma_{x,F}=\sup_{\varphi\in\partial_\xi F(x,\xi),\ \xi\in\mathbb{R}_+^d}\left\|\varphi\right\|_2=\left\|b\right\|_2.
\]
\end{corollary}
\begin{proof}
By Proposition \ref{prop:subdiff-News}, any subgradient $\varphi$ has components $\varphi_i \in [g_i-v_i, b_i]$. By assumption \eqref{eq:assumptions-News}, $b_i \ge v_i-g_i > 0$. This implies $|g_i-v_i| = v_i-g_i \le b_i = |b_i|$. Therefore, the component with the largest absolute value in the interval $[g_i-v_i, b_i]$ is $b_i$. The subgradient vector with the maximum $\ell_2$-norm is $\varphi = b$, and its norm is $\left\|b\right\|_2$.
\end{proof}

\subsubsection{Application of SBR-SAA and MISOCP Reformulation}

We apply our SBR-SAA model \eqref{eq:SBR-SAA_sec} to the newsvendor problem. Let $\Xi_{\mathrm{reg}}=\{\zeta_{1},\ldots,\zeta_{m}\} \subset \hat{\Xi}_n$ be a selected subset of the empirical sample, with weights $r_j \ge 0, \sum r_j = 1$. The problem is:
\begin{equation}\label{eq:model-penal-News}
\min_{x\in\mathcal{X}}\ \frac{1}{n}\sum_{i=1}^n F(x,\widehat\xi_i)\;+\;\varepsilon \left(\sum_{j=1}^m r_j\,\left(\max_{\varphi_j \in \partial_\xi F(x,\zeta_j)}\left\|\varphi_j\right\|_2\right)^{2} \right)^{\frac{1}{2}}.
\end{equation}
We select the subgradient $\varphi_j$ from the subdifferential $\partial_\xi F(x,\zeta_j)$ that has the largest $\ell_2$-norm, consistent with a robust approach. As shown in Corollary \ref{cor:lipschitz-News}, the components of this worst-case subgradient are $\varphi_{j,i} = b_i$ if $x_i \le \zeta_{j,i}$ and $\varphi_{j,i} = g_i-v_i$ if $x_i > \zeta_{j,i}$.

\begin{theorem}[Exact MISOCP Reformulation]\label{thm:MISOCP-News}
Under assumptions \eqref{eq:assumptions-News}, the SBR-SAA problem \eqref{eq:model-penal-News} is equivalent to the following Mixed Integer Second Order Conic Program (MISOCP):
\begin{align}
\min_{x,w,y,z,t,s}\quad & (c-g)^\top x\;+\;\frac{1}{n}\sum_{i=1}^n \left( (g-v-b)^\top w_i + b^\top \widehat\xi_i \right) \;+\;\varepsilon \cdot s \label{obj:MISOCP-News}\\
\text{s.t.}\quad 
& 0\le x\le a, \label{cotas-x-News}\\
& w_i \le x,\ \ w_i \le \widehat\xi_i & (\forall i \in [n]), \label{min-lin-News}\\
& x_k \ge \zeta_{j,k}+\delta_k - M_k(1-y_{j,k}) & (\forall k \in [d], \forall j \in [m]), \label{gate-y1-News}\\
& x_k \le \zeta_{j,k}+M_k y_{j,k} &d (\forall k \in [d], \forall j \in [m]), \label{gate-y2-News}\\
& z_{j,k}= b_k + (g_k-v_k-b_k)\,y_{j,k} & (\forall k \in [d], \forall j \in [m]), \label{z-linear-News}\\
& r_{j}^{\frac{1}{2}}\left\|z_j\right\|_2 \le t_j & (\forall j \in [m]), \label{cone-norm-News}\\
& \|t\|_{2} \leq s & \\
& y_{j,k}\in\{0,1\} & (\forall k \in [d], \forall j \in [m]), \label{bin-y-News}
\end{align}
where $\delta_k > 0$ is a small tolerance (e.g., machine epsilon) and $M_k$ is a sufficiently large constant (e.g., $M_k \ge a_k$).
\end{theorem}
\begin{proof}
First, the objective \eqref{obj:MISOCP-News} substitutes the compact form \eqref{eq:F-compacta-News} for $F(x,\widehat\xi_i)$ and includes the regularization term $\varepsilon \left(\sum_{j=1}^m r_j\,\left(\max_{\varphi_j \in \partial_\xi F(x,\zeta_j)}\left\|\varphi_j\right\|_2\right)^{2} \right)^{\frac{1}{2}}$, which can be expressed as $\varepsilon \left(\sum_{j=1}^m t_{j}^{2} \right)^{\frac{1}{2}}$, this is equivalent to $\varepsilon\|t\|_{2}$.
Constraints \eqref{min-lin-News} linearize the $\min\{x, \widehat\xi_i\}$ term. Since its coefficient in the objective, $(g-v-b)$, is component-wise negative, $w_i$ will be driven to its upper bound, $w_i = \min\{x, \widehat\xi_i\}$, at optimality.
Constraints \eqref{gate-y1-News} and \eqref{gate-y2-News} use the binary variable $y_{j,k}$ to encode the relationship between $x_k$ and $\zeta_{j,k}$:
\begin{itemize}
    \item If $y_{j,k}=1$: \eqref{gate-y1-News} becomes $x_k \ge \zeta_{j,k}+\delta_k$ (i.e., $x_k > \zeta_{j,k}$); \eqref{gate-y2-News} becomes $x_k \le \zeta_{j,k}+M_k$ (non-binding).
    \item If $y_{j,k}=0$: \eqref{gate-y1-News} becomes $x_k \ge \zeta_{j,k}+\delta_k - M_k$ (non-binding); \eqref{gate-y2-News} becomes $x_k \le \zeta_{j,k}$.
\end{itemize}
The small $\delta_k > 0$ breaks the tie at $x_k=\zeta_{j,k}$ and forces $y_{j,k}=0$, selecting the subgradient component $b_k$ (the one with larger magnitude, as $b_k \ge v_k-g_k = |g_k-v_k|$).
Constraint \eqref{z-linear-News} then constructs the worst-case subgradient component $z_{j,k}$ based on $y_{j,k}$:
\begin{itemize}
    \item If $y_{j,k}=1$ ($x_k > \zeta_{j,k}$), then $z_{j,k} = b_k + (g_k-v_k-b_k) = g_k-v_k$.
    \item If $y_{j,k}=0$ ($x_k \le \zeta_{j,k}$), then $z_{j,k} = b_k$.
\end{itemize}
This exactly matches the components of the max-norm subgradient from $\partial_\xi F(x,\zeta_j)$.
Finally, \eqref{cone-norm-News} is the standard second-order cone constraint, $t_j \ge \left\|z_j\right\|_2 r_{j}^{\frac{1}{2}}$, which, by optimality, will be tight.
\end{proof}

\subsubsection{Analysis of Standard WDRO Approaches ($p=1$ and $p=2$)}\label{Lipproblem}

This problem setting exemplifies a scenario where standard WDRO approaches are ill-suited, motivating our SBR-SAA alternative.
Let $\widehat{\mathbb{P}}_n=\tfrac{1}{n}\sum_{j=1}^n \delta_{\widehat\xi_j}$ be the empirical distribution and
\(
\mathcal{B}_{\varepsilon}^{p}(\widehat{\mathbb{P}}_n):=\Big\{\mathbb{Q}:\ W_p(\mathbb{Q},\widehat{\mathbb{P}}_n)\le \varepsilon\Big\},
\)
be the $p$-Wasserstein ball with the $\ell_2$-norm as the ground cost. The WDRO problem is
\begin{equation}\label{eq:WDRO-News}
\min_{x\in\mathcal{X}}\ \sup_{\mathbb{Q}\in\mathcal{B}_{\varepsilon}^{p}(\widehat{\mathbb{P}}_n)}\ \mathbb{E}_{\xi\sim\mathbb{Q}}[F(x,\xi)].
\end{equation}

\paragraph{Case $p=1$:} The 1-WDRO formulation collapses to the SAA solution.
\begin{theorem}[1-WDRO Equivalence to SAA]\label{thm:WDRO-p1-News}
Under assumptions \eqref{eq:assumptions-News}, the $1$-WDRO problem \eqref{eq:WDRO-News} with $p=1$ is equivalent to
\begin{equation}\label{eq:WDRO-p1-final-News}
\min_{x\in\mathcal{X}}\ \frac{1}{n}\sum_{j=1}^n F(x,\widehat\xi_j)\ +\ \varepsilon\,\left\|b\right\|_2.
\end{equation}
Consequently, the optimal decision $\hat{x}$ of the $1$-WDRO problem is identical to the optimal decision of the SAA problem.
\end{theorem}
\begin{proof}
By Kantorovich-Rubinstein duality and Corollary \ref{cor:lipschitz-News}, the inner supremum of the WDRO problem is
\[
\sup_{\mathbb{Q}\in\mathcal{B}_{\varepsilon}^{1}(\widehat{\mathbb{P}}_n)} \mathbb{E}_{\xi\sim\mathbb{Q}}[F(x,\xi)] = \mathbb{E}_{\xi\sim\widehat{\mathbb{P}}_n}[F(x,\xi)] + \varepsilon \cdot \gamma_{x,F} = \frac{1}{n}\sum_{j=1}^n F(x,\widehat\xi_j) + \varepsilon \left\|b\right\|_2.
\]
Since the regularization term $\varepsilon \left\|b\right\|_2$ is a constant that does not depend on $x$, minimizing this objective yields the same solution as minimizing the SAA objective alone.
\end{proof}

\paragraph{Case $p=2$:} The multi-product newsvendor problem is a classic two-stage stochastic linear program. As established by \cite{HanasusantoGraniKuhn2018}, two-stage stochastic linear programs under 2-WDRO (with the $\ell_2$-norm ground cost) are generally NP-hard. Approximations exist, such as the copositive optimization approach in \cite{HanasusantoGraniKuhn2018} or the linear programming upper bound in \cite{WangShanshanDelageErick2024}, but these remain computationally demanding or provide no guarantee on the gap between the approximation and the true 2-WDRO solution.

Given that 1-WDRO is equivalent to SAA and 2-WDRO is computationally intractable, this problem serves as a perfect testbed for our SBR-SAA method as a tractable and effective alternative.

\subsubsection{Numerical Experiment: Design and Setup} \label{subsec:DesignAndSetupNewsvendorMulti}

We test our SBR-SAA method against SAA (which we know is equivalent to 1-WDRO).

\paragraph{Data Generation.}
We set $d=5$ products. The economic parameters are set as:
$c = (4, 5.5, 6, 8, 9)^\top$, $v = (11, 13, 14, 17, 19)^\top$, $g = (1, 1.5, 1.5, 2, 2.5)^\top$, $b = (12, 13.5, 14.5, 17, 18.5)^\top$, and $a = (12, 18, 24, 30, 36)^\top$.
These parameters satisfy the economic assumptions $v>g$ and $b \ge v-g$, representing a context where shortage costs are significant. The bounds $a$ are set to be approximately $1.5$ times the marginal means, ensuring the capacity constraints are relevant but not uniformly binding.
Demands $\xi$ are drawn from a multivariate Lognormal distribution, coupling Lognormal marginals with a Gaussian copula. The marginals are defined by means $\mu = (8, 12, 16, 20, 24)^\top$ and coefficients of variation $\mathrm{CV} = (0.4, 0.5, 0.5, 0.6, 0.4)^\top$. The copula uses a common correlation $\rho=0.3$. This design is standard for modeling non-negative, skewed demands with moderate positive correlation.
We use $n=50$ samples for training and evaluate performance on $10^5$ out-of-sample scenarios drawn from the true distribution $\mathbb{P}$.

\paragraph{SBR-SAA Scenario Selection.}
For a given sample $\hat{\Xi}_{n}$, we select the $m$ scenarios $\Xi_{\mathrm{reg}} = \{\zeta_1, \dots, \zeta_m\}$ and weights $\{r_j\}$ to emulate a WDRO approach. This is done in two steps:
\begin{enumerate}
\item  We apply $k$-medoids clustering (with $k=m$) to $\hat{\Xi}_{n}$ to find $m$ representative medoids, which become our scenarios $\{\zeta_j\}$.
\item  We find the weights $\{r_j\}$ that minimize the $1$-Wasserstein distance between the empirical measure $\hat{\mathbb{P}}_n$ and the new discrete measure $\mathbb{P}_{r,m} := \sum_{j=1}^m r_j \delta_{\zeta_j}$. That is, we solve $\min_{r \ge 0, \sum r_j = 1} W_1(\hat{\mathbb{P}}_n, \mathbb{P}_{r,m})$.
\end{enumerate}
This principled approach creates a compressed $m$-point distribution that is as close as possible to the full $n$-point empirical measure. We test two versions: 5-SBR-SAA ($m=5$) and 15-SBR-SAA ($m=15$).

\paragraph{Performance Metrics.}
We aim to show that SBR-SAA produces solutions $\hat{x}_{m,n}$ that outperform the SAA solution $\hat{x}_{\mathrm{SAA}}$ out-of-sample. "Better" performance means both a lower expected cost $\mathbb{E}_{\xi\sim\mathbb{P}}[F(\hat{x},\xi)]$ and a lower risk, which we measure as the tail risk premium $\mathrm{CVaR}_{5\%,\xi\sim\mathbb{P}}[F(\hat{x},\xi)] - \mathbb{E}_{\xi\sim\mathbb{P}}[F(\hat{x},\xi)]$.
For a given sample $\hat{\Xi}_{n}$, we define the out-of-sample performance curve $\hat{\mathcal{C}}_{m,n}$ by solving the SBR-SAA problem for $\varepsilon$ over a logarithmic grid $\mathcal{E}$:
\begin{equation}
\hat{\mathcal{C}}_{m,n}=\left\{ \left(
    \begin{multlined}[c]
        \mathrm{CVaR}_{5\%,\xi\sim\mathbb{P}}\left[F(\hat{x}_{m,n},\xi)\right] \\
        - \mathbb{E}_{\xi\sim \mathbb{P}}\left[F(\hat{x}_{m,n},\xi)\right]
    \end{multlined},
    \mathbb{E}_{\xi\sim \mathbb{P}}\left[F(\hat{x}_{m,n},\xi)\right]
    \right)
    \: : \:
    \begin{array}{@{}l@{}}
        \hat{x}_{m,n} \text{ is a solution from} \\
        \text{SBR-SAA for } \varepsilon\in\mathcal{E}
    \end{array}
\right\},
\label{eq:performance_set_1_News}
\end{equation}
The SAA solution corresponds to $\varepsilon=0$. We expect to find $\varepsilon > 0$ that yield solutions dominating the SAA point (i.e., are to the "south-west" of it).

\subsubsection{Results and Discussion}

Figure \ref{fig:Newsvendor_Curve_Single} plots the performance curve $\hat{\mathcal{C}}_{m,n}$ for a single, representative training sample. The SAA solution ($\varepsilon=0$) is the rightmost point on each curve. As $\varepsilon$ increases (moving left along the curves), the SBR-SAA solutions achieve both a lower out-of-sample expected cost and a lower tail risk premium. This demonstrates that for a typical sample, SBR-SAA finds solutions that are strictly dominant over SAA. Furthermore, using more representative scenarios ($m=15$) yields a superior curve (it dominates the $m=5$ curve), suggesting a performance benefit to a better approximation of the empirical measure.

\begin{figure}[t]
    \centering
    \includegraphics[width=0.6\linewidth]{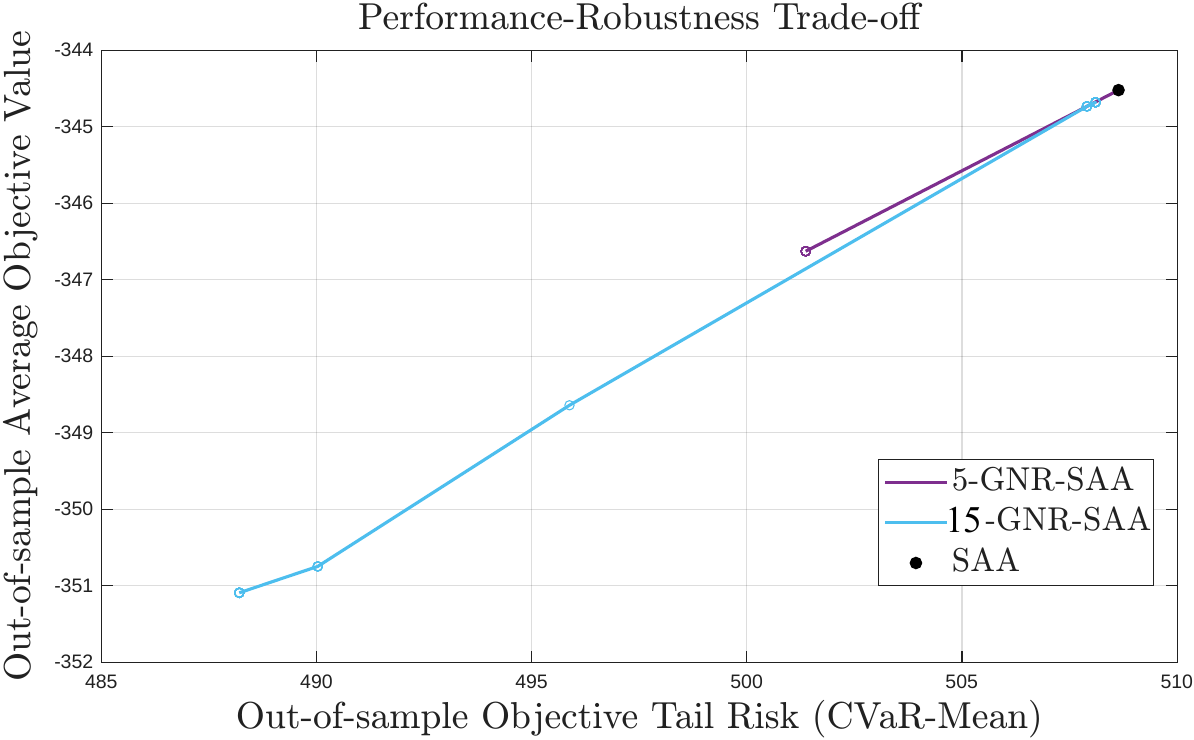} 
    \caption{Out-of-sample performance curves $\hat{\mathcal{C}}_{m,n}$ for $n=50$ and $m=5$ (5-SBR-SAA) and $m=15$ (15-SBR-SAA), based on a single training sample. The y-axis represents the out-of-sample expected cost, while the x-axis represents the out-of-sample tail risk premium (CVaR$_{5\%}$ - Mean).}
    \label{fig:Newsvendor_Curve_Single}
\end{figure}

To ensure this result is not an artifact of a single sample, we repeat the experiment 200 times, each with a new training sample $\hat{\Xi}_n$ of size $n=50$. For each replication, we find the $\varepsilon_{m,n}^*$ that minimizes the out-of-sample expected cost $\mathbb{E}_{\xi\sim\mathbb{P}}[F(\hat{x}_{m,n},\xi)]$ along the curve $\hat{\mathcal{C}}_{m,n}$. We then plot the performance of this "best" SBR-SAA solution relative to the SAA solution for that same sample.
We define the relative performance point $\hat{P}_{m,n}$ as:
\[
\hat{P}_{m,n} = (\hat{P}_{m,n,x}, \hat{P}_{m,n,y})
\]
where
\begin{align*}
\hat{P}_{m,n,x} &:= \left(\mathrm{CVaR}_{\alpha}(F(\hat{x}_{m,n},\xi)) - \mathbb{E}_{\xi}(F(\hat{x}_{m,n},\xi))\right) - \left(\mathrm{CVaR}_{\alpha}(F(\hat{x}_{\mathrm{SAA}},\xi)) - \mathbb{E}_{\xi}(F(\hat{x}_{\mathrm{SAA}},\xi))\right) \\
\hat{P}_{m,n,y} &:= \mathbb{E}_{\xi\sim\mathbb{P}}[F(\hat{x}_{m,n},\xi)] - \mathbb{E}_{\xi\sim\mathbb{P}}[F(\hat{x}_{\mathrm{SAA}},\xi)]
\end{align*}
(with $\alpha=5\%$). Negative values for $\hat{P}_{m,n,x}$ or $\hat{P}_{m,n,y}$ indicate that SBR-SAA improved upon SAA in that dimension (tail risk or expected cost, respectively).

Figure \ref{fig:Newsvendor_Scatter_200} shows the scatter plot of these 200 $\hat{P}_{m,n}$ points for $m=5$ and $m=15$. The results are conclusive. In both cases, the vast majority of points lie in the bottom-left quadrant (the "south-west" quadrant), indicating that SBR-SAA simultaneously improves both the expected cost and the tail risk premium compared to SAA across a wide range of samples. The fitted regression lines, with their clear negative slopes, confirm this strong trend. Comparing (a) and (b), the cluster of points for $m=15$ is visibly lower and further to the left than for $m=5$, confirming that a better $m$-point approximation of the empirical measure leads to more robust and superior out-of-sample solutions.

\begin{figure}[t]
    \centering
    \begin{tabular}{cc}
    \includegraphics[width=0.49\linewidth]{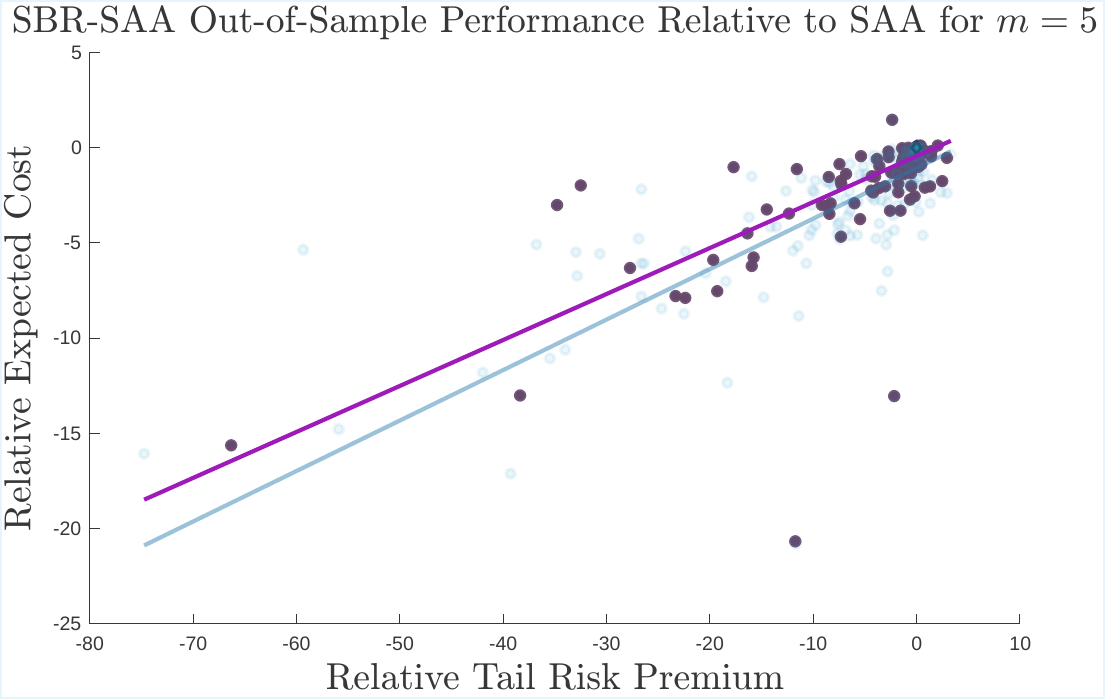} & 
    \includegraphics[width=0.49\linewidth]{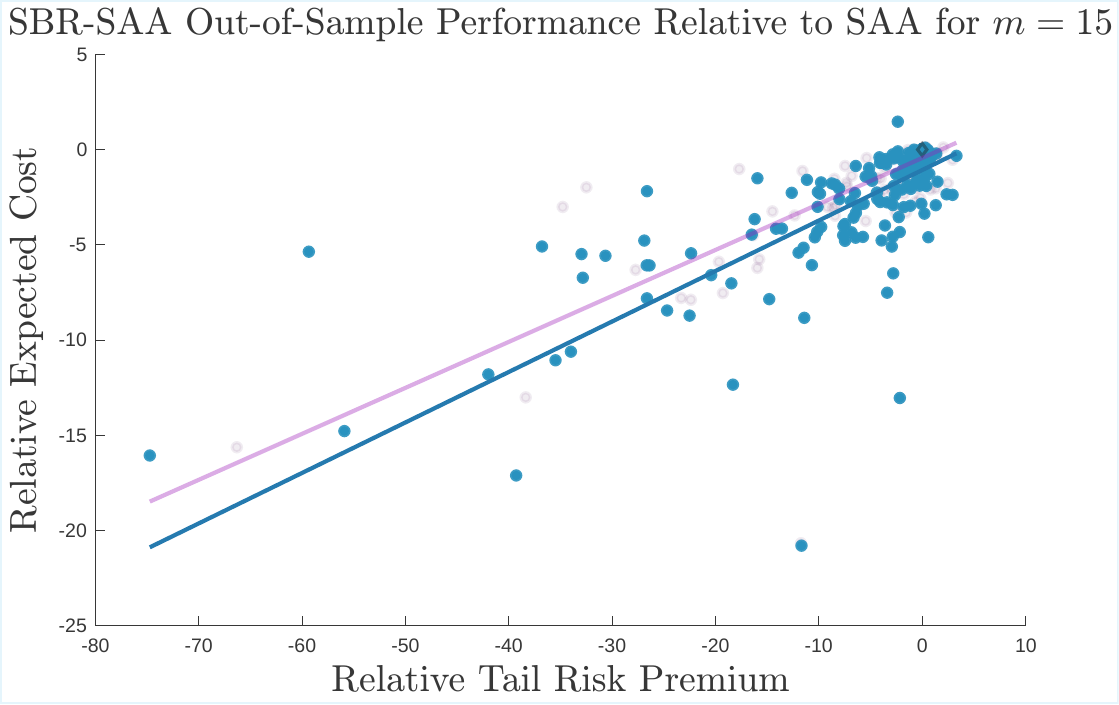}\\ 
     (a) 5-SBR-SAA ($m=5$) & (b) 15-SBR-SAA ($m=15$)
    \end{tabular}
    \caption{Out-of-sample performance of SBR-SAA relative to SAA across 200 replications ($n=50$). Each point represents the relative performance of the SBR-SAA solution that minimized the out-of-sample expected cost for a given sample. Negative values on both axes indicate an improvement over SAA. (a) $m=5$ scenarios. (b) $m=15$ scenarios.}
    \label{fig:Newsvendor_Scatter_200}
\end{figure}

%


\subsection{Mean-Risk Portfolio Optimization}
\label{subsec:Portfolio}

To demonstrate the second key use of our framework, its role as a tool for targeted robustness against adverse scenarios, we consider a mean-risk portfolio optimization problem with real financial data. This setting is natural for our purposes: it features an inherent trade-off between expected performance and tail risk, and it allows us to test how effectively scenario-based regularization leverages historically adverse episodes compared to standard SAA and Wasserstein DRO benchmarks.


\subsubsection{Problem Formulation}
\label{subsec:PortfolioProblem}

We consider the classical mean–risk portfolio optimization problem. Let $w \in \mathbb{R}^d$ be the vector of weights allocated across $d$ financial assets. The set of feasible portfolios $\mathcal{W}$ is the standard simplex:
\[
\mathcal{W}=\{w\in\mathbb{R}^d: \mathbf{1}^\top w=1,\; w\ge 0\}.
\]
The objective is to find a portfolio that minimizes a weighted sum of the expected loss (negative return) and the Conditional Value-at-Risk (CVaR) of this loss. Formally, we address the problem:
\begin{equation}\label{eqn:MeanRiskTrue}
  \min_{w\in\mathcal{W}} \mathbb{E}_{\xi\sim\mathbb{P}}[-w^\top \xi]+\rho\cdot \mathrm{CVaR}_{\alpha,\xi\sim\mathbb{P}}(-w^\top \xi),
\end{equation}
\sloppy where $\xi$ is the random vector of asset returns, $\rho > 0$ is a risk-aversion parameter, and $\mathrm{CVaR}_{\alpha,\xi\sim\mathbb{P}}(-w^\top \xi)$ is the CVaR at the $\alpha$ confidence level, representing the average of the $\alpha\%$ worst portfolio losses.

Following the seminal work of \cite{Rockafellar2000}, the CVaR term can be reformulated as the minimum of an expected value function by introducing an auxiliary scalar variable $\tau$. This allows problem \eqref{eqn:MeanRiskTrue} to be cast into the general stochastic optimization framework of \eqref{eqn:SP_intro}:
\begin{equation}\label{eqn:StocasticFormMeanRiskTrue}
  \min_{w\in\mathcal{W},\tau\in\mathbb{R}} \mathbb{E}_{\xi\sim\mathbb{P}}\left[-w^\top \xi+\rho\left(\tau+\frac{1}{\alpha}\Big(-w^\top \xi-\tau\Big)_+\right)\right],
\end{equation}
where $(\cdot)_+ = \max\{0, \cdot\}$. Within our framework, the decision variable is the concatenated vector $x=(w,\tau)$, and the objective function $F(x,\xi)$ is:
\begin{equation}\label{eq:F_portfolio}
F(w,\tau,\xi)=-w^\top \xi+\rho\left(\tau+\frac{1}{\alpha}\Big(-w^\top \xi-\tau\Big)_+\right).
\end{equation}

\subsubsection{Subgradient characterization for the mean–CVaR loss}
\label{subsec:SubgradientsPortfolio}

We characterize here the (sub)gradients of $F(w,\tau,\zeta)$ with respect to the return vector $\zeta$, as these are the primitives driving our scenario–based regularizers. Let \(
\phi(z,\tau)\;=\;\tau+\frac{1}{\alpha}(z-\tau)_+\), so that \(F(w,\tau,\zeta)\;=\;-w^\top \zeta+\rho\,\phi(-w^\top\zeta,\tau).
\)
The subdifferential of $\phi$ with respect to $z$ is
\[
\partial_z \phi(z,\tau)=
\begin{cases}
 \dfrac{1}{\alpha}, & z>\tau,\\[0.6ex]
 0, & z<\tau,\\[0.6ex]
 \left[0,\dfrac{1}{\alpha}\right], & z=\tau.
\end{cases}
\]

\begin{proposition}[Subgradient structure of $F(w,\tau,\cdot)$]\label{prop:SubgradPortfolio}
For any $(w,\tau)$ and any scenario $\zeta\in\mathbb{R}^d$,
\[
\partial_\zeta F(w,\tau,\zeta)\;=\;-\,w\Bigl[1+\rho\,\partial_z \phi(-w^\top\zeta,\tau)\Bigr].
\]
Equivalently, with the Euclidean norm $\|\cdot\|_2$,
\begin{enumerate}
    \item[(i)] If  $-w^\top\zeta>\tau$ then  $\nabla_\zeta F(w,\tau,\zeta) \;=\; -\,w\!\left(1+\tfrac{\rho}{\alpha}\right)$, and  $\bigl\|\nabla_\zeta F(w,\tau,\zeta)\bigr\|_2 \;=\; \Bigl(1+\tfrac{\rho}{\alpha}\Bigr)\|w\|_2$.
\item[(ii)] If $-w^\top\zeta<\tau$, then  $\nabla_\zeta F(w,\tau,\zeta) \;=\; -\,w$, and $ \bigl\|\nabla_\zeta F(w,\tau,\zeta)\bigr\|_2 \;=\; \|w\|_2$.
\item[(iii)] If $-w^\top\zeta=\tau$, then $\partial_\zeta F(w,\tau,\zeta)=\bigl\{-\,w\,(1+\rho\eta):\ \eta\in[0,\tfrac{1}{\alpha}]\bigr\}$, and thus  $\bigl\|\partial_\zeta F(w,\tau,\zeta)\bigr\|_2\ \in\ \bigl[\|w\|_2,\ (1+\tfrac{\rho}{\alpha})\|w\|_2\bigr]$.
\end{enumerate}
\end{proposition}

\begin{proof}
Set $z(\zeta)=-w^\top\zeta$. Then $\nabla_\zeta z(\zeta)=-w$. By the chain rule for subdifferentials,
\[
\partial_\zeta F(w,\tau,\zeta)
=\nabla_\zeta(-w^\top\zeta)+\rho\,\partial_\zeta\phi\bigl(z(\zeta),\tau\bigr)
= -w+\rho\bigl(\partial_z\phi(z(\zeta),\tau)\bigr)\,(-w),
\]
which yields the stated form and the three regimes according to the sign of $z(\zeta)-\tau$. In the kink case $z(\zeta)=\tau$, the set $\partial_z\phi(\tau,\tau)=[0,1/\alpha]$ implies
$\partial_\zeta F(w,\tau,\zeta)=\{-w(1+\rho\eta):\eta\in[0,1/\alpha]\}$, and the Euclidean norm ranges from $\|w\|_2$ to $(1+\rho/\alpha)\|w\|_2$. \qedhere
\end{proof}

The portfolio loss $-w^\top\zeta$ partitions scenarios into a \emph{non–tail region} (case (ii)) with sensitivity $\|w\|_2$ and a \emph{tail region} (case (i)) with amplified sensitivity $\bigl(1+\rho/\alpha\bigr)\|w\|_2$. At the quantile boundary (case (iii)) the subgradient magnitude can take any value in the interval $\bigl[\|w\|_2,\,(1+\rho/\alpha)\|w\|_2\bigr]$. In line with the robustness motivation in Section~\ref{subsec:FormulationGNR}, namely, to penalize the worst local sensitivity, we adopt the convention of selecting, at the threshold, the maximal–norm representative of the subdifferential, i.e., the element with magnitude $(1+\rho/\alpha)\|w\|_2$.

\subsubsection{Scenario-based regularization and tractable reformulations}
\label{subsec:ReformulationsPortfolio}

Let $\{\zeta_j\}_{j=1}^m$ denote the chosen adverse scenarios (constructed from historical data as explained below), and let $r_j\ge 0$ with $\sum_{j=1}^m r_j = 1$ be fixed weights. For this problem we employ the Euclidean norm $\|\cdot\|_2$ on gradients. The general SBR-SAA formulation introduced earlier considers a regularizer of the form
\(
\widehat{R}_m(x)
=
\left(\sum_{j=1}^m r_j \big\|\nabla_\xi F(x,\zeta_j)\big\|_2^2\right)^{1/2},
\)
which corresponds to the gradient-norm aggregation suggested by the asymptotic expansion in Lemma~\ref{Lemma:taylorExpansion_intro_revised} for intermediate Wasserstein orders. Instantiated in our portfolio setting, this yields the model
\begin{equation}\label{eq:AdversoSAAMeanRisk_quad}
\min_{w\in\mathcal{W},\,\tau\in\mathbb{R}}
\frac{1}{n}\sum_{i=1}^{n} F(w,\tau,\widehat{\xi}_i)
+
\varepsilon\,
\left(\sum_{j=1}^{m} r_j \big\|\nabla_\xi F(w,\tau,\zeta_j)\big\|_2^2\right)^{1/2},
\end{equation}
which we refer to in this section as the \emph{quadratic-aggregation} SBR-SAA portfolio model.

In addition to \eqref{eq:AdversoSAAMeanRisk_quad}, we also consider a closely related variant obtained by replacing the quadratic aggregation of gradient norms with a linear aggregation,
\begin{equation}\label{eq:AdversoSAAMeanRisk}
\min_{w\in\mathcal{W},\,\tau\in\mathbb{R}}
\frac{1}{n}\sum_{i=1}^{n} F(w,\tau,\widehat{\xi}_i)
+
\varepsilon\,
\sum_{j=1}^{m} r_j \big\|\nabla_\xi F(w,\tau,\zeta_j)\big\|_2.
\end{equation}
We refer to this linear-aggregation variant as SBR-SAA-L. It preserves the same modeling principles, scenario-based penalization of sensitivity to adverse returns, while providing a uniform aggregation of the selected scenarios (cf.\ Lemma~\ref{Lemma:taylorExpansion_intro_revised}). From an implementation perspective, \eqref{eq:AdversoSAAMeanRisk} is also convenient, as the regularizer is directly representable by second-order cone constraints coupled with binary variables.

\begin{proposition} \label{prop:RefAdversoSAAMeanRisk_sec}
Let $K=\max\{\max_{j\in[m]}\|\zeta_{j}\|_{\infty},\max_{i\in[n]}\|\widehat{\xi}_{i}\|_{\infty}\}$, let $C>0$, and define
\[
M_{1}:=2(K+C)+\delta,
\qquad
M_{2}:=1+K,
\qquad
M_{3}:=1+\frac{\rho}{\alpha},
\]
for a small $\delta > 0$. Then problem \eqref{eq:AdversoSAAMeanRisk} (SBR-SAA-L) can be reformulated as the MISOCP
\begin{equation}\label{eqn:AdvFormulationV3}
\left\{
\begin{array}{cll}
\displaystyle \min_{w,\tau,\gamma,z}
&
\displaystyle
\frac{1}{n}\sum_{i=1}^{n} F(w,\tau,\widehat{\xi}_i)
+
\varepsilon\, \sum_{j=1}^{m} r_j \gamma_{j} &
\\[0.5ex]
\text{subject to}
&
w\in\mathcal{W},\ \tau\in\mathbb{R},\ \gamma\in\mathbb{R}^{m},\ z\in\{0,1\}^m, &
\\[0.5ex]
&
-w^\top \zeta_j - \tau \leq M_{1}(1-z_{j}),
& j\in[m],
\\
&
w^\top \zeta_j + \tau - \delta \leq M_{1}z_{j},
& j\in[m],
\\
&
\|w\|_2 - \gamma_{j} \leq M_2 (1-z_{j}),
& j\in[m],
\\
&
\|w\|_2 \left(1 + \tfrac{\rho}{\alpha}\right) - \gamma_{j} \leq M_3 z_j,
& j\in[m].
\end{array}
\right.
\end{equation}
and the problem \eqref{eq:AdversoSAAMeanRisk_quad} (SBR-SAA) can be reformulated as the MISOCP
\begin{equation}\label{eqn:AdvFormulationV3Quad}
\left\{
\begin{array}{cll}
\displaystyle \min_{w,\tau,\gamma,z,s}
&
\displaystyle
\frac{1}{n}\sum_{i=1}^{n} F(w,\tau,\widehat{\xi}_i)
+
\varepsilon\, \cdot s &
\\[0.5ex]
\text{subject to}
&
w\in\mathcal{W},\ \tau\in\mathbb{R},\ \gamma\in\mathbb{R}^{m},\ z\in\{0,1\}^m, &
\\[0.5ex]
&
-w^\top \zeta_j - \tau \leq M_{1}(1-z_{j}),
& j\in[m],
\\
&
w^\top \zeta_j + \tau - \delta \leq M_{1}z_{j},
& j\in[m],
\\
&
r_{j}^{\frac{1}{2}}\|w\|_2 - \gamma_{j} \leq M_2 (1-z_{j}),
& j\in[m],
\\
&
r_{j}^{\frac{1}{2}}\|w\|_2 \left(1 + \tfrac{\rho}{\alpha}\right) - \gamma_{j} \leq M_3 z_j,
& j\in[m],\\
& \|\gamma\|_2 \leq s.& 
\end{array}
\right.
\end{equation}
\end{proposition}

\noindent The proof, which follows standard big-$M$ logic linearization arguments for piecewise definitions of the gradient norm, is provided in Appendix~\ref{app:proof_prop_RefAdversoSAAMeanRisk}. The approximation arises from using a small tolerance $\delta>0$ to handle the strict inequalities that appear in the logical conditions of the gradient calculation, a standard technique in mixed-integer programming.

We next investigate the empirical behavior of both the quadratic-aggregation SBR-SAA model \eqref{eq:AdversoSAAMeanRisk_quad} and its linear-aggregation variant SBR-SAA-L \eqref{eq:AdversoSAAMeanRisk} in two case studies. In each case, we compare against SAA, 1-WDRO, and 2-WDRO, all calibrated on the same data.

\subsubsection{Case Study 1: Performance in the Post-Pandemic Era (2020--2022)}
\label{subsec:CaseStudy1}

\paragraph{Experimental Setup}

Our first numerical experiment is based on real historical data. The asset universe consists of 23 of the largest companies by market capitalization in the S\&P 500 index as of 2022, with a substantial exposure to the technology sector. The specific assets are listed in Table~\ref{tab:asset_list_2022_Port}. Hence, the portfolio dimension is $d=23$. The random vector $\xi$ represents the daily returns of these assets.

\begin{table}[t]
\centering
\caption{List of assets used in Case Study 1 (2022 portfolio).}
\label{tab:asset_list_2022_Port}
\begin{tabular}{lll}
\toprule
AAPL - Apple & INTC - Intel & PG - Procter \& Gamble \\
AMZN - Amazon & JNJ - Johnson \& Johnson & T - AT\&T \\
BAC - Bank of America & JPM - JPMorgan Chase & UNH - UnitedHealth Group \\
BRK-A - Berkshire Hathaway & KO - The Coca-Cola Co. & VZ - Verizon \\
CVX - Chevron & MA - Mastercard & WFC - Wells Fargo \\
DIS - The Walt Disney Co. & MRK - Merck \& Co. & WMT - Walmart \\
HD - The Home Depot & MSFT - Microsoft & XOM - Exxon Mobil \\
PFE - Pfizer & GOOG - Alphabet Inc. & \\
\bottomrule
\end{tabular}
\end{table}

We structure our data chronologically to simulate a realistic workflow. The market conditions of these three years were notably distinct, making this split particularly insightful:
\begin{itemize}
 	\item \textbf{Adverse Scenarios (Year 2020):} Data from this highly volatile year, marked by the onset of the COVID-19 pandemic, is used exclusively to identify the set of adverse scenarios $\{\zeta_j\}$.
 	\item \textbf{Training Data (Year 2021):} Daily returns from 2021, a year of strong market recovery, form the in-sample dataset $\hat{\Xi}_{n}$ used to train all models.
 	\item \textbf{Out-of-Sample Test Data (Year 2022):} Daily returns from 2022, a year marked by the war in Ukraine, serve as the testing ground to evaluate if robustness learned from the pandemic generalizes to a new, distinct period of market stress.
\end{itemize}

Adverse scenarios are identified from the daily returns of the S\&P 500 index in 2020. We select three negative-return thresholds and define as adverse every day on which the index return falls below the corresponding threshold. For each such day, the vector of returns of the 23 assets becomes an adverse scenario $\zeta_j$. As illustrated in Figure~\ref{fig:RetornosS&P500Ano2020_Port}, the thresholds and resulting numbers of scenarios are:
$-2\%$ (26 days, $m=26$),
$-3.5\%$ (11 days, $m=11$),
and $-5\%$ (5 days, $m=5$).

\begin{figure}[t]
\centering
\includegraphics[scale=0.6]{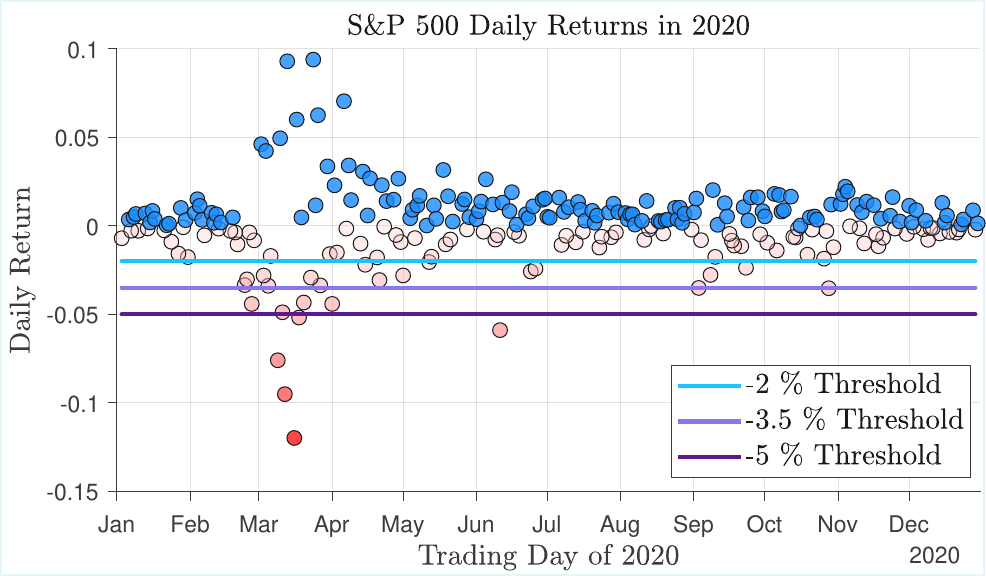}
\caption{Daily returns of the S\&P 500 index in 2020, with thresholds at $-2\%$, $-3.5\%$, and $-5\%$ used to identify adverse market days for Case Study 1.}
\label{fig:RetornosS&P500Ano2020_Port}
\end{figure}

For each threshold, we implement two variants of our scenario-based regularization framework:
(i) the original SBR-SAA model, which employs the quadratic aggregation of gradient norms as in \eqref{eq:AdversoSAAMeanRisk_quad},
and (ii) the linear-aggregation variant, denoted SBR-SAA-L, which replaces the quadratic aggregation with the sum-type regularizer \eqref{eq:AdversoSAAMeanRisk}.
In both cases, the same sets of adverse scenarios are used; only the way in which their associated sensitivities are aggregated differs.
These models are benchmarked against:
the standard SAA estimator,
and two Wasserstein DRO formulations (1-WDRO and 2-WDRO) based on ambiguity sets with Euclidean ground metric, both of which admit tractable SOCP reformulations (see Appendix~\ref{app:wdro_reforms}).

All robust methods are evaluated over a common logarithmic grid
\(
\mathcal{E} = \{ i \times 10^j : i \in \{1,\dots,9\},\ j \in \{-3,-2,-1,0\} \}
\)
of regularization parameters or Wasserstein radii, depending on the method.
Unless stated otherwise, we fix $\alpha = 5\%$ and $\rho = 10$ in the mean-CVaR objective \eqref{eq:F_portfolio} for all approaches.
This ensures that differences in performance are attributable to the robustness mechanism rather than to heterogeneous risk-aversion specifications.

\paragraph{Performance-robustness curves}

To compare the out-of-sample performance of the portfolios generated by each methodology $\mathsf{M}$, we construct performance-robustness curves. We utilize the test dataset comprised of the 2022 daily returns, denoted as $\hat{\Xi}_{out}=\{\hat{\xi}_{1}',\hat{\xi}_{2}',\ldots,\hat{\xi}_{k}'\}$, and let $\hat{\mathbb{P}}_{k}'$ be its corresponding empirical probability distribution. For each solution $(w,\tau)$ obtained from a given method for some $\varepsilon$, we generate a performance point. This point plots the out-of-sample expected cost on the y-axis against the "cost risk premium" (out-of-sample $\mathrm{CVaR}_{5\%}$ minus out-of-sample Mean) on the x-axis. This creates a set of performance points for each methodology, formally defined as:
\begin{equation}
\mathcal{C}_{\mathsf{M}}=\left\{ \left(
  	\begin{multlined}[c]
  	 	\mathrm{CVaR}_{5\%,\xi\sim\hat{\mathbb{P}}_{k}'}\left[F(w,\tau,\xi)\right] \\
  	 	- \mathbb{E}_{\xi\sim \hat{\mathbb{P}}_{k}'}\left[F(w,\tau,\xi)\right]
  	\end{multlined},
  	\mathbb{E}_{\xi\sim \hat{\mathbb{P}}_{k}'}\left[F(w,\tau,\xi)\right]
  	\right)
  	\: : \:
  	\begin{array}{@{}l@{}}
  	 	(w,\tau) \text{ is a solution from} \\
  	 	\text{method }\mathsf{M} \text{ for some } \varepsilon\in\mathcal{E}
  	\end{array}
\right\},
\label{eq:performance_set_1_Port}
\end{equation}
with $\mathsf{M}$ ranging over all implemented variants:
SAA, 1-WDRO, 2-WDRO, SBR-SAA (for each threshold), and SBR-SAA-L (for each threshold).

\paragraph{Results for the original SBR-SAA regularizer}

Figure~\ref{fig:ComparacionCurvasFronteraEffQuad_Port} reports the performance–robustness frontier for the original SBR-SAA model (quadratic aggregation), alongside SAA, 1-WDRO, and 2-WDRO. 
The frontier is the analogous construct to that introduced in Section~\ref{subsec:DesignAndSetupNewsvendorMulti} for the multi-product newsvendor experiment; here it is applied in the portfolio setting exactly as specified in the preceding paragraph. 
This presentation isolates the impact of the original regularizer.

\begin{figure}[t]
\centering
\includegraphics[scale=0.6]{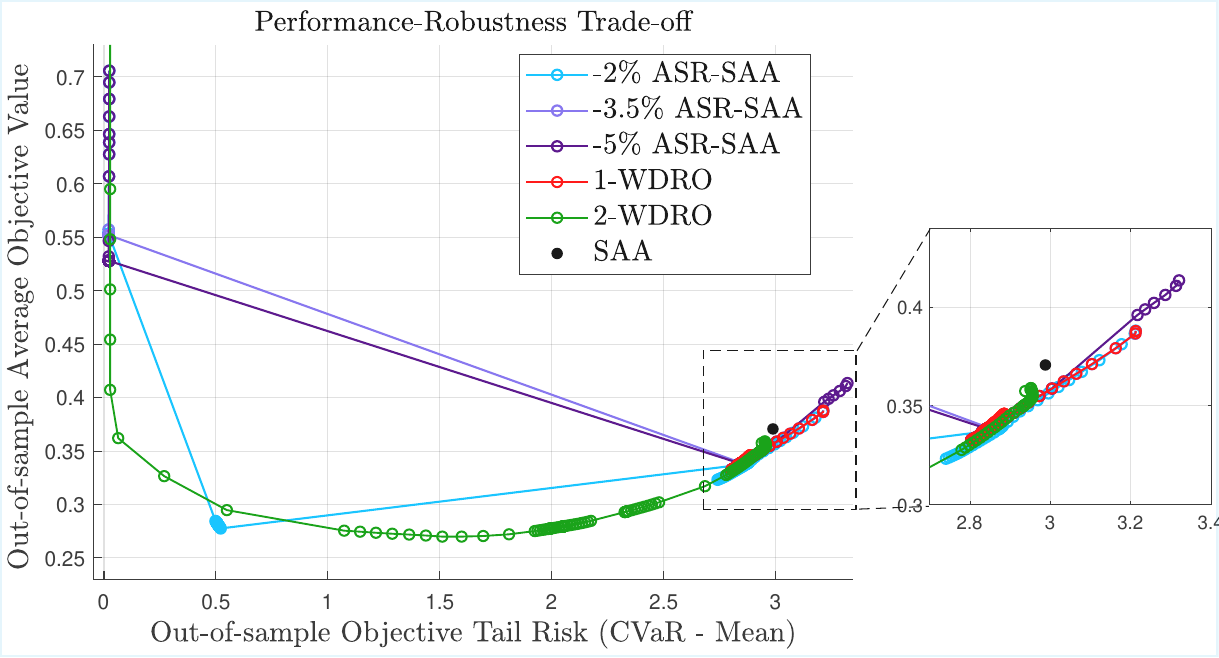}
\caption{Case Study 1 (2020--2022). Performance-robustness frontiers for the original SBR-SAA model with quadratic gradient-norm aggregation, compared against SAA, 1-WDRO, and 2-WDRO. The vertical axis reports the out-of-sample expected cost $\mathbb{E}[F]$, and the horizontal axis the out-of-sample cost tail risk premium $\mathrm{CVaR}_{5\%}[F]-\mathbb{E}[F]$.}
\label{fig:ComparacionCurvasFronteraEffQuad_Port}
\end{figure}

Three observations are worth highlighting.
First, all SBR-SAA frontiers (for the three thresholds) dominate the SAA point, confirming that even the purely theory-driven regularizer yields portfolios with both lower average cost and lower tail risk than the naive empirical approach.
Second, the SBR-SAA curves uniformly improve upon 1-WDRO, indicating that encoding targeted gradient-based robustness around adverse days is more effective here than relying on a Wasserstein ball of order $p=1$.
Third, the comparison with 2-WDRO is nuanced but encouraging:
for the most stringent threshold ($-2\%$), the best SBR-SAA configuration attains an out-of-sample expected cost comparable to the best 2-WDRO configuration, while achieving a strictly smaller cost tail risk premium; for the intermediate thresholds ($-3.5\%$ and $-5\%$), SBR-SAA remains competitive but does not uniformly dominate the 2-WDRO frontier.
These patterns show that, when adverse scenarios are informative, the original SBR-SAA regularizer can perform on par with a strong WDRO benchmark while offering a transparent link between robustness and specific stress episodes.

\paragraph{Results for the linear-aggregation SBR-SAA-L variant}

We now turn to the SBR-SAA-L variant, whose performance-robustness curves are displayed in Figure~\ref{fig:ComparacionCurvasFronteraEff_Port}.
Recall that SBR-SAA-L is a direct extension of our framework that replaces the quadratic aggregation of gradient norms by a linear aggregation; it preserves the same adverse scenarios and conceptual interpretation, while offering a slightly different emphasis on their joint effect.

\begin{figure}[t]
\centering
\begin{tabular}{c}
 \includegraphics[scale=0.6]{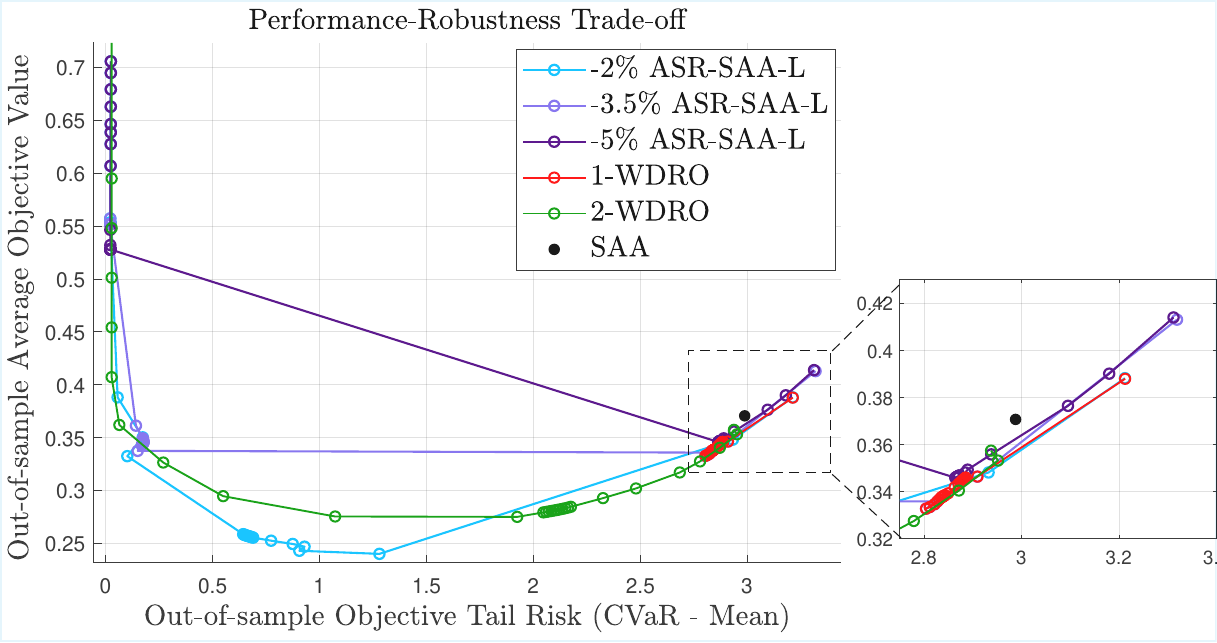}
\end{tabular}
\caption{Case Study 1 (2020--2022). Performance-robustness frontiers for the SBR-SAA-L variant (linear aggregation), compared against SAA, 1-WDRO, and 2-WDRO. Axes as in Figure~\ref{fig:ComparacionCurvasFronteraEffQuad_Port}.}
\label{fig:ComparacionCurvasFronteraEff_Port}
\end{figure}

Here the advantages of the scenario-based approach become more pronounced.
All SBR-SAA-L frontiers clearly dominate SAA and 1-WDRO.
More importantly, the variant based on the $-2\%$ threshold, which leverages the richest set of adverse scenarios while still focusing on genuinely stressed days, produces a frontier that extends furthest towards the lower-left region and is not dominated by 2-WDRO.
Across a wide range of $\varepsilon$, the $-2\%$ SBR-SAA-L solutions simultaneously achieve low expected costs and low tail risk premia, whereas 2-WDRO either matches one criterion at the expense of the other or lies strictly above/right.
The $-3.5\%$ and $-5\%$ SBR-SAA-L variants also yield robust frontiers, but, as expected given the smaller number of scenarios, their improvements are more localized.

Taken together, Case Study 1 conveys a coherent message.
The original SBR-SAA model, directly motivated by the worst-case expectation expansion, already behaves as a competitive robustification scheme.
The SBR-SAA-L extension, which remains fully consistent with our framework, accentuates the benefits when adverse scenarios are rich and relevant, and provides clear empirical evidence that targeted, scenario-based regularization can match or surpass state-of-the-art WDRO models in realistic stressed environments.

\subsubsection{Case Study 2: The 2008 Financial Crisis and Scenario Selection}
\label{subsec:CaseStudy2}

\paragraph{Experimental Setup}

The second case study examines a complementary question: how sensitive is our approach to the quality and representativeness of the adverse scenarios?
Here we deliberately construct a mismatch between the period used to define adverse scenarios and the period in which portfolios are tested, using the 2008 financial crisis as a stress benchmark.

The asset universe reflects the market composition at the start of 2008, with a stronger presence of financial and industrial firms.
The 23 assets are listed in Table~\ref{tab:asset_list_2008_Port}.

\begin{table}[t]
\centering
\caption{List of assets used in Case Study 2 (2008 portfolio).}
\label{tab:asset_list_2008_Port}
\begin{tabular}{lll}
\toprule
XOM – Exxon Mobil Corp. & JNJ – Johnson \& Johnson & VZ – Verizon Communications \\
GE – General Electric & JPM – JPMorgan Chase \& Co. & INTC – Intel Corporation \\
MSFT – Microsoft Corp. & BRK-A – Berkshire Hathaway & WFC – Wells Fargo \& Co. \\
C – Citigroup Inc. & HPQ – HP Inc. & KO – The Coca-Cola Co. \\
T – AT\&T Inc. & CVX – Chevron Corp. & COP – ConocoPhillips \\
BAC – Bank of America Corp. & CSCO – Cisco Systems, Inc. & MO – Altria Group, Inc. \\
PG – The Procter \& Gamble Co. & IBM – IBM Corp. & AIG – American Int'l Group \\
WMT – Walmart Inc. & PFE – Pfizer Inc. & \\
\bottomrule
\end{tabular}
\end{table}

The temporal segmentation is analogous in structure but very different in spirit:
\begin{itemize}
 	\item \textbf{Adverse Scenarios (Year 2006):} Data from this stable, pre-crisis year is used to identify scenarios. As shown in Figure \ref{fig:RetornosS&P500Ano2006_Port}, the daily losses in 2006 were mild, never exceeding -2\%.
 	\item \textbf{Training Data (Year 2007):} Daily returns from this final pre-crisis year are used for training.
 	\item \textbf{Out-of-Sample Test Data (Year 2008):} The performance is evaluated during the severe market crash of 2008.
\end{itemize}

\begin{figure}[t]
\centering
\includegraphics[scale=0.6]{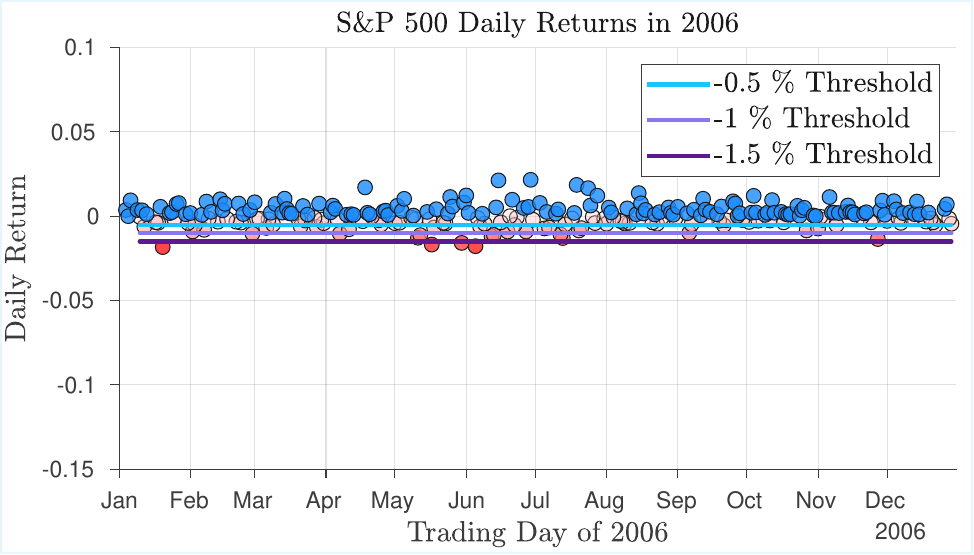}
\caption{Daily returns of the S\&P 500 index in 2006 for Case Study 2. The mild volatility highlights the limited severity of the adverse scenarios used to construct the regularizers.}
\label{fig:RetornosS&P500Ano2006_Port}
\end{figure}

Following the same procedure as in Case Study 1, scenarios are extracted from 2006 using several modest negative-return thresholds on the S\&P 500.
Because losses in 2006 are small, the resulting adverse scenarios are numerous but relatively mild.
These scenarios are then used to build both the quadratic-aggregation SBR-SAA regularizer and the SBR-SAA-L regularizer. Due to the low volatility of this period, we use much smaller thresholds: `-0.5\%` ($m=36$), `-1.0\%` ($m=13$), and `-1.5\%` ($m=4$).
All methods, SAA, 1-WDRO, 2-WDRO, SBR-SAA, and SBR-SAA-L, are calibrated on 2007 and evaluated on 2008, across the same grid $\mathcal{E}$ and with the same $(\alpha,\rho)$ as in the first case study.

\paragraph{Results and Discussion}

To streamline the exposition and avoid redundancy, we report a single representative set of performance–robustness frontiers on the 2008 test data. Specifically, Figure~\ref{fig:ComparacionCurvasFronteraEff2008_Port} displays the curves for the SBR-SAA-L variant alongside SAA, 1-WDRO, and 2-WDRO. The original SBR-SAA model with quadratic aggregation produces a frontier that is numerically and visually indistinguishable under this setup; we therefore omit its separate plot and summarize the comparison in the discussion below.

\begin{figure}[t]
\centering
\begin{tabular}{c}
 \includegraphics[scale=0.6]{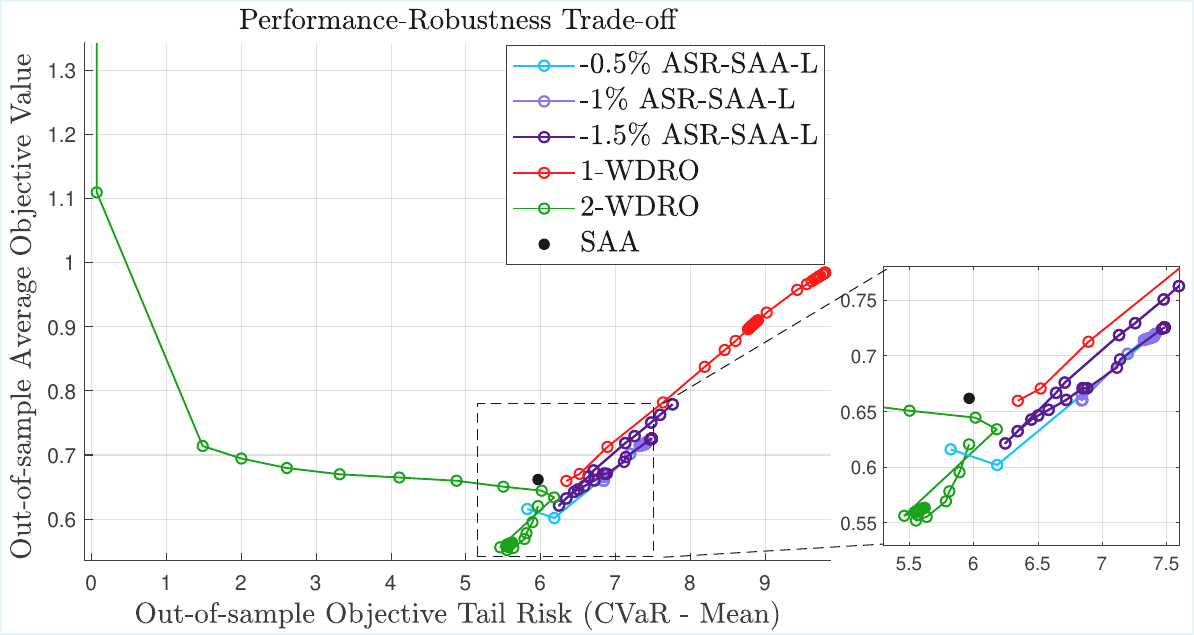}
\end{tabular}
\caption{Case Study 2 (2006--2008). Performance–robustness frontiers for the SBR-SAA-L variant, compared against SAA, 1-WDRO, and 2-WDRO.}
\label{fig:ComparacionCurvasFronteraEff2008_Port}
\end{figure}

The conclusions differ markedly from those of Case Study~1.
In Figure~\ref{fig:ComparacionCurvasFronteraEff2008_Port}, 2-WDRO emerges as the dominant methodology:
its efficient frontier lies consistently below and to the left of the alternatives, reflecting superior control of both expected cost and tail risk during the crisis.
By contrast, the frontiers corresponding to SBR-SAA-L, although clearly improving upon SAA and typically upon 1-WDRO, do not surpass 2-WDRO.
Importantly, we verified that the curve obtained with the original SBR-SAA (quadratic aggregation) is essentially the same as the one reported here for SBR-SAA-L, with discrepancies below plotting resolution. This near coincidence is consistent with the construction: both regularizers are driven by adverse scenarios that severely understate the magnitude and structure of the risk realized in 2008, and thus both encode a form of robustness that is, in hindsight, too mild.

This experiment plays a crucial conceptual role.
First, it confirms that our scenario-based regularization framework is not artificially tuned to always appear superior: when the adverse scenarios are poorly aligned with the realized stress, WDRO with an appropriately chosen Wasserstein radius can be more effective.
Second, it shows that even in such a misspecified setting, both SBR-SAA and SBR-SAA-L still deliver meaningful improvements relative to SAA, demonstrating that incorporating even limited stress information is rarely harmful.
Third, the practical indistinguishability between the quadratic and linear aggregations in this case highlights that the dominant driver of performance is the quality of the scenarios, not the precise choice of aggregation.

Taken together, the two case studies provide a balanced assessment.
When adverse scenarios are informative and reflect genuine stress, our SBR-SAA framework—particularly its SBR-SAA-L variant—yields frontiers that are competitive with or superior to WDRO.
When the scenarios are uninformative or misleading, SBR-SAA behaves as a reasonable robustification of SAA but does not outperform a well-calibrated WDRO.
This dependence on scenario quality is not a weakness but a transparent feature: it makes explicit which pieces of information drive robustness, and it clarifies under which conditions targeted, scenario-based regularization should be preferred to more agnostic distributionally robust approaches.

\section*{Acknowledgments}

M. Junca was supported by the Research Fund of Facultad de Ciencias, Universidad de Los Andes INV-2025-213-3470.



\bibliography{refs}


\appendix

\section{Detailed construction of the motivating example} \label{Appendix:MotivatingExample}

This appendix provides the explicit construction underlying Example~\ref{MotivationalExample}. The aim is to isolate, in a controlled one-dimensional setting, how different choices of the Wasserstein order $p$ in WDRO can lead to solutions with similar out-of-sample expected performance but systematically different tail behavior. The example is designed so that:
(i) the original stochastic program \eqref{eqn:SP_intro} has two optimal solutions located at the boundary of the feasible set, 
(ii) the robust optimization limit of WDRO selects a distinct interior solution, and 
(iii) 1-WDRO and 2-WDRO, for suitable radii, tend to concentrate around different true optima, which allows us to compare their induced tail risks.

We consider a scalar decision variable \(
x \in \mathcal{X} := [\beta,\beta+1]\), \(\beta>0
\), and a scalar random disturbance \(
\xi \sim \mathcal{N}(0,\sigma^{2})
\) with $\sigma>0$. For the numerical illustrations reported in Figures~\ref{fig:OutOfSampleParticularEjeploMotivac_intro} and \ref{fig:ComparacionEjemploMotivacional_intro}, we fix $\beta = 0.5$, $\sigma = \tfrac{1}{2}$, $R = 0.5$.

The loss function $F(x,\xi)$ is constructed as follows. Define
\[
\ell(x)
=
\big(\sqrt{\beta+1}+\sqrt{\beta}\big)^{2}
\big(x-\sqrt{\beta(\beta+1)}\big)^{2},
\]
and
\[
r(x)
=
\big(c x+\beta (1-c)\big)\,(x-\beta-1),
\]
where the constants $a,b,c$ are given by
\[
a
=
\frac{\big(A-B\big)\sqrt{\beta}\,(\sqrt{\beta+1}-\sqrt{\beta})-R}
{\sqrt{\beta(\beta+1)}\,(\sqrt{\beta+1}-\sqrt{\beta})^{2}},
\qquad
b
=
A-B-a(\beta+1),
\]
\[
c
=
-\frac{\sqrt{\beta}}{\sqrt{\beta+1}-\sqrt{\beta}}
\left(
\frac{R}{(\sqrt{\beta+1}-\sqrt{\beta})\,\beta\sqrt{\beta+1}}
+1
\right),
\]
with
\[
A = \beta\left(\sigma\sqrt{\tfrac{2}{\pi}} - 1\right),
\qquad
B = 2(\beta+1)\left(\sigma^{2}\varphi_{\sigma}(1)-\Phi_{\sigma}(-1)\right),
\]
and where $\varphi_{\sigma}$ and $\Phi_{\sigma}$ denote the probability density and distribution functions of $\mathcal{N}(0,\sigma^{2})$, respectively.

We then define
\begin{equation}\label{eq:FAppendixDefinition}
F(x,\xi)
=
\max\big\{-\ell(x)(\xi+1),\, r(x),\, \ell(x)(\xi-1)\big\}
+
\big(a x+b\big)(x-\beta),
\end{equation}
for $x\in[\beta,\beta+1]$ and $\xi\in\mathbb{R}$.

The coefficients are chosen so that the following properties hold (their verification is straightforward but tedious and therefore omitted). First, the stochastic program \eqref{eqn:SP_intro} associated with $F$ admits two minimizers at $x=\beta$ and $x=\beta+1$, and no other minimizers in $(\beta,\beta+1)$. Second, if the expectation in \eqref{eqn:SP_intro} is replaced by a suitable worst-case operator over a compact interval of $\xi$ (the robust optimization limit that WDRO approaches as the Wasserstein radius grows), the resulting robust problem has a unique solution $x_{\mathrm{RO}}$ strictly inside $(\beta,\beta+1)$. Thus, the boundary and robust optima are explicitly separated. Third, the piecewise affine structure in $\xi$, modulated by $\ell(x)$ and $r(x)$, generates regions where $F(x,\xi)$ is more or less sensitive to perturbations in $\xi$ depending on the choice of $x$, which is precisely what different Wasserstein orders $p$ exploit.

Figure~\ref{fig:FuncionEjemploMotivacional_intro} illustrates the evolution of $F(x,\xi)$ as a function of $\xi$ for representative values of $x$ between $\beta$ and $\beta+1$ using the parameter configuration above. As $x$ moves from one endpoint to the other, the active affine branches change and the magnitude of their slopes varies, creating locations where the loss is relatively flat in $\xi$ and others where it is steep. This controlled pattern of sensitivities is central to the behavior observed when applying 1-WDRO and 2-WDRO.

\begin{figure}[t]
\centering
\includegraphics[scale=0.6]{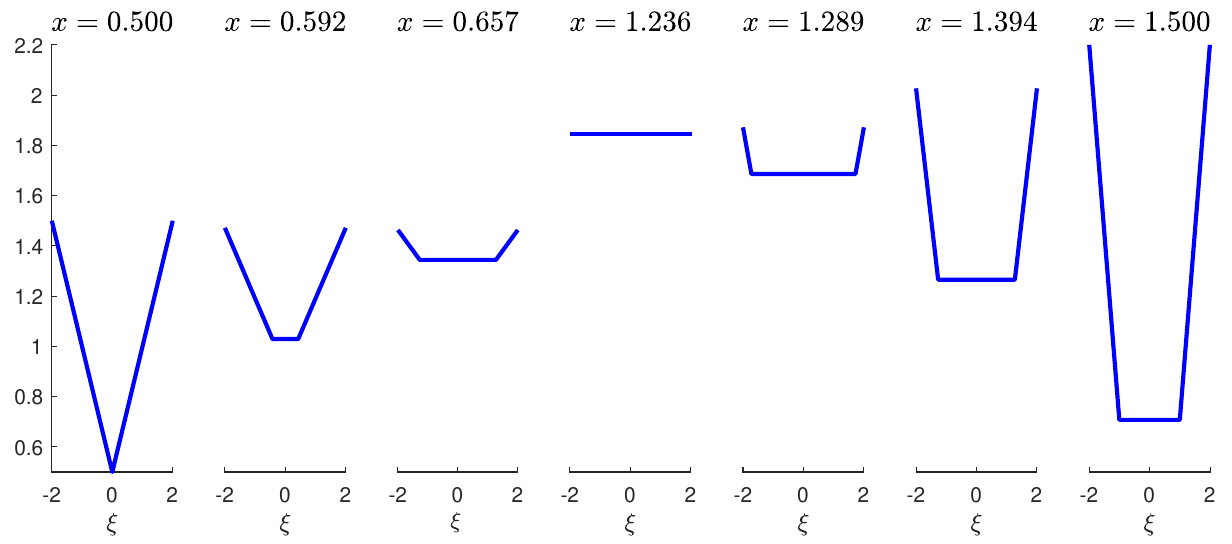}
\caption{Evolution of $F(x,\xi)$ as $x$ varies from $\beta$ to $\beta+1$ for $\beta=0.5$, $\sigma=\tfrac{1}{2}$, and $R=0.5$. The construction induces regions with different sensitivities of the loss to perturbations in $\xi$, which are exploited differently by Wasserstein DRO models with distinct orders $p$.}
\label{fig:FuncionEjemploMotivacional_intro}
\end{figure}

The numerical experiments discussed in the introduction are based on empirical WDRO problems of the form
\[
\min_{x\in[\beta,\beta+1]}
\sup_{\mathbb{Q} : W_p(\mathbb{Q},\widehat{\mathbb{P}}_{n}) \le \varepsilon}
\mathbb{E}_{\xi\sim\mathbb{Q}}[F(x,\xi)],
\]
with $p\in\{1,2\}$, the Euclidean ground metric, and empirical distribution $\widehat{\mathbb{P}}_{n}$ constructed from i.i.d. samples of $\xi$. For Figure~\ref{fig:OutOfSampleParticularEjeploMotivac_intro}, a single training sample is drawn, the 1-WDRO and 2-WDRO problems are solved for a common moderate radius $\varepsilon$, and the resulting decisions $\hat{x}_{1}$ and $\hat{x}_{2}$ are evaluated on a large independent test set (a cloud of out-of-sample evaluations and the corresponding boxplots). Both decisions yield similar out-of-sample means, but the distribution of $F(\hat{x}_{2},\xi)$ exhibits substantially larger upper extremes, revealing a higher exposure to rare but severe losses.

Figure~\ref{fig:ComparacionEjemploMotivacional_intro} reports a more systematic experiment: multiple independent training samples are generated, and the procedure above is repeated over a range of radii $\varepsilon$. For each configuration and each model ($p=1$ and $p=2$), we record the out-of-sample expected loss and the difference between out-of-sample $\mathrm{CVaR}_{\alpha}$ (for a fixed small $\alpha$) and the corresponding mean. The tubes between the 10\% and 90\% quantiles, together with their averages, show that 1-WDRO and 2-WDRO achieve comparable expected performance, while 2-WDRO persistently displays larger $\mathrm{CVaR}_{\alpha}$-to-mean gaps. This confirms that in this example the choice of $p$ affects the tail-risk profile in a systematic way, even when both models are tuned to similar levels of distributional robustness in expectation.

%

\section{Proofs of some  theoretical results}

\subsection{Proof of Theorem \ref{Thm:EquivalenceRegSAAvsWDRO} for a Deterministic Regularizer} \label{Sec:Appendix:ProofThmEquivalenceRegSAAvsWDRO}

Before proceeding with the proof of Theorem \ref{Thm:EquivalenceRegSAAvsWDRO} for the case when $R$ is deterministic, we need the following lemma which allows us to express the feasible set of problem (\ref{eqn:WDROFormSAARegGen_sec}) in terms of finite dimensional variables. 

\begin{lemma} \label{Lemma:RobustVersionOfMean}
Let $\varepsilon>0$, $\zeta$ be a random variable with unknown distribution $\mathbb{P}$ with support $\Xi\subseteq\mathbb{R}$, and we consider a sample $\widehat{\zeta}_{1},\ldots,\widehat{\zeta}_{n}$ of $\zeta$. Let $\widehat{\mathbb{P}}_{n}:=\frac{1}{n}\sum_{i=1}^{n}\delta_{\widehat{\zeta}_{i}}$ be the empirical distribution, and denote $\bar{\zeta}:=\mathbb{E}_{\zeta\sim\widehat{\mathbb{P}}_{n}}[\zeta]=\frac{1}{n}\sum_{i=1}^{n}\widehat{\zeta}_{i}$.  
\begin{enumerate}
    \item[(a)] If $p=1$ and $\Xi$ is an interval, then 
    $$  \sup_{ \mathbb{Q}\in \mathcal{B}_{\varepsilon}(\widehat{\mathbb{P}}_{n})}\mathbb{E}_{\mathbb{Q}}[\zeta]=\min\left\{\bar{\zeta}+\varepsilon,\sup_{\xi\in\Xi}\xi\right\} \quad \mbox{and}\quad   \inf_{ \mathbb{Q}\in \mathcal{B}_{\varepsilon}(\widehat{\mathbb{P}}_{n})}\mathbb{E}_{\mathbb{Q}}[\zeta]=\max\left\{\bar{\zeta}-\varepsilon,\inf_{\xi\in\Xi}\xi\right\}.
$$
    \item[(b)] If $p\geq 1$, $\Xi$ is an interval, and $\sup_{\zeta\in\Xi}\zeta=\infty$, then
    $$\sup_{ \mathbb{Q}\in \mathcal{B}_{\varepsilon}(\widehat{\mathbb{P}}_{n})}\mathbb{E}_{\mathbb{Q}}[\zeta] =\bar{\zeta}+\varepsilon
    \quad \mbox{and}\quad  \inf_{ \mathbb{Q}\in \mathcal{B}_{\varepsilon}(\widehat{\mathbb{P}}_{n})}\mathbb{E}_{\mathbb{Q}}[\zeta]= \bar{\zeta}-\varepsilon.
    $$
\end{enumerate}
\end{lemma}
\proof
It suffices to demonstrate the first equality for each case, as the second is readily obtained by setting $\inf_{ \mathbb{Q}\in \mathcal{B}_{\varepsilon}(\widehat{\mathbb{P}}_{n})}\mathbb{E}_{\mathbb{Q}}[\zeta]=-\sup_{ \mathbb{Q}\in \mathcal{B}_{\varepsilon}(\widehat{\mathbb{P}}_{n})}\mathbb{E}_{\mathbb{Q}}[-\zeta]$. For $p$-Wasserstein distances, by Theorem  \ref{Thm:ReformulacionDROWInterno} we have that
\begin{equation}  \label{ReformulationInf}
 \sup_{ \mathbb{Q}\in \mathcal{B}_{\varepsilon}(\widehat{\mathbb{P}}_{n})}\mathbb{E}_{\mathbb{Q}}[\zeta]=\left\{\begin{array}{lll}
{\displaystyle \inf_{\lambda\geq 0 } } & \lambda\varepsilon^{p}+\frac{1}{n}\sum_{i=1}^{n}s_{i} &\\
\mbox{subject to}  & \sup_{\zeta\in \Xi}\left(\zeta -\lambda\left|\zeta-\widehat{\zeta}_{i} \right|^{p} \right)\leq s_{i} & \forall\:i=1,\ldots,n, 
\end{array}\right.
\end{equation}
In case (a), considering $B:=\sup_{\xi\in\Xi}\xi$ and $A:=\inf_{\xi\in\Xi}\xi$, it is observed that if $B=\infty$, then
$$ \sup_{\zeta\in \Xi}\left(\zeta -\lambda\left|\zeta-\widehat{\zeta}_{i} \right| \right)=\sup_{\zeta\in (A,\infty)}\left(\zeta -\lambda\left|\zeta-\widehat{\zeta}_{i} \right| \right)=\left\{\begin{array}{ll}
    \widehat{\zeta}_{i} & \mbox{ if }\lambda\geq 1,  \\
     \infty &  \mbox{ if }\lambda<1.
\end{array}\right.$$
Therefore, 
$$\sup_{ \mathbb{Q}\in \mathcal{B}_{\varepsilon}(\widehat{\mathbb{P}}_{n})}\mathbb{E}_{\mathbb{Q}}[\zeta]=\bar{\zeta}+\varepsilon=\min\{\bar{\zeta}+\varepsilon,B\} $$
In addition, if $B<\infty$, then
\begin{align}
\sup_{\zeta\in \Xi}\left(\zeta -\lambda\left|\zeta-\widehat{\zeta}_{i} \right| \right)=& \sup_{\zeta\in (A,B)}\left(\zeta -\max_{|z_{i}|\leq \lambda} z_{i}\left(\zeta-\widehat{\zeta}_{i} \right) \right) \nonumber\\
=& \min_{|z_{i}|\leq \lambda}  \sup_{\zeta\in (A,B)}\left(\zeta - z_{i}\left(\zeta-\widehat{\zeta}_{i} \right) \right) \label{Paso:MinMax}\\
=& \min_{|z_{i}|\leq \lambda} \max\left\{(\widehat{\zeta}_{i}-A)z_{i}+A,(\widehat{\zeta}_{i}-B)z_{i}+B\right\} \nonumber\\
=& \left\{\begin{array}{ll}\widehat{\zeta}_{i} & \mbox{ if }\lambda\geq 1, \\ (\widehat{\zeta}_{i}-B)\lambda +B & \mbox{ if }\lambda <1. \end{array}\right. \nonumber
\end{align}
Equality  (\ref{Paso:MinMax})  is guaranteed by  Von Neumann's minimax theorem (see \citep{Bertsekas}). Therefore, we obtain
\begin{align}
\sup_{ \mathbb{Q}\in \mathcal{B}_{\varepsilon}(\widehat{\mathbb{P}}_{n})}\mathbb{E}_{\mathbb{Q}}[\zeta] =& \min\left\{\inf_{\lambda\geq 1}\left(\lambda\varepsilon+\frac{1}{n}\sum_{i=1}^{n}\widehat{\zeta}_{i}\right),\inf_{\lambda< 1}\left(\lambda\varepsilon+\frac{1}{n}\sum_{i=1}^{n}((\widehat{\zeta}_{i}-B)\lambda+B)\right)\right\} \nonumber\\
=& \min\left\{\bar{\zeta}+\varepsilon,\min\{\bar{\zeta}+\varepsilon,B\}\right\} .\nonumber\\
=& \min\left\{\bar{\zeta}+\varepsilon,B\right\} .\nonumber
\end{align}
In case (b), for $p> 1$ we have that
\begin{equation*}
    \sup_{\zeta\in \Xi}\left(\zeta -\lambda\left|\zeta-\widehat{\zeta}_{i} \right|^{p} \right)=\widehat{\zeta}_{i}+\frac{p-1}{\lambda^{1/(p-1)}p^{p/(p-1)}}.
\end{equation*}
Therefore, 
\begin{equation*}
    \sup_{ \mathbb{Q}\in \mathcal{B}_{\varepsilon}(\widehat{\mathbb{P}}_{n})}\mathbb{E}_{\mathbb{Q}}[\zeta] =  \inf_{\lambda\geq 0}\left(\lambda\varepsilon^p+\frac{p-1}{\lambda^{1/(p-1)}p^{p/(p-1)}}+\bar{\zeta}\right)=\bar{\zeta}+\varepsilon.
\end{equation*}
\endproof

\proof[Proof of Theorem \ref{Thm:EquivalenceRegSAAvsWDRO} for a Deterministic Regularizer]
The presented outcome follows straightforwardly from Lemma \ref{Lemma:RobustVersionOfMean}. Specifically, Lemma \ref{Lemma:RobustVersionOfMean}-(a) implies that for $p=1$ and $F(x,\Xi)$ as an interval,  we have the following equivalence
\begin{align*}
    \widehat{J}_{R,n}^{\mathrm{DD}}(\varepsilon) &=\underset{ x\in\mathbb{X}  }{\min} \min\left\{\frac{1}{n}{\displaystyle\sum_{i=1}^{n}}F\left(x,\widehat{\xi}_{i}\right)+\varepsilon C, \sup_{\xi\in\Xi}F(x,\xi)\right\},
\end{align*}
As $\sup_{\xi\in\Xi}F(x,\xi)=\infty$, then (\ref{eqn:WDROFormSAARegGen_presec}) hold. The proof  for the case $p\geq 1$, $F(x,\Xi)$ as an interval, and $\sup_{\xi\in\Xi}F(x,\xi)=\infty$ for each $x\in\mathcal{X}$, can be established in a similar way using Lemma \ref{Lemma:RobustVersionOfMean}-(b).
\endproof


\subsection{Proof of Lemma \ref{lem:WassersteinPropagationBound_user}.}\label{Apendice:Pruebalem:WassersteinPropagationBound_user}

\proof[Proof of Lemma \ref{lem:WassersteinPropagationBound_user}] Let  $\widetilde{\xi}_{1},\ldots,\widetilde{\xi}_{m}$ be another sample of $\xi$, then let   $\widetilde{\mathbb{P}}_{m}$ be the empirical distribution generated by this sample. This last sample of $\xi$ induces the sample $\widetilde{\zeta}^{x,F}_{1},\ldots,\widetilde{\zeta}^{x,F}_{m}$ of $\zeta^{x,F}$ given by  $\widetilde{\zeta}^{x,F}_{i}:=F\left( x, \widetilde{\xi}_{i}\right)$, so we consider $F(x,\cdot)_{\#}\widetilde{\mathbb{P}}_{m}$, the push-forward of $\widetilde{\mathbb{P}}_{m}$ via $F(x,\cdot)$, which is also the empirical distribution generated by the sample  $\left\{\widetilde{\zeta}^{x,F}_{i}\right\}_{i=1}^{m}$. Because $\widetilde{\mathbb{P}}_{m} \rightarrow \mathbb{P}$ and $ F(x,\cdot)_{\#}\widetilde{\mathbb{P}}_{m} \rightarrow F(x,\cdot)_{\#}\mathbb{P}$ weakly as $m$ goes to $\infty$, by Corollary 6.11 en \citep{Villani2008} we have that
\begin{equation} \label{Eq:ConvergenciaEnWassesrtenEmiric}
W_{p}\left( \widehat{\mathbb{P}}_{n} ,\widetilde{\mathbb{P}}_{m}\right)\overset{{\scriptstyle m\rightarrow \infty}}{\longrightarrow}W_{p}\left( \widehat{\mathbb{P}}_{n} ,\mathbb{P}\right) \quad \mbox{ and }\quad W_{p}\left( F(x,\cdot)_{\#}\widehat{\mathbb{P}}_{n} ,F(x,\cdot)_{\#}\widetilde{\mathbb{P}}_{m}\right)\overset{{\scriptstyle m\rightarrow \infty}}{\longrightarrow}W_{p}\left( F(x,\cdot)_{\#}\widehat{\mathbb{P}}_{n} ,F(x,\cdot)_{\#}\mathbb{P}\right). 
\end{equation}
Additionally, for each $m$ we get that
\begin{align*}
W_{p}^{p}\left( F(x,\cdot)_{\#}\widehat{\mathbb{P}}_{n} ,F(x,\cdot)_{\#}\widetilde{\mathbb{P}}_{m}\right) &= \inf\left\{\sum_{i=1}^{n}\sum_{j=1}^{m} \lambda_{i,j}\left|\widehat{\zeta}^{x,F}_{i}-\widetilde{\zeta}^{x,F}_{j}  \right|^{p} \:\left|\: \begin{array}{l} \sum_{i=1}^{n}\lambda_{i,j}=\frac{1}{m},\\ \sum_{j=1}^{m}\lambda_{i,j}=\frac{1}{n},\\ \lambda_{i,j}\geq 0, \\ i=1,\ldots,n,\\ j=1,\ldots,m \end{array}  \right.\right\} \\
&= \inf\left\{\sum_{i=1}^{n}\sum_{j=1}^{m} \lambda_{i,j}\left|F\left( x, \widehat{\xi}_{i}\right)-F\left( x, \widetilde{\xi}_{j}\right)   \right|^{p} \:\left|\: \begin{array}{l} \sum_{i=1}^{n}\lambda_{i,j}=\frac{1}{m},\\ \sum_{j=1}^{m}\lambda_{i,j}=\frac{1}{n},\\ \lambda_{i,j}\geq 0, \\ i=1,\ldots,n,\\ j=1,\ldots,m \end{array}  \right.\right\}  \\
&\leq  \inf\left\{\sum_{i=1}^{n}\sum_{j=1}^{m} \lambda_{i,j} \gamma_{x,F}^{p}\left\| \widehat{\xi}_{i}- \widetilde{\xi}_{j}  \right\|^{p} \:\left| \: \begin{array}{l} \sum_{i=1}^{n}\lambda_{i,j}=\frac{1}{m},\\ \sum_{j=1}^{m}\lambda_{i,j}=\frac{1}{n},\\ \lambda_{i,j}\geq 0, \\ i=1,\ldots,n,\\ j=1,\ldots,m \end{array}  \right.\right\} \:\: \begin{array}{l}\mbox{by Assumption \ref{AssumptionPrincipal_sec}}\\ \mbox{}\end{array}\\
&= \gamma_{x,F}^{p}W_{p}^{p}\left( \widehat{\mathbb{P}}_{n} ,\widetilde{\mathbb{P}}_{m}\right).
\end{align*} 
Therefore, by (\ref{Eq:ConvergenciaEnWassesrtenEmiric}) we conclude 
$$W_{p}^{p}\left( F(x,\cdot)_{\#}\widehat{\mathbb{P}}_{n} ,F(x,\cdot)_{\#}\mathbb{P}\right)\leq \gamma_{x,F}^{p} W_{p}^{p}\left( \widehat{\mathbb{P}}_{n} ,\mathbb{P}\right).$$
\endproof


\section{Proof of Theorem \ref{thm:AsymptoticConsistency_sec}} \label{sec:ProofAsymptoticConsistency_sec}

\proof[Proof of Theorem \ref{thm:AsymptoticConsistency_sec}]
As $\hat{x}_{m,n}^{\mathrm{DD}} \in \mathcal{X}$, we have $J^* \le \mathbb{E}_{\xi\sim\mathbb{P}}[F(\hat{x}_{m,n}^{\mathrm{DD}},\xi)]$. Moreover, Theorem \ref{Thm:IneuqlityProbaOptimalValueAdv} implies that
$$
\mathbb{P}^{n}\left( J^{*} \leq \mathbb{E}_{\xi\sim\mathbb{P}}\left[F\left(\hat{x}_{m,n}^{\mathrm{DD}},\xi\right)\right]\leq \widehat{J}_{m,n}^{\mathrm{DD}} +\varepsilon_n(\beta_n) \alpha \right)>1- \beta_{n}.
$$
for all $n \in \mathbb{N}$. As $\sum_{n=1}^\infty \beta_n < \infty$, the Borel--Cantelli Lemma further implies that
$$
\mathbb{P}^{\infty}\left( J^{*} \leq \mathbb{E}_{\xi\sim\mathbb{P}}\left[F\left(\hat{x}_{m,n}^{\mathrm{DD}},\xi\right)\right]\leq \widehat{J}_{m,n}^{\mathrm{DD}} +\varepsilon_n(\beta_n) \alpha \text{ for all sufficiently large } n\right)=1.
$$
Since we have $\lim_{n \to \infty}\varepsilon_n(\beta_n)=0$, then this last equality implies
$$
\mathbb{P}^{\infty}\left( J^{*} \leq \liminf_{n\rightarrow\infty} \widehat{J}_{m,n}^{\mathrm{DD}}  \right)=1.
$$

To prove assertion (i), it thus remains to be shown that $\limsup_{n \to \infty} \widehat{J}_{m,n}^{\mathrm{DD}} \leq  J^*$ with probability 1. We choose any $\delta > 0$, fix a $\delta$-optimal decision $x_\delta \in \mathcal{X}$ for (\ref{eqn:SP_intro}) with $\mathbb{E}_{\xi\sim\mathbb{P}}[F(x_\delta,\xi)] \le J^* + \delta$, and for every $n \in \mathbb{N}$ let $\hat{\mathbb{Q}}_n^{x_{\delta}} \in \mathcal{B}_{\varepsilon\cdot (\widehat{R}_{m}(x_{\delta})+\alpha)}\left(F\left(x_{\delta},\cdot\right)_{\#}\widehat{\mathbb{P}}_{n}\right) $ be a $\delta$-optimal distribution corresponding to $x_\delta$ with
\[
\sup_{\mathbb{Q}\in\mathcal{B}_{\varepsilon\cdot (\widehat{R}_{m}(x_{\delta})+\alpha)}\left(F\left(x_{\delta},\cdot\right)_{\#}\widehat{\mathbb{P}}_{n}\right) } \mathbb{E}_{\varsigma\sim\mathbb{Q}}[\varsigma] \le \mathbb{E}_{\varsigma\sim\hat{\mathbb{Q}}_n^{x_{\delta}} }[\varsigma] + \delta.
\]

Then, we have
\begin{align}
\limsup_{n \to \infty} \widehat{J}_{m,n}^{\mathrm{DD}} & \le \limsup_{n \to \infty} \sup_{\mathbb{Q}\in\mathcal{B}_{\varepsilon\cdot (\widehat{R}_{m}(x_{\delta})+\alpha)}\left(F\left(x_{\delta},\cdot\right)_{\#}\widehat{\mathbb{P}}_{n}\right) } \mathbb{E}_{\varsigma\sim\mathbb{Q}}[\varsigma]  \nonumber\\
& \le \limsup_{n \to \infty} \mathbb{E}_{\varsigma\sim\hat{\mathbb{Q}}_n^{x_{\delta}} }[\varsigma] + \delta \nonumber\\
&= \limsup_{n \to \infty} \left( \mathbb{E}_{\varsigma\sim F(x_{\delta},\cdot)_{\#}\mathbb{P} }[\varsigma] + \mathbb{E}_{\varsigma\sim\hat{\mathbb{Q}}_n^{x_{\delta}} }[\varsigma] - \mathbb{E}_{\varsigma\sim F(x_{\delta},\cdot)_{\#}\mathbb{P} }[\varsigma] \right) + \delta \nonumber\\
& \le \limsup_{n \to \infty} \left( \mathbb{E}_{\varsigma\sim F(x_{\delta},\cdot)_{\#}\mathbb{P} }[\varsigma] + W_{1}( F(x_{\delta},\cdot)_{\#}\mathbb{P},\hat{\mathbb{Q}}_n^{x_{\delta}})\right) + \delta  \label{ineq:UsoOfThm:KantorovichRubinsteinDD} \\
&= \mathbb{E}_{\varsigma\sim F(x_{\delta},\cdot)_{\#}\mathbb{P} }[\varsigma] + \delta \quad\:\: \mathbb{P}^\infty-\mbox{almost surely}   \label{Ineq:UsoOfLimitDecisionsDist}\\
&= \mathbb{E}_{\xi\sim\mathbb{P}}[F(x_\delta,\xi)] + \delta   \quad\:\: \mathbb{P}^\infty-\mbox{almost surely}  \nonumber\\
&\le J^{*}+2\delta   \quad\:\: \mathbb{P}^\infty-\mbox{almost surely} \nonumber
\end{align}
Inequality (\ref{ineq:UsoOfThm:KantorovichRubinsteinDD}) is due to Theorem \ref{Thm:KantorovichRubinsteinDD_sec}, and equality (\ref{Ineq:UsoOfLimitDecisionsDist}) is due to Lemma \ref{lem:ConvergenceOfdistributions_sec}. Therefore, since $\delta > 0$ was chosen arbitrarily, we thus conclude that 
$$
\mathbb{P}^{\infty}\left(   \limsup_{n\rightarrow\infty} \widehat{J}_{m,n}^{\mathrm{DD}}  \leq  J^{*}\right)=1.
$$

To prove assertion (ii), fix an arbitrary realization of the stochastic process $\{\hat{\xi}_n\}_{n \in \mathbb{N}}$ such that $J^* = \lim_{n \to \infty} \widehat{J}_{m,n}^{\mathrm{DD}}$ and $J^* \le \mathbb{E}_{\xi\sim\mathbb{P}}[F(\hat{x}_{m,n}^{\mathrm{DD}},\xi)] \le \widehat{J}_{m,n}^{\mathrm{DD}}$ for all sufficiently large $n$. From the proof of assertion (i) we know that these two conditions are satisfied $\mathbb{P}^\infty$-almost surely. Using these assumptions, one easily verifies that
\begin{equation} \label{ineq:LimitsBetwwenTrueAndDD}
\liminf_{n \to \infty}\mathbb{E}_{\xi\sim\mathbb{P}}[F(\hat{x}_{m,n}^{\mathrm{DD}},\xi)] \le \lim_{n \to \infty} \widehat{J}_{m,n}^{\mathrm{DD}} = J^*.
\end{equation}

Next, let $x^*$ be an accumulation point of the sequence $\{\hat{x}_{m,n}^{\mathrm{DD}}\}_{n \in \mathbb{N}}$, and note that $x^* \in \mathcal{X}$ as $\mathcal{X}$ is closed. By passing to a subsequence, if necessary, we may assume without loss of generality that $x^* = \lim_{n \to \infty}\hat{x}_{m,n}^{\mathrm{DD}}$. Thus,
\[
J^* \le \mathbb{E}_{\xi\sim\mathbb{P}}[F(x^*,\xi)] \le \mathbb{E}_{\xi\sim\mathbb{P}}\left[\liminf_{n \to \infty} F(\hat{x}_{m,n}^{\mathrm{DD}},\xi)\right] \le \liminf_{n \to \infty}\mathbb{E}_{\xi\sim\mathbb{P}}[F(\hat{x}_{m,n}^{\mathrm{DD}},\xi)] \le J^*,
\]
where the first inequality exploits that $x^* \in \mathcal{X}$, the second inequality follows from the lower semicontinuity of $F(x,\xi)$ in $x$, the third inequality holds due to Fatou’s lemma (which applies because $F(x,\xi)$ grows at most linearly in $\xi$), and the last inequality follows from (\ref{ineq:LimitsBetwwenTrueAndDD}). Therefore, we have $\mathbb{E}_{\xi\sim\mathbb{P}}[F(x^*,\xi)] = J^*$.

\endproof

%

\section{SOCP Reformulations for WDRO Benchmarks for Mean-Risk Portfolio Optimization}
\label{app:wdro_reforms}

This section provides the detailed derivations for the tractable reformulations of the `1-WDRO` and `2-WDRO` benchmark models applied to the mean-CVaR portfolio optimization problem \eqref{eqn:StocasticFormMeanRiskTrue}.

\begin{lemma}[SOCP Reformulations for WDRO Benchmarks]
\label{Lemma:Reformulatiosn12WDROMeanRisk}
    Assuming the support set is $\Xi=\mathbb{R}^{d}$ and the ground metric is the Euclidean norm $\mathbf{d}(\cdot,\cdot) = \|\cdot\|_2$, the `1-WDRO` approach for problem \eqref{eqn:StocasticFormMeanRiskTrue} is equivalent to the following Second-Order Cone Program (SOCP):
    \begin{equation}\label{eqn:ReformulCVaRCasoP1_Appendix}
    \left\{
\begin{array}{cll}
{\displaystyle \inf_{w,\tau,\lambda,s}} & {\displaystyle \lambda \varepsilon +\frac{1}{n}\sum_{i=1}^{n}s_{i}} &\\
\mbox{s.t.} & {\displaystyle a_{k}\langle w,\widehat{\xi}_{i}\rangle + b_{k}\tau  \leq s_{i}  } & \forall i\in[n], k\in\{1,2\} \\
& {\displaystyle \|w\|_2|a_{k}|\leq \lambda} & \forall k\in\{1,2\} \\
& w\in\mathcal{W},\;\tau\in\mathbb{R},\;\lambda\geq 0,\;s\in\mathbb{R}^{n}.
\end{array}
\right.
\end{equation}
Furthermore, the `2-WDRO` approach for problem \eqref{eqn:StocasticFormMeanRiskTrue} is equivalent to the following SOCP:
    \begin{equation}\label{eqn:ReformulCVaRCasoP2_Appendix}
\left\{
\begin{array}{cll}
{\displaystyle \inf_{w,\tau,\lambda,s,z}} & {\displaystyle \lambda \varepsilon^{2} +\frac{1}{n}\sum_{i=1}^{n}s_{i}} & \\
\mbox{s.t.} & {\displaystyle \frac{a_{k}^{2}}{4}z + a_{k}\langle w,\widehat{\xi}_{i}\rangle + b_{k}\tau \leq s_{i}} & \forall i\in[n], k\in\{1,2\} \\
& {\displaystyle \left\| \begin{pmatrix} 2w \\ \lambda - z \end{pmatrix} \right\|_2 \leq \lambda + z} & \\
& w\in\mathcal{W},\;\tau\in\mathbb{R},\;\lambda\geq 0,\;s\in\mathbb{R}^{n}, \; z \geq 0.
\end{array}
\right.
\end{equation}
In both formulations, the constants are defined as $a_{1}=-1$, $a_{2}=-1-\frac{\rho}{\alpha}$, $b_{1}=\rho$, and $b_{2}=\rho\left(1-\frac{1}{\alpha}\right)$.
\end{lemma}

\begin{proof}
We recall from Theorem \ref{Thm:ReformulacionDROWInterno} that for $p \in \{1,2\}$, the WDRO problem can be reformulated as:
\begin{equation}\label{Eqn:ReformulacionDROWManRisk_Appendix}
\left\{
\begin{array}{cll}
{\displaystyle \inf_{w,\tau,\lambda,s}} & {\displaystyle \lambda \varepsilon^{p} +\frac{1}{n}\sum_{i=1}^{n}s_{i}} &\\
\mbox{s.t.} & {\displaystyle \sup_{\xi\in\mathbb{R}^{d}}\left(F(w,\tau,\xi)-\lambda \|\xi-\widehat{\xi}_{i}\|^{p}_{2} \right) \leq s_{i}  } & \forall i \in [n] \\
& w\in\mathcal{W},\;\tau\in\mathbb{R},\;\lambda\geq 0,\;s\in\mathbb{R}^{n}.
\end{array}
\right.
\end{equation}
A key step in deriving the specific reformulations is to express the objective function $F$ from \eqref{eq:F_portfolio} in its equivalent max-affine form:
$$F(w,\tau,\xi)=\max_{k\in\left\{1,2\right\}}\left(a_{k}\langle w,\xi \rangle+b_{k}\tau\right).$$
This characterization allows the inner supremum term in the constraint of \eqref{Eqn:ReformulacionDROWManRisk_Appendix} to be rewritten by interchanging the supremum and maximum operators:
\begin{align}
    \sup_{\xi\in\mathbb{R}^{d}}\left(F(w,\tau,\xi)-\lambda \|\xi-\widehat{\xi}_{i}\|_{2}^{p} \right) &= \sup_{\xi\in\mathbb{R}^{d}}\left( \max_{k\in\left\{1,2\right\}}\left(a_{k}\langle w,\xi \rangle+b_{k}\tau\right) -\lambda \|\xi-\widehat{\xi}_{i}\|_{2}^{p} \right) \nonumber\\
    & = \max_{k\in\left\{1,2\right\}}\left( \sup_{\xi\in\mathbb{R}^{d}}\left(a_{k}\langle w,\xi \rangle+b_{k}\tau -\lambda \|\xi-\widehat{\xi}_{i}\|_{2}^{p}\right) \right). \label{Eqn:InnerSoppremumReformul_Appendix}
\end{align}
We now analyze the inner supremum for $p=1$ and $p=2$.

\textbf{Case 1: 1-WDRO ($p=1$)}
The inner supremum is an unconstrained problem whose objective function's conjugate is known. We have:
\begin{align}
    \sup_{\xi\in\mathbb{R}^{d}}\left(a_{k}\langle w,\xi \rangle - \lambda \|\xi-\widehat{\xi}_{i}\|_{2}\right) + b_k \tau &= \sup_{\delta \in \mathbb{R}^d} \left(a_k \langle w, \delta + \widehat{\xi}_i \rangle - \lambda \|\delta\|_2 \right) + b_k \tau \nonumber \\
    &= a_k \langle w, \widehat{\xi}_i \rangle + b_k \tau + \sup_{\delta \in \mathbb{R}^d} \left( \langle a_k w, \delta \rangle - \lambda \|\delta\|_2 \right). \nonumber
\end{align}
The final supremum term is the definition of the convex conjugate of the function $g(\delta) = \lambda \|\delta\|_2$, evaluated at $a_k w$. This conjugate is finite if and only if the dual norm of $a_k w$ is less than or equal to $\lambda$. The dual norm of the Euclidean norm is itself. Thus,
\[
\sup_{\delta \in \mathbb{R}^d} \left( \langle a_k w, \delta \rangle - \lambda \|\delta\|_2 \right) =
\begin{cases}
    0 & \text{if } \|a_k w\|_2 \le \lambda \\
    \infty & \text{otherwise}.
\end{cases}
\]
Since $\|a_k w\|_2 = |a_k| \|w\|_2$, the constraint $\sup_{\xi\in\mathbb{R}^{d}}\left(F(w,\tau,\xi)-\lambda \|\xi-\widehat{\xi}_{i}\|_{2} \right)\leq s_{i}$ holds if and only if:
\[
 \begin{cases}
     a_{k}\langle w,\widehat{\xi}_{i}\rangle + b_{k}\tau  \leq s_{i},  & \forall  k\in\{1,2\},\\
    |a_{k}|\|w\|_2\leq \lambda & \forall k\in\{1,2\}.
\end{cases}
\]
Substituting this back into \eqref{Eqn:ReformulacionDROWManRisk_Appendix} yields the SOCP formulation \eqref{eqn:ReformulCVaRCasoP1_Appendix}.

\textbf{Case 2: 2-WDRO ($p=2$)}
The inner supremum is an unconstrained convex quadratic maximization problem:
\[ \sup_{\xi\in\mathbb{R}^{d}}\left(a_{k}\langle w,\xi \rangle+b_{k}\tau -\lambda \|\xi-\widehat{\xi}_{i}\|_{2}^{2}\right). \]
The optimal $\xi^*$ is found by setting the gradient with respect to $\xi$ to zero: $a_k w - 2\lambda(\xi^* - \widehat{\xi}_i) = 0$, which gives $\xi^* = \widehat{\xi}_i + \frac{a_k}{2\lambda}w$. Substituting this back into the expression yields:
\[
a_k \langle w, \widehat{\xi}_i + \frac{a_k}{2\lambda}w \rangle + b_k \tau - \lambda \|\frac{a_k}{2\lambda}w\|_2^2 = a_k \langle w, \widehat{\xi}_i \rangle + b_k \tau + \frac{a_k^2 \|w\|_2^2}{4\lambda}.
\]
Therefore, the constraint from \eqref{Eqn:ReformulacionDROWManRisk_Appendix} becomes:
\[
\max_{k \in \{1,2\}} \left( \frac{a_{k}^{2}\|w\|_2^{2}}{4\lambda}+a_{k}\langle w,\widehat{\xi}_{i}\rangle + b_{k}\tau \right) \leq s_{i}.
\]
This is equivalent to the set of constraints $\frac{a_{k}^{2}\|w\|_2^{2}}{4\lambda}+a_{k}\langle w,\widehat{\xi}_{i}\rangle + b_{k}\tau  \leq s_{i}$ for $k=1,2$.
By introducing an auxiliary variable $z \ge 0$ such that $\|w\|_2^2 \le z \lambda$, the term $\frac{\|w\|_2^2}{\lambda}$ can be replaced by $z$. The hyperbolic constraint $\|w\|_2^2 \le z \lambda$ is equivalent to the rotated second-order cone constraint $\left\| \begin{pmatrix} 2w \\ \lambda - z \end{pmatrix} \right\|_2 \leq \lambda + z$. This sequence of equivalences proves that the formulation \eqref{eqn:ReformulCVaRCasoP2_Appendix} is correct.
\end{proof}


\section{MISOCP Reformulations for the Scenario-Based Mean-Risk Models}
\label{app:proof_prop_RefAdversoSAAMeanRisk}

This section provides a detailed proof of Proposition~\ref{prop:RefAdversoSAAMeanRisk_sec}, establishing valid MISOCP reformulations for the SBR-SAA-L and SBR-SAA portfolio models introduced in Section~\ref{subsec:PortfolioProblem} .

\begin{proof}[Proof of Proposition~\ref{prop:RefAdversoSAAMeanRisk_sec}]
We begin by recalling the structure of the mean-CVaR loss function
\begin{equation}\label{eq:app_F_portfolio}
F(w,\tau,\xi)
=
-w^\top \xi
+
\rho\left(
\tau + \frac{1}{\alpha}(-w^\top \xi - \tau)_+
\right),
\end{equation}
where $w\in\mathcal{W}$, $\tau\in\mathbb{R}$, and $\xi\in\mathbb{R}^d$.
Throughout, $\|\cdot\|_2$ denotes the Euclidean norm, and $\|\cdot\|_\infty$ the max-norm.

\medskip

\noindent\textbf{Step 1: Bounds and notation.}
Let
\[
K=\max\left\{
\max_{j\in[m]}\|\zeta_{j}\|_{\infty},
\max_{i\in[n]}\|\widehat{\xi}_{i}\|_{\infty}
\right\},
\]
and let $C>0$ be such that there exists an optimal solution with $|\tau|\le C$ for the problems under consideration.\footnote{Such a bound is standard for mean-CVaR formulations (cf.\ \cite{Rockafellar2000}) and can be enforced explicitly if desired without altering optimal solutions.}
Because $w\in\mathcal{W}$ implies $w\ge 0$ and $\mathbf{1}^\top w=1$, we have
\begin{equation}\label{eq:app_bound_w}
\|w\|_1 = 1,
\qquad
\|w\|_2 \le 1.
\end{equation}
Hence, for any adverse scenario $\zeta_j$ and any feasible $(w,\tau)$ with $|\tau|\le C$,
\begin{equation}\label{eq:app_bound_inner}
|w^\top \zeta_j|
\le
\|w\|_1 \|\zeta_j\|_\infty
\le
K,
\qquad
|-w^\top \zeta_j - \tau|
\le
K + C.
\end{equation}
We set
\[
M_{1}:=2(K+C)+\delta,
\qquad
M_{2}:=1+K,
\qquad
M_{3}:=1+\frac{\rho}{\alpha},
\]
for some fixed $\delta>0$.
With these choices, $M_1$ is sufficiently large so that the big-$M$ implications used below are valid for all feasible $(w,\tau)$; $M_2$ and $M_3$ will be used to decouple the gradient-norm expressions from the binary variables.

\medskip

\noindent\textbf{Step 2: Subgradient structure at adverse scenarios.}
Let $\zeta_j$ be a fixed adverse scenario.
As derived in Section~\ref{subsec:PortfolioProblem} , define
$\phi(z,\tau) = \tau + \frac{1}{\alpha}(z-\tau)_+$
and $z=-w^\top\zeta_j$.
Then
\[
F(w,\tau,\zeta_j)
=
-w^\top \zeta_j
+
\rho\,\phi(-w^\top \zeta_j,\tau),
\]
and the subdifferential of $F$ with respect to $\zeta_j$ is
\[
\partial_{\zeta} F(w,\tau,\zeta_j)
=
-w\left(1 + \rho\,\partial_z \phi(-w^\top \zeta_j,\tau)\right).
\]
Using the convention (Section~\ref{subsec:PortfolioProblem} ) that at the kink $-w^\top \zeta_j = \tau$ we select the subgradient with largest norm, we obtain the following explicit expression for the chosen (sub)gradient:
\[
\nabla_{\zeta} F(w,\tau,\zeta_j)
=
\begin{cases}
-w, & -w^\top \zeta_j < \tau,\\[0.5ex]
-w\left(1+\dfrac{\rho}{\alpha}\right), & -w^\top \zeta_j \ge \tau.
\end{cases}
\]
Consequently,
\begin{equation}\label{eq:app_grad_norm_cases}
\big\|\nabla_{\zeta} F(w,\tau,\zeta_j)\big\|_2
=
\begin{cases}
\|w\|_2, & -w^\top \zeta_j < \tau,\\[0.5ex]
\left(1+\dfrac{\rho}{\alpha}\right)\|w\|_2, & -w^\top \zeta_j \ge \tau.
\end{cases}
\end{equation}
This piecewise-constant structure in $\zeta_j$ and the fact that $\|w\|_2\le 1$ are the key ingredients behind the MISOCP reformulations.

\medskip

\noindent\textbf{Step 3: Encoding the loss-region logic via binaries.}
For each $j\in[m]$, introduce a binary variable $z_j\in\{0,1\}$ intended to distinguish the two regimes in \eqref{eq:app_grad_norm_cases}:
\[
z_j
=
\begin{cases}
1, & \text{if } -w^\top \zeta_j < \tau \quad\text{(``small-loss'' region)},\\
0, & \text{if } -w^\top \zeta_j \ge \tau - \delta \quad\text{(``large-loss'' / tail region)}.
\end{cases}
\]
The small tolerance $\delta>0$ is used to separate the two regimes in a numerically stable way; it does not affect the modeling power of the formulation.

The following pair of constraints is used in both MISOCP formulations:
\begin{align}
-w^\top \zeta_j - \tau &\le M_{1}(1-z_{j}),
\label{eq:app_logic_1}
\\
w^\top \zeta_j + \tau - \delta &\le M_{1}z_{j},
\label{eq:app_logic_2}
\end{align}
for all $j\in[m]$.

We now verify that, with $M_1$ as defined above, \eqref{eq:app_logic_1}--\eqref{eq:app_logic_2} correctly encode the intended disjunction.

\emph{(i) Case $z_j = 1$.}
Constraint~\eqref{eq:app_logic_1} becomes
\(
-w^\top \zeta_j - \tau \le 0,
\)
or equivalently
\(
-w^\top \zeta_j \le \tau.
\)
Thus, $z_j=1$ is only compatible with the small-loss region
$-w^\top \zeta_j \le \tau$.
Constraint~\eqref{eq:app_logic_2} reads
\(
w^\top \zeta_j + \tau - \delta \le M_1,
\)
which is automatically satisfied for all feasible $(w,\tau)$ because
\(
w^\top \zeta_j + \tau - \delta
\le
K + C
\le
M_1
\)
by \eqref{eq:app_bound_inner} and the definition of $M_1$.
Hence, for $z_j=1$, the only effective restriction is $-w^\top \zeta_j \le \tau$, i.e., the small-loss regime.

\emph{(ii) Case $z_j = 0$.}
Constraint~\eqref{eq:app_logic_2} becomes
\(
w^\top \zeta_j + \tau - \delta \le 0,
\)
or equivalently
\(
-w^\top \zeta_j \ge \tau - \delta.
\)
Thus, $z_j=0$ is only compatible with the large-loss (or near-threshold) region.
Constraint~\eqref{eq:app_logic_1} becomes
\(
-w^\top \zeta_j - \tau \le M_1,
\)
which is again automatically satisfied for all feasible $(w,\tau)$ by \eqref{eq:app_bound_inner} and the choice of $M_1$.
Therefore, for $z_j=0$, the effective restriction is $-w^\top \zeta_j \ge \tau - \delta$, approximating the condition $-w^\top \zeta_j \ge \tau$.
Because $\delta>0$ can be chosen arbitrarily small, this approximation can be made as tight as desired.

Thus, constraints \eqref{eq:app_logic_1}--\eqref{eq:app_logic_2} consistently encode the two regimes in \eqref{eq:app_grad_norm_cases} via the binary variable $z_j$.

\medskip

\noindent\textbf{Step 4: MISOCP reformulation for SBR-SAA-L.}

The SBR-SAA-L model is
\begin{equation}\label{eq:app_SBR_SAA_L}
\min_{w\in\mathcal{W},\,\tau\in\mathbb{R}}
\frac{1}{n}\sum_{i=1}^{n} F(w,\tau,\widehat{\xi}_i)
+
\varepsilon\,
\sum_{j=1}^{m} r_j
\big\|\nabla_{\zeta} F(w,\tau,\zeta_j)\big\|_2.
\end{equation}
Define auxiliary variables $\gamma_j\in\mathbb{R}$ for $j\in[m]$ and consider the MISOCP formulation
\begin{equation}\label{eq:app_MISOCP_L}
\left\{
\begin{array}{cll}
\displaystyle \min_{w,\tau,\gamma,z}
&
\displaystyle
\frac{1}{n}\sum_{i=1}^{n} F(w,\tau,\widehat{\xi}_i)
+
\varepsilon\, \sum_{j=1}^{m} r_j \gamma_{j}
& 
\\[0.5ex]
\text{subject to}
&
w\in\mathcal{W},\ \tau\in\mathbb{R},\ \gamma\in\mathbb{R}^{m},\ z\in\{0,1\}^m,
&
\\[0.5ex]
&
\text{\eqref{eq:app_logic_1}--\eqref{eq:app_logic_2} hold for all } j\in[m],
&
\\[0.5ex]
&
\|w\|_2 - \gamma_{j} \leq M_2 (1-z_{j}),
& j\in[m],
\\
&
\|w\|_2 \left(1 + \tfrac{\rho}{\alpha}\right) - \gamma_{j} \leq M_3 z_j,
& j\in[m].
\end{array}
\right.
\end{equation}
We show that \eqref{eq:app_SBR_SAA_L} and \eqref{eq:app_MISOCP_L} are equivalent in the sense of having the same optimal value and that optimal solutions correspond via suitable choices of $(\gamma,z)$.

\begin{enumerate}
    \item \emph{From SBR-SAA-L to MISOCP.}
Let $(w,\tau)$ be feasible for \eqref{eq:app_SBR_SAA_L}.
For each $j\in[m]$, define
\[
z_j
=
\begin{cases}
1, & \text{if } -w^\top \zeta_j < \tau,\\
0, & \text{if } -w^\top \zeta_j \ge \tau,
\end{cases}
\]
and set
\[
\gamma_j
=
\big\|\nabla_{\zeta} F(w,\tau,\zeta_j)\big\|_2
=
\begin{cases}
\|w\|_2, & z_j=1,\\[0.3ex]
\left(1+\frac{\rho}{\alpha}\right)\|w\|_2, & z_j=0,
\end{cases}
\]
according to \eqref{eq:app_grad_norm_cases}.
By Step~3, for sufficiently small $\delta>0$, the pair $(w,\tau,z)$ satisfies \eqref{eq:app_logic_1}--\eqref{eq:app_logic_2}.
We now check the remaining constraints in \eqref{eq:app_MISOCP_L}:
\begin{itemize}
    \item  If $z_j=1$ (small-loss):
\[
\|w\|_2 - \gamma_j = 0 \le 0,
\qquad
\|w\|_2\left(1+\frac{\rho}{\alpha}\right) - \gamma_j
=
\|w\|_2\frac{\rho}{\alpha} \le M_3,
\]
so both inequalities hold.

\item If $z_j=0$ (large-loss):
\[
\|w\|_2 - \gamma_j
=
\|w\|_2 - \left(1+\frac{\rho}{\alpha}\right)\|w\|_2
\le M_2,
\]
which is trivially satisfied since $M_2$ is positive and $\gamma_j \ge 0$;
and
\[
\|w\|_2 \left(1 + \tfrac{\rho}{\alpha}\right) - \gamma_{j}
=
0 \le 0.
\]
\end{itemize}

Thus $(w,\tau,\gamma,z)$ is feasible for \eqref{eq:app_MISOCP_L}, and its objective value equals that of \eqref{eq:app_SBR_SAA_L}, since
\[
\sum_{j=1}^m r_j \gamma_j
=
\sum_{j=1}^m r_j \big\|\nabla_{\zeta} F(w,\tau,\zeta_j)\big\|_2.
\]

\medskip

\item \emph{From MISOCP to SBR-SAA-L.}
Conversely, let $(w,\tau,\gamma,z)$ be feasible for \eqref{eq:app_MISOCP_L}.
Using \eqref{eq:app_logic_1}--\eqref{eq:app_logic_2}, we distinguish two cases:
\begin{itemize}
    \item  If $z_j=1$, then $-w^\top \zeta_j \le \tau$, i.e., the small-loss regime.
Constraint $\|w\|_2 - \gamma_j \le 0$ implies $\gamma_j \ge \|w\|_2$.
From \eqref{eq:app_grad_norm_cases}, we have
$\|\nabla_{\zeta} F(w,\tau,\zeta_j)\|_2 = \|w\|_2$ in this regime,
thus $\gamma_j \ge \|\nabla_{\zeta} F(w,\tau,\zeta_j)\|_2$.

\item If $z_j=0$, then $-w^\top \zeta_j \ge \tau - \delta$, approximating the large-loss regime.
Constraint
\(
\|w\|_2 \left(1 + \tfrac{\rho}{\alpha}\right) - \gamma_{j} \le 0
\)
implies
\(
\gamma_j \ge \left(1 + \tfrac{\rho}{\alpha}\right)\|w\|_2.
\)
In the large-loss regime,
$\|\nabla_{\zeta} F(w,\tau,\zeta_j)\|_2 = \left(1 + \frac{\rho}{\alpha}\right)\|w\|_2$, hence again
$\gamma_j \ge \|\nabla_{\zeta} F(w,\tau,\zeta_j)\|_2$.
\end{itemize}

Therefore, for every feasible $(w,\tau,\gamma,z)$,
\[
\sum_{j=1}^m r_j \gamma_j
\;\ge\;
\sum_{j=1}^m r_j \big\|\nabla_{\zeta} F(w,\tau,\zeta_j)\big\|_2.
\]
Since the objective in \eqref{eq:app_MISOCP_L} minimizes $\sum_j r_j\gamma_j$, at optimality we must have
$\gamma_j = \|\nabla_{\zeta} F(w,\tau,\zeta_j)\|_2$
for all $j$ (up to the negligible $\delta$-approximation at the threshold),
and hence the optimal objective value of \eqref{eq:app_MISOCP_L}
coincides with that of \eqref{eq:app_SBR_SAA_L}.
This proves the stated MISOCP reformulation for SBR-SAA-L.

\end{enumerate}

\medskip

\noindent\textbf{Step 5: MISOCP reformulation for SBR-SAA (quadratic aggregation).}

We now consider the original SBR-SAA model with quadratic aggregation:
\begin{equation}\label{eq:app_SBR_SAA_quad}
\min_{w\in\mathcal{W},\,\tau\in\mathbb{R}}
\frac{1}{n}\sum_{i=1}^{n} F(w,\tau,\widehat{\xi}_i)
+
\varepsilon\,
\left(
\sum_{j=1}^{m} r_j
\big\|\nabla_{\zeta} F(w,\tau,\zeta_j)\big\|_2^2
\right)^{1/2}.
\end{equation}
Using \eqref{eq:app_grad_norm_cases}, we have
\[
\left(
\sum_{j=1}^{m} r_j
\big\|\nabla_{\zeta} F(w,\tau,\zeta_j)\big\|_2^2
\right)^{1/2}
=
\left\|
\left(
r_1^{1/2}\big\|\nabla_{\zeta} F(w,\tau,\zeta_1)\big\|_2,
\dots,
r_m^{1/2}\big\|\nabla_{\zeta} F(w,\tau,\zeta_m)\big\|_2
\right)
\right\|_2.
\]

Introduce variables $\gamma_j\in\mathbb{R}$ for $j\in[m]$, a scalar $s\in\mathbb{R}$, and the same binaries $z_j\in\{0,1\}$ as before.
Consider the MISOCP
\begin{equation}\label{eq:app_MISOCP_quad}
\left\{
\begin{array}{cll}
\displaystyle \min_{w,\tau,\gamma,z,s}
&
\displaystyle
\frac{1}{n}\sum_{i=1}^{n} F(w,\tau,\widehat{\xi}_i)
+
\varepsilon\, s
&
\\[0.5ex]
\text{subject to}
&
w\in\mathcal{W},\ \tau\in\mathbb{R},\ \gamma\in\mathbb{R}^{m},\ z\in\{0,1\}^m,\ s\in\mathbb{R},
&
\\[0.5ex]
&
\text{\eqref{eq:app_logic_1}--\eqref{eq:app_logic_2} hold for all } j\in[m],
&
\\[0.5ex]
&
r_j^{1/2}\|w\|_2 - \gamma_{j} \leq M_2 (1-z_{j}),
& j\in[m],
\\
&
r_j^{1/2}\|w\|_2 \left(1 + \tfrac{\rho}{\alpha}\right) - \gamma_{j} \leq M_3 z_j,
& j\in[m],
\\
&
\|\gamma\|_2 \le s. &
\end{array}
\right.
\end{equation}
We claim that \eqref{eq:app_SBR_SAA_quad} and \eqref{eq:app_MISOCP_quad} are equivalent.

\begin{enumerate}
    \item \emph{From SBR-SAA to MISOCP.}
Let $(w,\tau)$ be feasible for \eqref{eq:app_SBR_SAA_quad}.
For each $j\in[m]$, define $z_j$ as in Step~4(a), and set
\[
\gamma_j
=
r_j^{1/2}\big\|\nabla_{\zeta} F(w,\tau,\zeta_j)\big\|_2.
\]
Exactly as before, constraints \eqref{eq:app_logic_1}--\eqref{eq:app_logic_2} are satisfied for sufficiently small $\delta$.
Moreover:

\begin{itemize}
    \item  If $z_j=1$ (small-loss), then $\|\nabla_{\zeta} F\|_2 = \|w\|_2$, hence
\(
\gamma_j = r_j^{1/2}\|w\|_2
\)
and
\(
r_j^{1/2}\|w\|_2 - \gamma_j = 0 \le 0,
\)
while the second inequality for $\gamma_j$ is slack.

\item If $z_j=0$ (large-loss), then
\(
\gamma_j = r_j^{1/2}\left(1+\frac{\rho}{\alpha}\right)\|w\|_2
\)
and the second inequality is tight, while the first is slack.

\end{itemize}

Thus $(w,\tau,\gamma,z)$ satisfies all constraints except the last one, for which we set
\[
s
=
\|\gamma\|_2
=
\left(
\sum_{j=1}^{m} r_j
\big\|\nabla_{\zeta} F(w,\tau,\zeta_j)\big\|_2^2
\right)^{1/2}.
\]
Then $\|\gamma\|_2 \le s$ holds with equality, and the objective of \eqref{eq:app_MISOCP_quad} coincides with that of \eqref{eq:app_SBR_SAA_quad}.

\medskip

\item \emph{From MISOCP to SBR-SAA.}
Conversely, let $(w,\tau,\gamma,z,s)$ be feasible for \eqref{eq:app_MISOCP_quad}.
As in Step~4(b), the logic constraints imply:

\begin{itemize}
    \item  If $z_j=1$ (small-loss region), then $-w^\top \zeta_j \le \tau$, hence $\|\nabla_{\zeta} F(w,\tau,\zeta_j)\|_2 = \|w\|_2$, and
\[
r_j^{1/2}\|w\|_2 - \gamma_j \le 0
\quad\Rightarrow\quad
\gamma_j \ge r_j^{1/2}\|w\|_2 = r_j^{1/2}\|\nabla_{\zeta} F(w,\tau,\zeta_j)\|_2.
\]

\item  If $z_j=0$ (large-loss region), then $-w^\top \zeta_j \gtrsim \tau$, and
\[
r_j^{1/2}\|w\|_2 \left(1 + \tfrac{\rho}{\alpha}\right) - \gamma_j \le 0
\quad\Rightarrow\quad
\gamma_j
\ge
r_j^{1/2}\left(1 + \tfrac{\rho}{\alpha}\right)\|w\|_2
=
r_j^{1/2}\|\nabla_{\zeta} F(w,\tau,\zeta_j)\|_2.
\]

\end{itemize}

Therefore, for all $j$,
\[
\gamma_j
\ge
r_j^{1/2}\big\|\nabla_{\zeta} F(w,\tau,\zeta_j)\big\|_2,
\]
and hence
\[
\|\gamma\|_2
\ge
\left(
\sum_{j=1}^{m} r_j
\big\|\nabla_{\zeta} F(w,\tau,\zeta_j)\big\|_2^2
\right)^{1/2}.
\]
Since $\|\gamma\|_2 \le s$ by constraint, we obtain
\[
s
\ge
\left(
\sum_{j=1}^{m} r_j
\big\|\nabla_{\zeta} F(w,\tau,\zeta_j)\big\|_2^2
\right)^{1/2}.
\]
Thus, the MISOCP objective
\(
\frac{1}{n}\sum_i F(w,\tau,\widehat{\xi}_i) + \varepsilon s
\)
is an upper bound on the SBR-SAA objective \eqref{eq:app_SBR_SAA_quad} for the same $(w,\tau)$.
At optimality, any excess in $\gamma$ or $s$ would increase the objective and can be reduced while maintaining feasibility, so necessarily
\[
s
=
\left(
\sum_{j=1}^{m} r_j
\big\|\nabla_{\zeta} F(w,\tau,\zeta_j)\big\|_2^2
\right)^{1/2},
\quad
\gamma_j
=
r_j^{1/2}\big\|\nabla_{\zeta} F(w,\tau,\zeta_j)\big\|_2,
\]
for all $j$.
Hence, the optimal values of \eqref{eq:app_SBR_SAA_quad} and \eqref{eq:app_MISOCP_quad} coincide.

\end{enumerate}

\medskip

Combining Steps~4 and 5, we conclude that the formulations stated in Proposition~\ref{prop:RefAdversoSAAMeanRisk_sec} are valid MISOCP reformulations of the SBR-SAA-L and SBR-SAA portfolio problems, respectively.
\end{proof}

\end{document}